\documentclass[11pt,a4paper]{article}
\usepackage{amsmath,amssymb,amsthm,graphicx,color,bbm,array}
\usepackage{soul,csquotes}
\usepackage[left=1in,right=1in,top=1in,bottom=1in]{geometry}
\usepackage[british]{babel}
\usepackage{hyperref} 
\usepackage{enumitem}
\usepackage{tikz}
\usepackage{xcolor}
\usetikzlibrary{decorations.markings}
\hypersetup{
    colorlinks=false,
    pdfborder={0 0 0},
}

\makeatletter 
\@mparswitchfalse%
\makeatother
\reversemarginpar
\usepackage[colorinlistoftodos]{todonotes}

\usepackage[style=ext-numeric-comp,backend=bibtex,giveninits=true,isbn=false,url=false,doi=false,articlein=false,maxbibnames=99]{biblatex}
\DeclareFieldFormat[article]{title}{#1} 
\DeclareFieldFormat[phdthesis]{title}{#1} 
 \AtEveryBibitem{%
  \clearfield{number}}

\newtheorem{theorem}{Theorem}[section]
\newtheorem{corollary}[theorem]{Corollary}
\newtheorem{proposition}[theorem]{Proposition}
\newtheorem{lemma}[theorem]{Lemma}

\numberwithin{equation}{section}

\theoremstyle{definition}
\newtheorem{definition}[theorem]{Definition}

\newtheorem{examplex}[theorem]{Example}

\newenvironment{example}
  {\pushQED{\qed}\examplex}
  {\popQED\endexamplex}

\theoremstyle{remark}
\newtheorem{remark}[theorem]{Remark}
\newtheorem{remarks}[theorem]{Remarks}
\newtheorem*{remark*}{Remark}

 \allowdisplaybreaks

\newcommand{\1}[1]{{\mathbbm{1}\mkern -1.5mu}{\{#1\}}}

\newcommand{\2}[1]{{\mathbbm{1}}_{#1}}

\newcommand{\R}{{\mathbb R}}
\newcommand{\eps}{{\varepsilon}}
\newcommand{\Z}{{\mathbb Z}}
\newcommand{\N}{{\mathbb N}}
\newcommand{\ZP}{{\mathbb Z}_+}
\newcommand{\RP}{{\mathbb R}_+}

\DeclareMathOperator{\Exp}{\mathbb{E}}
\let\Pr\relax
\DeclareMathOperator{\Pr}{\mathbb{P}}

\DeclareMathOperator{\IP}{{\mathbb P}}
\DeclareMathOperator{\IE}{{\mathbb E}}

\DeclareMathOperator{\Var}{\mathbb{V}ar}

\DeclareMathOperator*{\argmin}{arg \mkern 1mu \min}
\DeclareMathOperator*{\argmax}{arg \mkern 1mu \max}

\DeclareMathOperator{\card}{\#}

\newcommand{\re}{{\mathrm{e}}}
\newcommand{\rc}{{\mathrm{c}}}
\newcommand{\ud}{{\mathrm d}}

\newcommand{\cC}{{\mathcal C}}
\newcommand{\cD}{{\mathcal D}}

\newcommand{\cF}{{\mathcal F}}
\newcommand{\cG}{{\mathcal G}}

\newcommand{\cL}{{\mathcal L}}

\newcommand{\as}{\ \text{a.s.}}

\newcommand{\Bigmid}{\; \Bigl| \;}

\newcommand{\bbX}{{\mathbb{X}}}

\newcommand{\bbXB}{{\mathbb{X}_\mathrm{F}}}

\renewcommand{\phi}{\varphi}

\newcommand{\eqd}{\overset{d}{=}}

\makeatletter
\def\namedlabel#1#2{\begingroup  
    (#2)%
    \def\@currentlabel{#2}%
    \phantomsection\label{#1}\endgroup
}
\makeatother

\newlist{myenumi}{enumerate}{10}
\setlist[myenumi]{leftmargin=0pt, labelindent=\parindent, listparindent=\parindent, labelwidth=0pt, itemindent=!, itemsep=1pt, parsep=4pt}

\newlist{thmenumi}{enumerate}{10}
\setlist[thmenumi]{leftmargin=0pt, labelindent=\parindent, listparindent=\parindent, labelwidth=0pt, itemindent=!}

\newcommand{\hv}{{\hat v}}
\newcommand{\tv}{{\tilde v}}

\title{Dynamics for  exclusion processes with random environments\\ of particle jump rates}

 \author{Mikhail Menshikov\footnote{\raggedright Department of Mathematical Sciences, Durham University, Durham, UK. \href{mailto:mikhail.menshikov@durham.ac.uk}{\texttt{mikhail.menshikov@durham.ac.uk}}, \href{mailto:andrew.wade@durham.ac.uk}{\texttt{andrew.wade@durham.ac.uk}}.} \and Serguei Popov\footnote{Centro de Matem\'atica, University of Porto, Porto, Portugal. \href{mailto:serguei.popov@fc.up.pt}{\texttt{serguei.popov@fc.up.pt}}.} \and Andrew Wade\footnotemark[1]}

\date{\today}

\begin{document}
\maketitle 

\begin{abstract}
We consider a continuous-time particle system 
with exclusion interaction on the integer lattice.
Each particle is assigned intrinsic left and right jump rates by an independent draw from a common random environment.
We identify the decomposition of the system into maximal
stable subsystems, which we call clouds.
We show different qualitative behaviour for the cloud
decomposition corresponding to different regimes for the
random environment law that parallel 
the classical Sinai and Kesten--Kozlov--Spitzer regimes from one-dimensional
random walk in random environment.
Our main result in the Sinai regime is that the system 
is decomposed into an infinite number of finite 
clouds which all go to minus infinity with rapidly
decreasing speed magnitudes, and we obtain explicit associated distributional limits. 
We analyse the model through the potential function associated with the random environment. In the Sinai regime, deep potential wells trap particles, producing stable clouds whose speeds are exponentially small in the well depth. Since successive wells are    deeper, the resulting clouds are successively slower. Thus the localization mechanism of Sinai's random walk manifests itself here as a jamming phenomenon, in which a potential well traps not one walker but an entire block of mutually excluding particles.  
\end{abstract}

\medskip


\medskip


\bigskip
\bigskip

\textsl{This paper is dedicated to the memory of Mikhail Menshikov, who passed away on May~1, 2026. We (S.P.\ and A.W.) wish to acknowledge that Mikhail was the driving force behind the research programme to which this work belongs (see \cite{MMPW23,MPW25,MPW25_elastic,MPW26_leftmost}), and that the original idea for this line of investigation was his. We are deeply grateful for his vision, insight, and guidance throughout the project. Beyond his scientific contributions, we will profoundly miss him as a dear friend and a lifelong mentor.}

\section{Introduction and results}
\label{s_intro_res}

\subsection{Formulation of the model}
\label{sec:model}

We study a semi-infinite exclusion process in which each particle has its own randomly chosen left and right jump rates. The interaction causes particles with different intrinsic drifts to organize into maximal stable groups, or clouds, whose internal separations remain tight while distinct clouds separate. Our aim is to describe the typical cloud structure generated by a random environment of particle rates.

Our main results describe the cloud partitions for natural classes of random environments, which parallel  the phase diagram for  environments in \emph{random walk in random environment} (RWRE) on~$\Z$~\cite{solomon,Zeitouni2004,revesz,KKS}.
For example, in the analogue of Sinai's regime~\cite{sinai82}, we show that, for almost all environments, the system decomposes into an infinite collection of finite subsystems whose size and intrinsic speed can be described asymptotically via a \emph{potential} function, indexed by particle labels, associated to the random environment. We show that the localization phenomenon for random walk in random environment has an analogous \emph{jamming} phenomenon in the particle system context.

The system is described as follows. Particles are enumerated  by $\N := \{1,2,3,\ldots\}$ from the left to the right (particle $1$ is the leftmost particle, and there is no rightmost particle). 
The $k$th particle has intrinsic jump rates $a_k, b_k \in (0,\infty)$, given by the realization of the random environment, 
where $a_k$ is the rate at which the particle attempts to jump one site to the left, and $b_k$ the rate at which it attempts to jump one site to the right. 
The exclusion interaction means that jumps that target a site occupied by another particle are suppressed; consequently, the order of the particles is preserved. 
See Figure~\ref{fig:particles} for a picture.

\begin{figure}[h]
\centering
\scalebox{0.85}{
 \begin{tikzpicture}[domain=0:1, scale=1.0]
\draw[dotted,<->] (0,0) -- (13,0);
\node at (13.4,0)       {$\Z$};
\draw[black,fill=white] (1,0) circle (.5ex);
\draw[black,fill=black] (2,0) circle (.5ex);
\node at (2,-0.6)       {\small $X_1 (t)$};
\draw[black,fill=white] (3,0) circle (.5ex);
\draw[black,fill=black] (4,0) circle (.5ex);
\node at (4,-0.6)       {\small $X_2 (t)$};
\draw[black,fill=black] (5,0) circle (.5ex);
\node at (5,-0.6)       {\small $X_3 (t)$};
\draw[black,fill=white] (6,0) circle (.5ex);
\draw[black,fill=white] (7,0) circle (.5ex);
\draw[black,fill=white] (8,0) circle (.5ex);
\draw[black,fill=black] (9,0) circle (.5ex);
\node at (9,-0.6)       {\small $X_4 (t)$};
\draw[black,fill=black] (10,0) circle (.5ex);
\node at (10,-0.6)       {\small $X_5 (t)$};
\draw[black,fill=black] (11,0) circle (.5ex);
\node at (11,-0.6)       {\small $~~~~~~ X_6 (t) ~ \cdots$};
\draw[black,fill=white] (12,0) circle (.5ex);
\node at (1.3,0.6)       {\small rate $a_1$};
\node at (2.5,0.6)       {\small $b_1$};
\draw[black,->,>=stealth] (2,0) arc (30:141:0.58);
\draw[black,->,>=stealth] (4,0) arc (30:141:0.58);
\node at (3.5,0.6)       {\small $a_2$};
\node at (5.5,0.6)       {\small $b_3$};
\draw[black,->,>=stealth] (9,0) arc (30:141:0.58);
\node at (8.5,0.6)       {\small $a_4$};
\node at (11.5,0.6)       {\small $b_6$};
\draw[black,->,>=stealth] (2,0) arc (150:39:0.58);
\draw[black,->,>=stealth] (5,0) arc (150:39:0.58);
\draw[black,->,>=stealth] (11,0) arc (150:39:0.58);
\end{tikzpicture}}
\caption{Schematic of the model showing the leftmost $6$ particles, illustrating some of the main notation. Filled circles represent particles, empty circles represent unoccupied lattice sites, and
directed arcs represent admissible transitions, with exponential rates indicated.}
\label{fig:particles}
\end{figure}
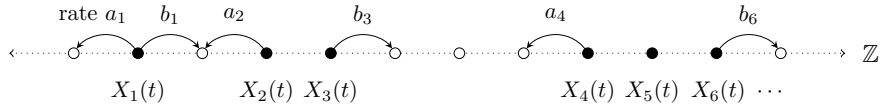

The configuration space of the system is $\bbX$, given by
\begin{equation}
\label{eq:configuration-space}
\bbX  := \big\{ (x_1, x_2, \ldots ) \in \Z^{\N} : x_1 < x_2 < \cdots  \big\} . 
\end{equation}
Let $X (t) = (X_1 (t), X_2(t), \ldots  ) \in \bbX$ be the configuration of the  process of particle positions 
at time $t \in \RP := [0,\infty)$,
started from a fixed initial configuration $X(0) \in \bbX$.

For deterministic jump rates, cloud decompositions were studied for finite systems in~\cite{MMPW23}, while~\cite{MPW25,MPW26_leftmost} investigated stability and dynamics of semi-infinite systems forming a single cloud. The goals and methods of the present paper are different. Here we study the cloud structure generated by an i.i.d.~environment of particle rates, combining potential methods from one-dimensional RWRE with coupling arguments and the finite-system theory of~\cite{MMPW23}.

We use the notation
$\omega:=\big((a_k,b_k), \, k \in \N\big)$ 
 for the random environment,
 which we assume lives on a probability space with probability measure $\IP$ and expectation $\IE$. We suppose that, under $\IP$, $(a_i, b_i)_{i\geq 1}$ are   independently 
sampled as random copies of a vector 
$(\Theta^-,\Theta^+)\in (0,\infty)^2$.
The particle system 
in random environment is then, for fixed $\omega$, defined by taking $(a_k,b_k)$
to be Markov transition rates of the $k$th particle, as described above.
When the initial configuration
is~$x\in\bbX$, we use the notation~$\IP^x_\omega$ for 
the \emph{quenched} (i.e., with fixed
environment~$\omega$)
law of $ (X(t), \, t \geq 0)$ started from $X(0)=x$, 
and we denote by $\IE^x_\omega$ the corresponding expectation. 

The central object in our analysis is the random \emph{potential} obtained by summing the log-ratios of the particle jump rates, which we define
formally in Section~\ref{sec:potentials} below. Roughly speaking, stable clouds correspond to \emph{traps} or \emph{wells} of this potential: a sufficiently deep well produces a group of particles which moves together, with a speed exponentially small in the depth of the well. Thus a major part of the narrative of this paper is explaining how the large-scale cloud structure is governed by the successive record wells of a one-dimensional random walk (in the limit, Brownian motion). 

We need to assume some regularity conditions on the random environment. The first ``nestling'' condition (the terminology comes from the RWRE literature; see~\cite{zerner,Zeitouni2004}) ensures that the particles may have drifts in both
directions; in the absence of such an assumption we would not get the possibility of deep potential wells.

\begin{description}
\item
[\namedlabel{ass:nestling}{{\bf N}}]  
\emph{Nestling
random environment.} 
It holds that $\IP[\Theta^-<\Theta^+] >0$ 
and $\IP[\Theta^->\Theta^+] >0$.
\end{description}

We assume a \emph{uniform ellipticity}
condition, i.e., that the rates are uniformly
bounded away from zero and infinity (this is standard in RWRE, for example).

\begin{description}
\item
[\namedlabel{ass:ue}{{\bf UE}}] 
\emph{Uniform ellipticity.} 
For some $\eps_0 \in (0,1)$ we have 
$\IP[\eps_0\leq\Theta^\pm\leq \eps_0^{-1}]=1$.
\end{description}

Observe that Condition~\eqref{ass:ue} implies that
\begin{equation}
\label{unif_ellipticity}
 \text{for $\kappa:=2 \, | \log \eps_0 | < \infty$ 
 we have $\IP$-a.s.\ that }
 -\kappa \leq \log ( \Theta^- / \Theta^+ ) \leq \kappa.
\end{equation}

For technical reasons we also need a
  mild regularity assumption on ascending ladder heights
  of the random walk
with independent increments distributed 
as $\log ( \Theta^- / \Theta^+ )$.
If we denote such random walk by~$V=(V_n, n\geq 0)$ and 
assume~$V_0=0$, then we can
denote $\tau_+:=\min\{k: V_k>0\}$
and let~$\eta:=V_{\tau_+}>0$ be the first
ascending ladder height. The next condition 
assures that the first ladder height is unlikely 
to be exceptionally small.

\begin{description}
\item
[\namedlabel{ass:lh}{{\bf LH}}] 
\emph{Ladder height.} 
Suppose
there exist positive constants $\gamma$ and~$\delta_0$ 
such that for all small enough $y>0$ it holds that 
$\IP[\eta \geq y]\geq 1 - \gamma y^{\delta_0}$.
\end{description}

Condition~\eqref{ass:lh}  holds, in particular, for lattice environments and under natural bounded-density hypotheses; see Lemma~\ref{lem:lh}. It need not follow from~\eqref{ass:ue} and~\eqref{ass:nestling} alone.
Here is an example for which all of the above conditions are satisfied.

\begin{example}
\label{ex:uniform}
    Let $\eps \in (0,1)$, and take $\Theta^- , \Theta^+$ to be independent $\text{U}[\eps,\eps^{-1}]$ random variables. Then hypotheses~\eqref{ass:nestling}, \eqref{ass:ue}, and~\eqref{ass:lh} 
    are satisfied (for the latter, apply Lemma~\ref{lem:lh}). 
\end{example}

The sign of $\Exp \log (\Theta^-/\Theta^+)$,
which is finite under hypothesis~\eqref{ass:ue}, 
informs a phase diagram resembling that for one-dimensional RWRE,
and our results cover all three regimes.
We give our results for the case $\Exp \log (\Theta^-/\Theta^+) = 0$ 
in \S\ref{s_results_Sinai} and 
for the cases $\Exp \log (\Theta^-/\Theta^+) \neq 0$
in \S\ref{sec:KKS-results}. 
In the zero-mean case, we show that the system consists of infinitely many finite clouds with doubly exponentially decreasing speed magnitudes. When the mean is negative, all but finitely many particles form a single semi-infinite stable cloud. When the mean is positive, there are again infinitely many finite clouds, but magnitudes of particle speeds decrease polynomially.
Before formally stating our results, we
formally define the concept of \emph{stable clouds} that we have so far discussed informally.

\subsection{Stable clouds}
\label{sec:stable-clouds}

Informally, a stable cloud is a maximal consecutive group of particles in which pairwise particle separations remain tight while its distance from particles outside the cloud diverges. 
Note that, typically, the decomposition into stable 
clouds will depend on the environment~$\omega$, and, in principle,
can depend also on the initial configuration, as explained in Section~\ref{sec:clouds-infinite} below.
We use notation $[i,j] := \{ \ell \in \ZP : i \leq \ell \leq j\}$ ($i \leq j$) for a discrete interval, and write $[n,\infty) := \{ \ell \in \ZP : \ell \geq n\}$ (throughout $\ZP := \N \cup \{0\} = \{0,1,2,\ldots\}$). 

The next definition says that a finite group of consecutive particles is a stable cloud if the distance between the outermost
particles in the group stays tight, and the distances of each outermost particle with its neighbour outside the group tends to $\infty$ in probability.

\begin{definition}[Stable cloud, finite case]
\label{def:stable-cloud-finite}
For $i \leq j$ the group $[i,j]$ 
is an $(\omega,x)$-stable cloud,
if (i) for every $\eps>0$ there exists $B \in \RP$
such that
\[ \Pr_\omega^x ( X_j (t) - X_i (t) \geq B ) \leq \eps , \text{ for all } t \in \RP ,\]
and (ii) for every $B \in \RP$ and every $\eps>0$ there exists $t_0 \in \RP$ such that
$\Pr_\omega^x ( X_{j+1}(t) - X_j(t) \geq B ) \geq 1-\eps$ for all $t \geq t_0$,
and (provided $i >1$)  
$\Pr_\omega^x ( X_{i}(t) - X_{i-1}(t) \geq B ) \geq 1-\eps$ for all $t \geq t_0$.
\end{definition}

    The semi-infinite case is similar.

\begin{definition}[Stable cloud, semi-infinite case]
\label{def:stable-cloud-infinite}
The group $[i,\infty)$ 
is an $(\omega,x)$-stable cloud,
if (i) for every $j, k \in [i,\infty)$, for every $\eps>0$ there exists $B \in \RP$
such that
\[ \Pr_\omega^x ( | X_j (t) - X_k (t)| \geq B ) \leq \eps , \text{ for all } t \in \RP ,\]
and (ii) (provided $i >1$)  
for every $B \in \RP$ and every $\eps>0$ there exists $t_0 \in \RP$ such that
$\Pr_\omega^x ( X_{i}(t) - X_{i-1}(t) \geq B ) \geq 1-\eps$ for all $t \geq t_0$.
\end{definition}

\subsection{Results: Sinai regime}
\label{s_results_Sinai}

Suppose that
\begin{equation}
\label{Sinai's_regime}
 \IE \log\frac{\Theta^-}{\Theta^+} = 0.
\end{equation}
By analogy with the RWRE case~\cite{sinai82}, we say this is \emph{Sinai's regime} for the random environment. Note that~\eqref{Sinai's_regime} holds whenever $\Theta^-$ and $\Theta^+$ have the same distribution, as they do in Example~\ref{ex:uniform} above.
Let $\sigma^2 := \Var \log ( \Theta^-/\Theta^+)$; note $\sigma^2 \in (0,\infty)$ under~\eqref{ass:nestling} and~\eqref{ass:ue}.

Our main result for Sinai's regime is the following.
Informally, in the Sinai regime the particles organize themselves into successively larger finite clouds, which, up to a finite initial section, we can identify explicitly with potential wells, and whose speeds decrease doubly exponentially with the cloud index. The ratios of successive logarithmic speeds have an asymptotic Pareto law, while cloud sizes and   squared logarithmic speeds satisfy a distributional limit result.
We write ``a.a.'' for ``almost all''.

\begin{theorem}
\label{thm:RE_Sinai}
 Assume Conditions~\eqref{ass:nestling}, \eqref{ass:ue}, and~\eqref{ass:lh}, and suppose 
that~\eqref{Sinai's_regime} holds. Let $X(0) = x \in \bbX$.  Then 
for $\IP$-almost all environments the system
is decomposed into infinitely many finite stable
clouds,
i.e., there exist pairwise disjoint, non-empty $\cC_1$, $\cC_2$, \ldots such that (i) for every $k \in \N$ and all $i \in \cC_k$ and $j \in \cC_{k+1}$, $i <j$, and (ii) $\cC_1 \cup \cC_2 \cup \cdots = \N$,
such that each $\cC_k$ is stable in the sense of Definition~\ref{def:stable-cloud-finite}.
Moreover, the cloud decomposition is independent of the initial configuration, and, for a.a.~$\omega$
there exist $k_1$ and $k_2$ such that sequences $( \cC_k, k \geq k_1)$ and $([h_k+1,h_{k+1}], k \geq k_2)$ coincide, where $h_k$ are defined in terms of the potential in \S\ref{s_infinite_Sinai} below.

To every cloud $k \in \N$ there is an associated \emph{speed}
$s_k=s_k(\omega)$ 
for which the following hold.
\begin{thmenumi}[label=(\roman*)]
\item\label{thm:RE_Sinai-i}
For every $k \in \N$ and every $j \in \cC_k$,
for   $\IP$-a.a.~$\omega$ 
we have the strong law
\begin{equation}
    \label{eq:slln}
    \lim_{t \to \infty} t^{-1} X_j (t) = s_k, ~\text{$\IP^x_\omega$-a.s.}
\end{equation}
 \item\label{thm:RE_Sinai-ii}
 For $\IP$-a.a.\ $\omega$, 
it holds that, as $k\to \infty$, 
\begin{equation}
\label{eq_speeds_asymp}
 s_k = -\exp\big(-\re^{k(1+o(1))}\big).
\end{equation}
\item    \label{thm:RE_Sinai-iii}
We have, for all $y \geq 1$,
\begin{equation}
\label{eq_relative_speeds}
\lim_{k \to \infty} \IP [ \xi_k
 >y  ] = \frac{1}{y}, \text{ where } \xi_k := \frac{\log|s_{k+1}|}{\log|s_k|}.
\end{equation}
Also, the random variables $\xi_k$
are asymptotically independent in the sense that, for each $m \in \ZP$,  
$(\xi_{n}, \xi_{n+1}, \ldots, \xi_{n+m} ) $ converges in distribution, as $n \to \infty$, to
 $( \theta_0, \theta_1, \ldots, \theta_m )$
 a vector of i.i.d.~variables with $\Pr [ \theta_0 > y ]=1/y$ for $y \geq 1$.
 \item\label{thm:RE_Sinai-iv}
 Let $S_k=S_k(\omega)$ denote the size of cloud $k \in \N$ (i.e., the number of particles it contains). 
 There is 
 a non-negative
random variable~$Y$,
the first hitting time of~$2$ by a 3-dimensional Bessel process started at $0$,
with
\begin{equation}
    \label{eq:Y-distribution}
    \Pr ( Y \leq y ) = 1 - 2 \sum_{m=1}^\infty (-1)^{m+1} \exp \left\{ - \frac{\pi^2 m^2 y}{8} \right\} ,
\end{equation} 
such that (where  $\Rightarrow$  denotes convergence in distribution)
\begin{equation}
\label{eq_sizes_depths}
  \frac{S_k}{\log^2|s_k|^{-1}}
\Rightarrow \frac{Y}{\sigma^2} , \text{ as }k\to\infty.
\end{equation}
\end{thmenumi}
\end{theorem}

While Theorem~\ref{thm:RE_Sinai} describes the speeds and sizes of stable clouds, the next result
describes the speed of a given particle.
The two results are related, since determining the particle's speed essentially amounts to determining which cloud it belongs to.

\begin{theorem}
\label{t_Sinai_particle}
  Assume 
  Conditions~\eqref{ass:nestling}, \eqref{ass:ue}, and~\eqref{ass:lh}, suppose 
that~\eqref{Sinai's_regime} holds. Let $X(0) = x \in \bbX$. 
To every particle $n \in \N$ there is an associated {speed}
$v_n =v_n(\omega) < 0$ 
such that, for   $\IP$-a.a.~$\omega$,
\begin{equation}
    \label{eq:slln-by-particle}
    \lim_{t \to \infty} t^{-1} X_n (t) = v_n, ~\text{$\IP^x_\omega$-a.s.}
\end{equation}
Moreover, there is a non-negative
random variable~$Z$,
given in terms of standard Brownian motion at~\eqref{df_Y_Sinai} below, 
with
\begin{equation}
    \label{eq:Z-distribution}
    \Pr ( Z \leq z ) = \frac{8}{\pi^2} \sum_{m=0}^\infty (2m+1)^{-2}
    \exp \left\{ - \frac{\pi^2 (2m+1)^2}{8z^2} \right\} ,
\end{equation} 
such that the speed of the $n$th particle satisfies
\begin{equation}
\label{eq_Sinai_particle}
   n^{-1/2} \log | v_n|^{-1} \Rightarrow \sigma Z, 
\text{ as }n\to\infty.
\end{equation}
\end{theorem}
\begin{remark}
\label{rem:distributions}
    The explicit distributions for the limit random variables $Y$ and $Z$ from~\eqref{eq:Y-distribution} and~\eqref{eq:Z-distribution} 
    are reminiscent of certain limit distributions in the theory of RWRE, such as identified
    by Golosov~\cite{golosov} and Kesten~\cite{kesten}.
    These distributions likewise arise from potential-well-type functionals of Brownian motion. We have, however, been unable to find either of the distributions of $Y$ or $Z$ explicitly in the RWRE literature.
\end{remark}

\subsection{Results: Kesten--Kozlov--Spitzer regime}
\label{sec:KKS-results}

The case when
$\IE \log ( \Theta^- / \Theta^+ ) \neq 0$ is 
sometimes referred to in the RWRE literature
as the \emph{Kesten--Kozlov--Spitzer} or KKS regime, after~\cite{KKS}.
For some of the results in this section, we restrict our initial configuration to the \emph{finite configurations} (meaning configurations with finitely many holes), $\bbXB \subset \bbX$, given by
\begin{equation}
    \label{eq:finite}
    \bbXB := \{ x \in \bbX : x_{i+1} - x_i =1 \text{ for all but finitely many } i \}.
\end{equation}
First, we will show that the case
\begin{equation}
\label{eq:KKS<0}
 \IE \log\frac{\Theta^-}{\Theta^+} < 0
\end{equation}
is relatively easy to treat:  some finite ($\omega$-dependent)
number of leftmost particles will escape
to~$-\infty$ (at least in probability), 
and the rest forms a stable system.
Indeed: 
\begin{theorem}
\label{thm:KKS<0} 
Assume
Condition~\eqref{ass:ue} and~\eqref{eq:KKS<0}. Let $X(0) = x \in \bbXB$. 
Then there exists 
an a.s.\ finite random variable 
$N=N(\omega)$, whose 
 precise definition is given in~\eqref{df_N_KKS} below, 
such that, for some $c_1, c_2>0$,  $\IP[N>k]\leq c_1 \re^{-c_2 k}$ for all $k \in \ZP$,
and,  $\Pr$-a.s.,
$\{N,N+1,\ldots\}$ is a semi-infinite $(\omega,x)$-stable cloud,  as 
defined in Definition~\ref{def:stable-cloud-infinite}.
\end{theorem}

Note that Theorem~\ref{thm:KKS<0} does not need Condition~\eqref{ass:nestling} or~\eqref{ass:lh}.
We expect that it should be possible to extend the range of initial conditions for which Theorem~\ref{thm:KKS<0} holds: see Remark~\ref{rem:KKS-IC} below. 
A more interesting case is 
\begin{equation}
\label{eq:KKS>0}
 \IE \log\frac{\Theta^-}{\Theta^+} > 0.
\end{equation}
In this case, we show that there are infinitely many stable clouds, and  all particles
will go to~$-\infty$ at speeds whose magnitudes decrease polynomially with the particle label. 

\begin{theorem}
\label{thm:KKS>0}
 Assume 
 Conditions~\eqref{ass:nestling}, \eqref{ass:ue}, 
and~\eqref{eq:KKS>0}. 
Let $\beta>0$ be the unique positive number
such that 
\[
 \IE\Bigl[ \bigl( {\Theta^+} / {\Theta^-}\bigr)^{1/\beta} \Bigr] =1.
\]
Let $X(0) = x \in \bbX$. 
Then for almost all environments the system 
is decomposed into an infinite number of finite 
stable clouds.
To every particle $n \in \N$ there is an associated {speed}
$v_n =v_n(\omega)$ 
such that, for   $\IP$-a.a.~$\omega$ 
we have the strong law
\begin{equation}
    \label{eq:slln-kks}
    \lim_{t \to \infty} t^{-1} X_n (t) = v_n, ~\text{$\IP^x_\omega$-a.s.},
\end{equation}
and we have, for $\IP$-a.a.~$\omega$, 
$v_n =-n^{-\beta(1+o(1))}$ as $n\to\infty$.
\end{theorem}
\begin{remarks}
\label{rem:beta-existence}
\begin{myenumi}[label=(\roman*)]
\item\label{rem:beta-existence-i}
    Existence and uniqueness of $\beta>0$ in the hypothesis of Theorem~\ref{thm:KKS>0} 
    follows from the fact that,
    using~\eqref{ass:nestling} and~\eqref{eq:KKS>0}, the function $s \mapsto \log \Exp [ ( \Theta^+ / \Theta^-)^s ]$ is strictly convex, has value $0$ and negative derivative at $s=0$, and tends to $+\infty$ as $s \to \infty$, so crosses $0$ exactly once for $s>0$.
    \item \label{rem:beta-existence-ii}
In the KKS regime it is harder to ``separate'' clouds. In Sinai's regime, the successive record-well depths grow exponentially with the well index, yielding strong separation between the associated speeds. In the KKS regime, the deepest relevant wells up to particle $n$ have depth only of order $\log n$, and the resulting speeds are separated only on polynomial scales. For this reason we do not provide an analogue of the detailed cloud identification in Theorem~\ref{thm:RE_Sinai}.
    \end{myenumi}
\end{remarks}

\subsection{Some comments on the literature}
\label{sec:literature}

The process of inter-particle separations
associated to the exclusion process is a particular case
of the \emph{zero-range} process. 
Since the 1990s, there has been extensive 
study of zero-range and exclusion processes with disordered 
environments, including heterogeneous but deterministic,  or randomly chosen,
rates;
in the exclusion process, rates attached to particles (as in our case)
correspond to rates attached to \emph{sites} in the zero-range process.
However, to the best of the authors' knowledge, the present setting is not considered in any previous work. 
We refer to~\cite{FS07,AFGL00,GL98,BMRS15,BMRS17,BMRS21} and references therein for
extensive results on   phenomena concerning convergence to equilibrium and hydrodynamic limits. Our emphasis is different: rather than equilibrium or hydrodynamic behaviour, we study the almost-sure large-scale decomposition of a fixed random environment into dynamically coherent particle clouds, and relate this decomposition to the record-well structure of its random potential.

The papers closest to the present setting include
work of 
Benjamini~\emph{et al.}~\cite{BFL96},
Krug and Ferrari~\cite{KF96}, and Evans~\cite{evans96,evans},  
and a more recent programme of Bahadoran~\emph{et al.}~\cite{BMRS15,BMRS17,BMRS21}. In several of these papers the (non-random or random) environment is  chosen so that there is a global drift of particles in one direction (for example, in~\cite{KF96,evans} the process is totally asymmetric with a random environment of one-directional rates attached to particles and taking values in a bounded interval), and various aspects of the dynamics are dominated by the slowest particles.

Another set of literature looks at exclusion processes in which the random environment is attached to sites, rather than particles: we refer to, e.g.~\cite{faggionato,LY24,frs,schmid19,jsi}. Because of the connection between exclusion processes and Jackson queueing networks
(see Section~\ref{sec:jackson} below) there is also to be found some adjacent literature on queueing networks in random environment: we refer to, e.g.~\cite{KDO16,GPSY16,zhu94}.

\subsection{Outline of the paper and proof strategy}
\label{sec:outline}

Section~\ref{sec:prep} gives some preparatory results, partially
known from prior work, including the representation of the particle
system as a Jackson network (\S\ref{sec:jackson}), some results on finite systems (\S\ref{sec:finite-systems}), and some results on stable clouds of infinite systems (\S\ref{sec:clouds-infinite}). 
Section~\ref{sec:potentials}
introduces the potential function that is our main tool in the proofs in both Sinai and KKS regimes.
Section~\ref{sec:sinai-proofs} gives the proofs in the Sinai regime,
constructing potential wells and analysing their depths (\S\ref{s_infinite_Sinai}),
identifying the candidate clouds and controlling their speeds (\S\ref{sec:candidate-clouds}),
and then presenting in turn the proofs of Theorems~\ref{thm:RE_Sinai} and~\ref{t_Sinai_particle} (\S\ref{sec:sinai-theorem-proofs}), up to the exact identifications of the distributions
of $Y$ and $Z$, which are given in \S\ref{sec:explicit_limits}. The proofs for the KKS regimes 
are given in Section~\ref{sec:proofs_KKS}. The Appendix
contains Lemma~\ref{lem:lh} on the condition~\eqref{ass:lh}.

Before getting into technical details, let us give some heuristics about the role of the potential function that lies at the heart of this paper. 
First, it turns out (due to explicit formulas seen e.g.~in~\cite{MMPW23}) that   a putative cloud of particles is stable if its shape corresponds to a  well in the potential, and, moreover, the depth of that well determines, roughly, the speed of that cloud.
The main difficulty is to pass from potential wells to clouds in the interacting infinite system. A block that is stable in isolation need not remain a separate cloud when other particles are present.
 In Sinai's regime, the potential function is approximately a trajectory of Brownian motion, and, roughly, scale invariance shows that the statistics of successive wells are related one to the next by a random scale factor. Understanding this picture gives rise to the main results on Sinai's regime. 
In the KKS regime the same potential picture persists, but the potential has non-zero drift, leading to a different hierarchy of wells and hence of cloud speeds.

\section{Preparatory results}
\label{sec:prep}

\subsection{Jackson network representation}
\label{sec:jackson}

In this section we recall from~\cite{MPW25} a representation of the particle system in terms of a 
Jackson network of infinitely many queues; 
the link between exclusion processes and Jackson networks is well known,
going back at least to~\cite{kipnis} (see~\cite{MMPW23,MPW25} for further literature). 

First, in some generality, consider a countably infinite system of queues, enumerated by~$\N$. The parameters of the system are arrival rates $\lambda = (\lambda_i)_{i \in \N}$,
service rates $\mu = (\mu_i)_{i \in \N}$ and $P = (P_{ij})_{i,j\in\N}$, a sub-stochastic routing matrix.
Exogenous customers entering the system arrive at queue $i \in \N$ via an independent Poisson process of rate $\lambda_i \in \RP$.
Queue $i \in \N$ serves customers
at exponential rate $\mu_i \in \RP$
(i.e., in Kendall's notation, we have 
an $M/M/1$ queue at each server).
 Once a customer at queue $i$ is served, it is routed to a queue~$j$
with probability $P_{ij}$, while with probability $\delta_i := 1 - \sum_{j \in \N} P_{ij}$ the customer departs from the system.

Provided $\sum_{i \in \N} \lambda_i >0$ and $\sum_{i \in \N} \delta_i >0$, customers both enter and leave the system, and
it is called an \emph{open Jackson network}. We assume that every queue can be \emph{filled}, 
meaning that, for every $i \in \N$, there is a $j \in \N$ and $k \in \ZP$ for which $\lambda_j > 0$ and $(P^k)_{ji}>0$,
and that every queue can be \emph{drained}, meaning that,  for every $i \in \N$, there is a $j \in \N$ and $k \in \ZP$ for which $\delta_j > 0$ and $(P^k)_{ij}>0$.

We now explain the Jackson network connection to the particle system
model as described in Section~\ref{sec:model}; we use the notation introduced there.
Define
$\eta_k (t) := X_{k+1} (t) - X_k (t) -1$,
to be
the number of holes
between consecutive particles at time~$t\in\RP$. Define
\begin{equation}
\label{eq:mu-def}
\mu_i := b_i + a_{i+1}, \text{ for }
i \in \N, \end{equation}
and
\begin{equation}
\label{eq:lambda-def}
\lambda_1 := a_1,   \text{ and } \lambda_i := 0 \text{ for } i \geq 2.
\end{equation} 
Also define the  matrix $P:= (P_{i,j})_{i,j \in \N}$ by
\begin{equation}
\label{eq:P-def}
\begin{split}
P_{i,i-1} & := \frac{b_i}{\mu_i} = \frac{b_i}{b_i+a_{i+1}}, \text{ for } i \geq 2; \\
P_{i,i+1} & := \frac{a_{i+1}}{\mu_i} = \frac{a_{i+1}}{b_i+a_{i+1}}, \text{ for } i \geq 1;
\end{split}
\end{equation}
and $P_{i,j}:=0$ for all $i,j$ with $\vert i-j\vert \neq 1$.

The process $\eta := (\eta_{i})_{i \in \N}$ can, under appropriate
conditions on the rates and initial states, 
be realised as the queue-length process for a corresponding Jackson network, namely,
the Jackson network with parameters $\lambda, \mu$ and $P$ given as functions of $(a_i,b_i)_{i \in \N}$ through formulas~\eqref{eq:mu-def}, \eqref{eq:lambda-def} and~\eqref{eq:P-def}. 

In the queueing network, the customers waiting at queue~$k$
are equal in number to the unoccupied sites between particles $k$ and $k+1$ in the particle system. 
Exogenous customers enter the queueing network only when the leftmost particle jumps to the
left (rate $a_1$). Customers at queue~$k$ are served if either particle~$k$ jumps right
(rate $b_k$) or if particle $k+1$ jumps left (rate $a_{k+1}$). If  
particle~$k$ jumps right, the customer 
is routed to queue $k-1$ (if $k \geq 2$, at rate $b_k = \mu_k P_{k,k-1}$) or leaves the system (if $k=1$).
If particle $k+1$ jumps left, the customer is routed to queue $k+1$ (at rate $a_{k+1} = \mu_k P_{k, k+1}$). 
Customers leave the system only when the leftmost particle jumps to the right (rate $\mu_1 \delta_1 = b_1$).
Then a particle conservation argument leads to a construction of $X_1$, the position of the left-most particle,  via
 \begin{equation}
     \label{eq:conservation}
     X_1 (t) - X_1 (0) = - \sum_{k \in \N} \bigl( \eta_k (t) - \eta_k (0) \bigr), \text{ whenever } \sum_{k \in \N} \eta_k (0) < \infty.
 \end{equation}
Under condition \eqref{ass:ue},
existence and construction for both the queueing network and the associated particle system is given for fixed 
environment in~\cite{MPW25}.

\subsection{Some results on finite systems}
\label{sec:finite-systems}

Later we will use repeatedly the concept of \emph{intrinsic stability} of a  group $[i,j]$ 
of particles in the semi-infinite system, which means that if all particles outside the group were removed, then the whole system would be a single stable cloud (in the sense of Definition~\ref{def:stable-cloud-finite}); we give a definition in terms of formulas at Definition~\ref{def:intrinsic-stability} below.
In our later argument we often also use coupling and comparison results between the infinite system and various finite systems.
To study these finite systems we rely on some relevant results from~\cite{MMPW23}, which we summarize in this section. 

The finite system of~$n$ particles (hence $n-1$ queues)
has parameters $a_1, b_1, \ldots, a_n, b_n$,
having the same meaning as in Section~\ref{sec:model},
so that particle~$k \in [1,n]$ 
performs a nearest neighbour random walk, rate~$a_k$ to the left and $b_k$ to the right,
subject to the exclusion rule.
The Jackson network correspondence is similar to that 
described in Section~\ref{sec:jackson} 
(see~\cite[\S 3]{MMPW23}), but now
customers may also enter or depart the system from queue~$n-1$ (at rates $b_n$ and $a_n$, respectively). We denote by $X_{n,k} (t)$ the location of particle $k$ at time $t$.

Let us recall the following notations from~\cite{MPW25}:
\begin{align}
\label{eq:alpha-k-def}
 \alpha_0 & := 1, ~\text{and}~ \alpha_k := \frac{a_1\cdots a_k}{b_1\cdots b_k} \text{ for } k \in \N; \\
\label{eq:beta-k-def}
\beta_0 & := 0 ,  ~\text{and}~ \beta_k := \frac{1}{b_k}+\frac{a_k}{b_k b_{k-1}}+
  \cdots + \frac{a_k\cdots a_2}{b_k\cdots b_1}
   \text{ for }k\in \N.
\end{align}
Also define
   \begin{equation}
\label{eq:v_explicit_variant} 
 v_{n} := \frac{1-\alpha_n}{\beta_n} = 
 \frac{b_1 \cdots b_n- a_1\cdots a_n}
 {b_1\cdots b_{n-1}
 +b_1 \cdots b_{n-2} a_n +
  \cdots + a_2 \cdots a_n}.
\end{equation}
and set $\rho_k := \alpha_k + v_n \beta_k$ for $k \in [1,n-1]$.

The following result gives the intrinsic speed of a finite stable system.

\begin{proposition}[Finite stable systems; Corollary~2.6 in~\cite{MMPW23}]
\label{prop:finite-cloud}
    Let $n \in \N$ and suppose that $a_k , b_k \in (0,\infty)$ for $1 \leq k \leq n$. 
Then if $\rho_{k} < 1$ for all $k \in [1,n-1]$ the $n$-particle system is a single stable cloud and, for every $k \in [1,n]$, we have 
\[ 
\lim_{t \to \infty} \frac{X_{n,k} (t)}{t} = v_{n}, \as 
\]
\end{proposition}

We will also consider modifications of the dynamics of a group $[i,j]$
in the semi-infinite system 
which we call  \emph{infinite pressure} from the right, which corresponds to  setting $b_j:=0$, that is, the rightmost 
particle of the group is forbidden to jump to the right
(i.e., one may imagine 
that there is \emph{always} another particle just next 
to it, hence ``infinite pressure'').
We discuss some applications of that later; for now,
let us note that it is intuitively clear that this
operation preserves the intrinsic stability 
of a group. 
Putting the infinite pressure 
 from the right on the $n$-particle system described 
above amounts to setting $b_{n}:=0$.

 \begin{remark}
\label{rem:malyshev-infinite-pressure}
    The results of~\cite{MMPW23} permit $a_i =0$ and so, by swapping the roles of the $a_i$ and $b_i$ apply to the finite systems with infinite pressure considered here.
\end{remark}

We need the following result about infinite-pressure systems.

\begin{proposition}[Finite stable systems 
with infinite pressure]
\label{prop:finite-cloud_inf_pressure}
Assume that the $n$-particle system is stable, 
and then change the value of~$b_n$ to zero.
Then the resulting system is stable 
and its speed is given by
\begin{equation}
\label{eq:tilde_v_explicit_variant} 
 {\tilde v}_{n} := 
 - \frac{\alpha_n}{\beta_n} =
 -\frac{a_1\cdots a_n}
 {b_1\cdots b_{n-1}
 +b_1 \cdots b_{n-2} a_n +
  \cdots + a_2 \cdots a_n}.
\end{equation}
(Observe that, even though $\alpha_n$ and $\beta_n$ depend on $b_n$, the expression in~\eqref{eq:tilde_v_explicit_variant}  does not.)
\end{proposition}
 \begin{proof}
Since the original system is stable, $\rho_{k} = \alpha_k +v_n \beta_k$ satisfies $\rho_{k} <1$ for all $k \in [1,n-1]$. Setting $b_n =0$ does not change the $\alpha_k, \beta_k$ for $k \in [1,n-1]$, but replaces $v_n$ by $\tilde v_n \leq v_n$,
obtained from removing the first term in the numerator from the rightmost expression in~\eqref{eq:v_explicit_variant} (Corollary~2.6 in~\cite{MMPW23} applies, see Remark~\ref{rem:malyshev-infinite-pressure}). 
Hence the $\rho_k$ corresponding to the modified system also satisfy $\rho_{k} <1$ for all $k \in [1,n-1]$. So the modified system is stable and its speed is $\tilde v_n$.
\end{proof}

We finish this section with some stochastic monotonicity results.

\begin{proposition}
\label{prop:stoch-monotonicity}
    \begin{thmenumi}[label=(\roman*)]
    Consider a finite system of particles labelled by $[1,n]$.
\item    
\label{prop:stoch-monotonicity-i}
For any $k \in [1,n-1]$, there is a coupling of the system $[1,n]$
and the system $[1,k]$ with infinite pressure from the right at~$k$, such that particles $[1,k]$ start in the same relative positions and  every particle in the smaller system is no farther to the right than 
its counterpart in the original system.
\item    
\label{prop:stoch-monotonicity-ii}
For any $k \in [1,n-1]$, there is a coupling of the system $[1,n]$
and the system $[k,n]$, such that particles $[k,n]$ start in the same relative positions and  every particle in the smaller system is no farther to the right than 
its counterpart in the original system.
\end{thmenumi}
\end{proposition}
\begin{proof}
    We prove part~\ref{prop:stoch-monotonicity-i}. The coupling
    construction uses the \emph{basic coupling} (or Harris graphical coupling, see~\cite[Ch.~II]{liggett}) for the exclusion dynamics.
     We describe the coupling between system $[1,n]$ (with particle positions $X_1 (t), \ldots, X_n (t)$) and system $[1,k]$ with infinite pressure at $k$, whose particle positions we denote $Y_1(t), \ldots, Y_{k} (t)$. We start the coupling with $X_i(0) = Y_i(0)$ for $1 \leq i \leq k$ and use the same Harris graphical construction to execute jumps. We claim the ordering $Y_i (t) \leq X_i(t)$ for all $i \in [1,k]$ and all $t \geq 0$ is preserved. 
To verify the ordering is preserved, suppose it holds immediately before  a clock rings. For the left-jump of particle~$i$, we must show that $X_i$ cannot jump left when $Y_i$ is blocked at a time when $X_i = Y_i$ (we suppress the time index for ease of reading). Indeed, if $Y_i$ is blocked then $Y_{i-1} = Y_i -1$ and since $Y_{i-1} \leq X_{i-1} < X_i = Y_i$ we must have $X_{i-1} = X_i - 1$, so the jump of $X_i$ is also blocked. A similar argument shows that for $i < k$, we cannot have that $Y_i$ jumps right when $X_i  $ is blocked, from $X_i = Y_i$. It follows that the ordering is preserved at every graphical event, demonstrating the existence of the coupling claimed in part~\ref{prop:stoch-monotonicity-i}. A similar coupling applies for the removal of the particles to the left and proves part~\ref{prop:stoch-monotonicity-ii}.
\end{proof}

\subsection{On stable clouds of infinite systems}
\label{sec:clouds-infinite}

Here we briefly discuss a general question about 
how to identify \emph{finite} stable clouds
of a semi-infinite system. 
This question in its full generality
was intentionally left out of our paper~\cite{MPW25}.
To indicate the difficulty, note that the notion of a  
 stable cloud in the semi-infinite case
 \emph{may} depend on the initial configuration,
as the following example shows.
Consider first a stable semi-infinite system that
has more than one admissible solution, with the 
corresponding speeds of these admissible solutions corresponding (say) to the interval~$[0,2]$
(to construct such an example see Example 1.17 of~\cite{MPW25}). To the left of
this semi-infinite system, add one particle with intrinsic speed equal to~$1$.
Then, when starting from a finite initial configuration,
the whole system will be stable (that is, we would have
only one stable cloud), with zero speed. However, if we start 
at a sufficiently sparse configuration 
where the semi-infinite cloud is able to achieve speed~$2$,
then it will separate from the unit speed leftmost particle,
meaning that there will be two stable clouds in the system.
 
In the context of the present paper, the phenomenon described above does not occur.
Instead, we will be typically in the situation where there are infinitely many stable clouds with strictly ordered speeds. In this case we are able to make use of  couplings between the infinite system and appropriate finite systems as described in the following lemma.
 
\begin{lemma}
\label{lem:couple_infinite_pressure}
Suppose that, for some~$k_0>0$, after putting
an infinite pressure (from the right)
to $k_0$th particle (i.e., setting $b_{k_0}:=0$,
thus ``detaching'' the first $k_0$ particles 
from the rest of the system), the resulting
finite system can be decomposed
into strictly
more than one stable clouds,
$\cC_1,\ldots,\cC_m,\cC_{m+1}$, where 
$k_0\in\cC_{m+1}$. 
Let $s_1 \leq s_2 \leq \cdots\leq s_{m+1} <0$
be the speeds of these clouds; assume now that 
there exists~$k\leq m$ such that $s_k<s_{k+1}$
(i.e., there is a transient queue separating~$\cC_k$
 from~$\cC_{k+1}$).

 In such situation, there exists a coupling
of the original queue process~$\eta$ 
with the infinite pressure queue 
process~$\tilde\eta$, such that there is a
random time~$T$ with the property that for all 
$j=1,\ldots,k$ it holds that
$\eta(t)|_{\cC_j}=\tilde\eta(t)|_{\cC_j}$
for all $t\geq T$.
\end{lemma}
\begin{proof}
The stated coupling can be described in the following way:
the customers in~$\tilde \eta$ are \emph{first-class 
customers}, and
 additional customers that 
may come to the original system in comparison with 
the infinite pressure one are \emph{second-class 
customers}\footnote{That is, the second-class 
customers are those that come to the rightmost queue from the outside.}
which have lesser priority in service
(i.e., they can only be served when there are no first-class 
customers in the queue). 
For the queue process~$\tilde \eta$,
let $T_0$ be a random time (which is not Markovian)
after which the transient queues which separate 
the stable clouds never become empty.
Then the crucial observation is that the second-class customers never leave such queues after~$T_0$, and therefore they are not 
transferred to $\cC_1,\ldots,\cC_k$.
Then, those second-class customers
which were present in $\eta|_{\cC_j}$, 
$j=1,\ldots,k$ at~$T_0$ will eventually either leave
the system or get stuck in a transient queue forever.

Indeed, consider any fixed second-class customer
which is to the left of the transient queue separating~$\cC_k$
 and~$\cC_{k+1}$. 
 For each stable cloud $\cC_j$, the first-class Jackson network on its finitely many internal queues is irreducible positive recurrent, so it returns to the all-zero state infinitely often a.s. At each such return there is a strictly positive probability that a given second-class customer present in that cloud is served through a prescribed finite sequence of transitions and exits the cloud before the next first-class arrival. 
 Since there are only finitely many second-class customers to the left of the transient queue at time $T_0$, and no new second-class customers can enter this region after $T_0$, the preceding probabilities may be bounded below uniformly by a positive constant. 
 Consider now any separating queue lying to the left of the transient
queue between $\cC_k$ and $\cC_{k+1}$.  Pathwise, either this queue is empty
infinitely often, or there is a finite random time after which it is
never empty.  In the latter case, any second-class customer which reaches
this queue is trapped there forever.  In the former case, at each time
at which the first-class queue is empty, a second-class customer present
there has a probability bounded below by a strictly positive constant of
being served before the next first-class arrival.  Hence, on the event
that the first-class queue is empty infinitely often, every second-class
customer present there is eventually served from the queue almost surely.

Thus a tagged second-class customer cannot remain forever in any stable
cloud, and at each separating queue it is either eventually passed on or
becomes trapped there forever.  Since there are only finitely many stable
clouds and separating queues to the left of the transient queue between
$\cC_k$ and $\cC_{k+1}$, every second-class customer present to the left of
this separator at time $T_0$ is eventually either removed from the
system to the left or trapped forever in a separating queue.  Since
there are only finitely many such second-class customers, there is a
finite random time $T$ after which no second-class customer is present
in any of the internal queues of $\cC_1,\ldots,\cC_k$. 
\end{proof}

\begin{corollary}
\label{cor:inf_finite_queues}
Assume that there are two sequences of positive 
integers $k_1<k_2<k_3<\cdots$ and 
$k'_1<k'_2<k'_3<\cdots$ also with the property
that $k'_j<k_j$ for all~$j$, such that, when 
putting infinite pressure at~$k_j$, the $k'_j$th 
queue is transient. Then the system is decomposed
into infinitely many finite clouds with non-decreasing speeds, and there is an infinite 
number of transient queues.
\end{corollary}
\begin{proof}
Couple two infinite-pressure truncations at $k<\hat k$ by regarding the additional customers generated by the particles in $(k,\hat k]$ as second-class customers. Starting from the same queues on $[1,k-1]$, this gives $\eta_i^{(k)}(t)\le \eta_i^{(\hat k)}(t)$ (superscript indicating the truncation), for all $i<k$ and all $t\ge0$.
Hence any queue which is transient in the truncation at $k$ remains transient after moving the truncation point to the right.

Fix $r$, and choose $J$ sufficiently large. Then the truncation at $k_J$ has the transient queues $k'_1,\ldots,k'_r$, and hence at least $r+1$ stable clouds separated by these queues. Apply  Lemma~\ref{lem:couple_infinite_pressure} using the transient separator $k'_r$. The internal queue processes of all clouds lying to its left eventually coincide with those in the infinite system. Moreover, under the same coupling the separating queues in the infinite system dominate their finite-system counterparts, so the separating gaps still diverge in probability. Thus these are stable clouds of the infinite system, with the same, non-decreasing speeds as in the finite truncation. Since $k'_r\to\infty$, every particle belongs to such a finite cloud. Finally, the domination also shows that every $k'_r$ remains transient in the infinite system. Hence there are infinitely many finite clouds and infinitely many transient queues.
\end{proof}

\section{Finding stable clouds with the potential}
\label{sec:potentials}

In this section we define the potential and 
argue that it is a convenient tool for identifying
stable clouds and their speeds.
The \emph{potential function} (or simply \emph{potential})
$V:
\ZP \to \R$ is defined by 
\begin{equation}
\label{df_potential}
V(0):=0, \text{ and } V(n) := \sum_{j=1}^n \log\frac{a_j}{b_j} \text{ for } n\in\N.
\end{equation}
The potential is used extensively 
in the context of 
one-dimensional RWRE: see 
e.g.\ the definition of~$w^{(n)}(t)$ on p.~257 of~\cite{sinai82}, 
or the first display in Section~2 of~\cite{FGP10}.
In the context of RWRE the idea is that one frequently 
represents various quantities associated to the behaviour of the random walk  in terms 
of the potential, so that an analysis of the behaviour of the potential
leads to conclusions about the behaviour of the random walk (in a particular
environment). In our context, something similar is true, as we shall explore in this section.

We point out that the choice of the initial condition $V(0)=0$
is rather arbitrary (although, of course, natural).
In the following, when discussing intrinsic stability
of some finite groups of particles,
we will frequently consider finite pieces of the potential
that do not necessarily start at zero; therefore, we 
typically present formulas in terms of the potential \emph{differences}
so that they can be applied whether or not the leftmost particle we are considering at the time is the leftmost particle overall.

We will frequently need to consider finite groups
of consecutive particles, and answer questions
like if that group would be intrinsically stable
(meaning it would be stable if  all the other particles were removed from 
the system). 
First, let us note that the formulas~\eqref{eq:alpha-k-def}--\eqref{eq:beta-k-def} can be rewritten 
in the following way in terms of the potential:
\begin{align}
 \alpha_k & := \re^{V(k)-V(0)} \text{ for }k\in\ZP,
 \text{ and } \beta_k := \sum_{\ell=1}^k a_\ell^{-1} \re^{V(k)-V(\ell-1)} \text{ for } k \in \N.
\label{alpha_to_V}
\end{align}
Then, consider a group $[i_0+1,\ldots,i_0+n]$
together with the potential~$V$ on 
the interval~$[i_0,\ldots,i_0+n]$
and observe that
the formulas~(2.16) and~(2.17) of~\cite{MMPW23}
can be rewritten as
\begin{align}
\label{V_rho_explicit}
 \rho_k &:= \re^{V(k)-V(i_0)}
  + \hv \sum_{\ell=i_0+1}^{k} a_\ell^{-1} \re^{V(k)-V(\ell-1)} , \text{ for } k \in [i_0+1,i_0+n] , \text{ where}\\
 \hv & :=\frac{1-\re^{V(i_0+n)-V(i_0)}}
{\sum_{\ell=i_0+1}^{i_0+n} a_\ell^{-1} \re^{V(i_0+n) - V(\ell-1)}}.
 \label{V_v_explicit}
\end{align}

\begin{definition}[Intrinsic stability]
    \label{def:intrinsic-stability}
    We say the group $[i_0+1,\ldots,i_0+n]$ of $n$ particles is \emph{intrinsically stable} if $\rho_k < 1$ for all $k \in [i_0+1,\ldots,i_0+n-1]$, where $\rho_k$ is given by~\eqref{V_rho_explicit}. If the group is intrinsically stable, we say its \emph{intrinsic speed} is $\hv$ given by~\eqref{V_v_explicit}.
\end{definition}
\begin{remarks}
\label{rems:intrinsic-stability}
    \begin{myenumi}[label=(\roman*)]
\item\label{rems:intrinsic-stability-i}
If the group $[i_0+1,\ldots,i_0+n]$ is intrinsically stable,
then when considered as an isolated $n$-particle system it forms a single
stable cloud (in the sense of Definition~\ref{def:stable-cloud-finite}) and has speed $\hv$ given by~\eqref{V_v_explicit} (cf.~Proposition~\ref{prop:finite-cloud}).
\item \label{rems:intrinsic-stability-ii}
If the group $[i_0+1,\ldots,i_0+n]$ is intrinsically stable,
then so is the same group with infinite pressure from the right, i.e.,
when we set $b_{i_0+n} = 0$. This follows in the same way as for isolated systems as in Proposition~\ref{prop:finite-cloud_inf_pressure}. 
\item \label{rems:intrinsic-stability-iii}
If the group $[i_0+1,\ldots,i_0+n]$ is intrinsically stable,
then, due to an elementary coupling argument, it is contained in a single stable cloud
in the system $[1,m]$ (with or without infinite pressure at $m$), 
where $m\geq i_0+n$.
\end{myenumi}
\end{remarks}

An immediate observation is that, in the case when
the group $[i_0+1,\ldots,i_0+n]$ is intrinsically stable,   
the speed~$\hv$ is positive
(respectively, negative, zero) if $V(i_0+n)<V(i_0)$
(respectively, $V(i_0+n)>V(i_0)$, 
$V(i_0+n)=V(i_0)$). So, we will frequently study 
finite pieces of the potential in order to draw 
conclusions about the intrinsic stability and speed 
of the corresponding group of particles.
Also, 
in this section we will typically assume~$i_0=0$,
since this obviously does not restrict generality.

Having motivated the relevance of the potential,
the main goal of this section is to understand  how the potential~$V$
of stable (and unstable) clouds should behave.
As an initial easy observation, note that if~$V$ looks 
like in Figure~\ref{f_stable0}, then we have a stable cloud.
\begin{figure}
\begin{center}
\includegraphics{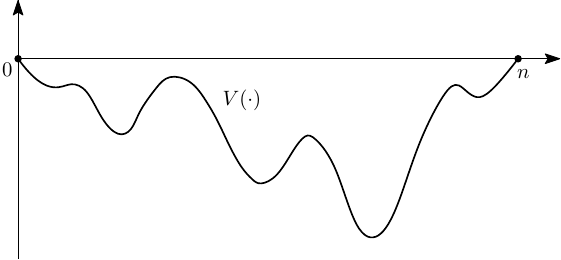}
\caption{An example of a potential profile 
of an intrinsically stable, in the sense of Definition~\ref{def:intrinsic-stability}, system $[1,n]$ (see Lemma~\ref{l_same_hights_stable}): $V(0)=V(n)=0$, and 
$V(k)<0$ for $k \in [1,n-1]$. In this and subsequent illustrations of the potential, we often take the licence to represent the potential by a smooth curve.}
\label{f_stable0}
\end{center}
\end{figure}
More precisely, the following fact holds:
\begin{lemma}
\label{l_same_hights_stable}
Assume that $V(0)=V(n)=0$; then, the group~$[1,n]$
is intrinsically stable (and has zero speed) if and only if
$V(k)<0$ for all $k \in [1,n-1]$.
\end{lemma}
\begin{proof}
Indeed, \eqref{V_v_explicit} implies that~$\hv=0$ so, 
by Definition~\ref{def:intrinsic-stability} and~\eqref{V_rho_explicit}, we 
have that the cloud is stable iff~$\rho_k=\re^{V(k)}<1$ for all $k \in [1,n-1]$,
which means $V(k)<0$ for all $k \in [1,n-1]$.
\end{proof} 

The above is a very basic and also somewhat 
``rare'' example 
(it is not easy to be so lucky to have~$V(n)$ exactly equal to~$0$);
the next situation is more general:
\begin{lemma}
\label{l_stable_almost0}
Assume that $V(k) \leq 0 \leq V(n)$ and $V(k) < V(n)$ for all $k\in [1, n-1]$. Then $[1, n]$
is intrinsically stable.
\end{lemma}
\begin{proof}
Suppose that $[1,n]$ is \emph{not} intrinsically stable;
consider its decomposition into stable clouds.
Let~$v$ be the speed of the leftmost stable cloud and~$v^*$
be the speed of the rightmost one. Then, \eqref{V_v_explicit} 
would imply that $v\geq 0$ and $v^*<0$; this is a contradiction,
since cloud speeds are non-decreasing from left to right (see~\cite[Theorem 2.1]{MMPW23}).
\end{proof}

It is important to note that we can
use Lemma~\ref{l_stable_almost0}
to identify groups
that should be parts of stable clouds;
see Figure~\ref{f_stable1}. Also note that 
this procedure finds clouds 
that have non-positive intrinsic speeds
(one can also
formulate an analogous result for finding 
clouds with non-negative intrinsic speeds, but the present form
of the lemma is better tailored for the needs of this 
paper).
\begin{figure}
\begin{center}
\includegraphics{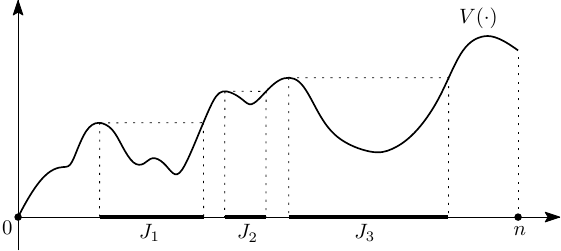}
\caption{Finding stable clouds: for each $j=1,2,3$,
the particles from~$J_j$ form an intrinsically stable group. We do \emph{not} claim that, in the infinite system, e.g., $J_1$ and~$J_2$
should be in \emph{different} stable clouds in the sense of Definition~\ref{def:stable-cloud-finite}.}
\label{f_stable1}
\end{center}
\end{figure}

The following result goes in the converse direction: it
 shows that for stable clouds
  the global maximum for the potential must be attained 
  at (one or both of) the extremities.
\begin{lemma}
\label{l_global_maxima}
 Assume that group $[1,n]$ is intrinsically stable, and let
\[
 M = \bigl\{k\in \{0,\ldots,n\}:
 V(k) = \max_{0 \leq j \leq n}V(j)\bigr\}
\]
be the set of the global maxima of the potential
on the interval $[0,n]$.
Then $M\subseteq \{0,n\}$.
\end{lemma}
\begin{proof}
 Note that we can assume without restricting 
generality\footnote{This 
is because one can ``mirror'' the system
by setting the transition rates to
$(b_n,a_n),\ldots,(b_1,a_1)$; then the potential
will also be ``mirrored''.} that $V(0)\geq V(n)$.
Then $\hv\geq 0$ by~\eqref{V_v_explicit}, so~\eqref{V_rho_explicit}
gives $\rho_k\geq \re^{V(k)-V(0)}$, and the stability condition
$\rho_k<1$ implies $V(k)<V(0)$ for all $k\in[1,n-1]$. 
Since also $V(n)\leq V(0)$, we get $M\subseteq\{0,n\}$.
\end{proof}

As discussed in Section~\ref{sec:finite-systems},
it is a simple observation that, if we have a stable 
finite group~$[1,n]$, then setting~$b_n$ to zero 
preserves the stability (and makes the speed smaller);
recall that we refer to this operation as 
putting \emph{infinite pressure} to the right extreme
of the group.
As before,
we will typically use the notation~$\tv$
for denoting the speed with infinite pressure from 
the right.
This is also useful for the 
original semi-infinite system: 
if we have a good estimate on~$\tv$, then we know that,
whatever the particles
on the right do, they cannot increase the
magnitude of the speed
of our cloud beyond~$|\tv|$.
Let us rewrite the formula for the speed
under infinite pressure (see Proposition~\ref{prop:finite-cloud_inf_pressure})
in terms of the potential, with the help
of~\eqref{alpha_to_V}:
\begin{equation}
\label{speed_infinite_pressure}
 \tv = - \frac{\re^{V(n)-V(0)}}
{\frac{1}{a_{n}}\re^{V(n)-V(n-1)}
  +\cdots + \frac{1}{a_1}\re^{V(n)-V(0)}}
\end{equation}
(note that $\re^{V(n)}$ will cancel, so, in fact, the above 
expression does not depend on~$b_n$).

Now, we define one of the central concepts of this paper.
For $[n_1,n_2]\subset \Z_+$, let the \emph{depth}
of the potential on this subinterval be defined by
\begin{equation}
\label{df_depth}
 \cD(n_1,n_2)
  = \max_{n_1\leq k\leq n_2}
   \max_{\substack{n_1\leq \ell_1\leq k\\
        k\leq \ell_2\leq n_2}}
        \big((V(\ell_1)-V(k))\wedge(V(\ell_2)-V(k))\big).
\end{equation}
In words, the depth is the ``maximal height to overcome''
when going out of the interval, maximized over all starting points: 
see Figure~\ref{f_depth}.
\begin{figure}
\begin{center}
\includegraphics{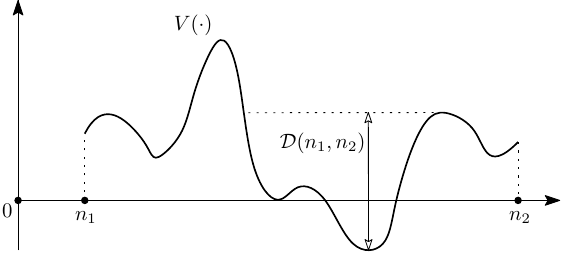}
\caption{Illustration of the definition of the depth at~\eqref{df_depth}.}
\label{f_depth}
\end{center}
\end{figure} 
Also note that~\eqref{df_depth} takes  the convenient form 
\begin{equation}
    \label{eq:D-when-max} 
    \cD(n_1,n_2)=\max_{n_1\leq k \leq n_2}
 \Bigl(\max_{n_1\leq \ell \leq k}V(\ell)-V(k) \Bigr)
 \text{ whenever } V(n_2)=\max_{n_1\leq j\leq n_2}V(j),
\end{equation}
i.e., 
the depth is equal to the maximal descent from 
the running maximum when the maximum potential is attained at the end of the interval. 

A very important observation (which will be made 
rigorous in Lemma~\ref{l_depth_speed} below), is that regions of 
significant depth, or \emph{potential wells} (such as the one below the horizontal dotted line 
in Figure~\ref{f_depth}) tend to slow down the 
corresponding particles due to what we call \emph{jamming}: the particles
in the left part of the well 
tend to want  
to go to the right, while the particles
in the right part of the well tend to want  
to go to the left, which results in a ``traffic jam''.
For this reason, we also refer to such regions 
as \emph{traps}; the role of traps is analogous to their role in localization for RWRE.

The next result says that the speed of a stable cloud is
roughly inversely proportional to the exponential 
of the depth of the associated potential well.

\begin{lemma}
\label{l_depth_speed}
Assume that $\eps_0\leq a_k \leq \eps_0^{-1}$ and 
$\eps_0\leq b_k \leq \eps_0^{-1}$ for some $\eps_0>0$,
and all $k \in [1,n]$ (i.e., the same assumption 
as in~\eqref{ass:ue}). 
Suppose that group $[1,n]$ is intrinsically stable, and 
assume also that
 $V(n)=\max_{0 \leq j \leq n} V(j)$,
which implies that the intrinsic speed satisfies $\hv \leq 0$.
Denote by~$\tv<0$ the speed of group $[1,n]$ 
with the infinite pressure from the right\footnote{Recall from Remark~\ref{rems:intrinsic-stability}\ref{rems:intrinsic-stability-ii}
that the system with infinite pressure is also stable.}
(note that $|\tv| > |\hv|$ by assumption that all $b_k>0$).
Then with $\cD$ defined at~\eqref{df_depth} we have
\begin{equation}
\label{depth_speed_upper}
|\tv|\leq \eps_0^{-1} \re^{-\cD(0,n)},
\end{equation}
and
\begin{equation}
\label{depth_speed_lower}
|\hv|\geq \eps_0n^{-1} (1-\re^{-(V(n)-V(0))})\re^{-\cD(0,n)}.
\end{equation}
\end{lemma}

\begin{figure}
\begin{center}
\includegraphics{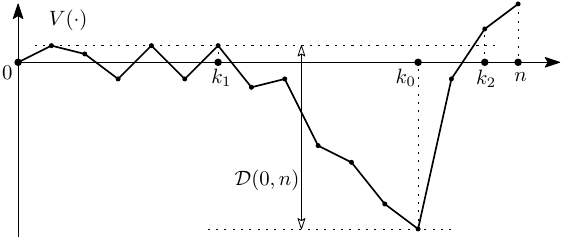}
\caption{Illustration for the proof of Lemma~\ref{l_depth_speed}; 
this time, for more precision,  the illustration retains the discrete nature of the potential.}
\label{f_depth_and_ks}
\end{center}
\end{figure} 
\begin{proof}
We refer to Figure~\ref{f_depth_and_ks} for a visualization of several of the quantities defined in this proof. 
First observe from~\eqref{eq:D-when-max}
and the fact that $V(n) = \max_{0\leq j \leq n} V(j)$ that
\begin{align}
\label{eq:D-0-n}
    \cD (0,n ) = \max_{0 \leq k \leq n} \Bigl( \max_{0 \leq \ell \leq k} V(\ell) - V(k) \Bigr) .
\end{align}
To prove~\eqref{depth_speed_lower}, let~$k'\in[0,n]$ be the smallest~$k$ such that $V(k ) = \min_{0\leq j \leq n} V(j)$. 
In the formula~\eqref{V_v_explicit} for $\hv$, since $V(n) \geq V(0)$, 
the absolute value of the numerator equals
$\re^{V(n)-V(0)} (1-\re^{-(V(n)-V(0))})$, while
the denominator is bounded above 
by $\eps_0^{-1}n \re^{V(n)-V(k')}$. That is,
\[ | \hv | \geq \eps_0 n^{-1}  (1-\re^{-(V(n)-V(0))}) \re^{ - ( V(0) - V (k'))} , \]
where, by~\eqref{eq:D-0-n}, $\cD (0,n ) \geq \max_{0 \leq k \leq n} ( V(0) - V(k) ) = V(0) - V(k')$,
which gives~\eqref{depth_speed_lower}.

Next, we prove~\eqref{depth_speed_upper}, so we consider the system with infinite pressure from the right.
Let~$k_0$ be the largest~$k \in [0,n]$ for which we have
\[
\max_{0 \leq \ell \leq k}V(\ell)-V(k)
=\cD (0,n),
\]
i.e., the largest site where the maximal descent from
the running maximum is achieved.  
Define~$k_1$ to be the largest~$k\in[0,k_0]$
such that $V(k)=\max_{0 \leq j \leq k_0} V(j)$.
Note that $V(k_1)\leq V(n)$, by hypothesis.
Observe that, by construction and~\eqref{eq:D-0-n}, $V(k_1) - V(k_0) = \cD (0,n)$.

Next note that $k_0 =n$ if and only if $\cD (0,n) = 0$, which by~\eqref{eq:D-0-n}, occurs if and only if $V(0) \leq V(1) \leq \cdots \leq V(n)$. In this case~\eqref{depth_speed_upper} follows directly
from the fact that the denominator in~\eqref{speed_infinite_pressure} is bounded below by $\eps_0 \re^{V(n) - V(0)}$. Suppose instead that $k_0 < n$. Then define
\[
 k_2 := \min\big\{k>k_0: V(k)\geq V(k_1)\big\},
\]
and observe that the group $[ k_1+1,k_2]$
is intrinsically stable, by Lemma~\ref{l_stable_almost0}.
Then, observe that the magnitude of the
speed of this 
cloud with infinite pressure from the right
will be at least~$|\tv|$, because (i) the infinite pressure at $k_2$ is at least the ``pressure'' from the presence of any particles in $[k_2+1,n]$ and infinite pressure at $n$, and (ii) any particles in $[0,k_1]$ 
serve to slow down the cloud's movement to the left.
More formally,
start the system $[1,n]$ with infinite pressure at $n$, and the subsystem $[k_1+1,k_2]$ from corresponding positions, with particles $[1,k_1]$ removed and infinite pressure imposed at $k_2$. Then 
applying both parts of  
Proposition~\ref{prop:stoch-monotonicity}
shows that there is a coupling under which each particle of the subsystem remains no farther to the right than its counterpart in the original system.
Hence an upper bound for $|\tv|$ is the magnitude of the speed
of the system $[ k_1+1,k_2]$ with infinite pressure from the right. For the latter, we use the 
appropriate version of the
formula~\eqref{speed_infinite_pressure}, the denominator of which is bounded below by $\eps_0 \re^{V(k_2) - V(k_0)}$ (recall that $k_0 < k_2$). Hence we get
\[ | \tv | \leq \eps_0^{-1} \re^{V(k_2)-V(k_1)-(V(k_2) - V(k_0))} =  \eps_0^{-1} \re^{-( V(k_1)-V(k_0))} ,\]
which yields~\eqref{depth_speed_upper}.
\end{proof} 

To give a flavour of what we are going to 
do in the next section,
let us discuss the situation shown in Figure~\ref{f_stable_right}: 
we want to argue that the whole cloud has to be intrinsically stable.
The idea is that, first, by Lemma~\ref{l_stable_almost0},
the particles in $[1,n_1]$ should belong to the same
stable cloud, and, by~\eqref{depth_speed_upper}, 
the absolute value of the speed
of that cloud has to be quite small, around $\re^{-\cD}$
at most. Now, let us think of the stable 
cloud that contains the rightmost particle ($n$).
If we assume that this cloud does not
contain those particles with indices in $[1,n_1]$,
then, by~\eqref{depth_speed_lower}, 
the absolute value of its speed would
be much larger than $\re^{-\cD}$ due to the fact that
$\cD(n_1,n)<\cD$
(note also that the speed has to be 
strictly negative since the 
maximum of the potential is at~$n$).
The associated strong law 
for finite systems (see Theorem 2.1 of~\cite{MMPW23})
then gives a contradiction, as particles to the right would want to go to the left faster than particles to the left.
It follows that there can be only one stable cloud,
travelling to the left with   speed of roughly $\re^{-\cD}$.
\begin{figure}
\begin{center}
\includegraphics{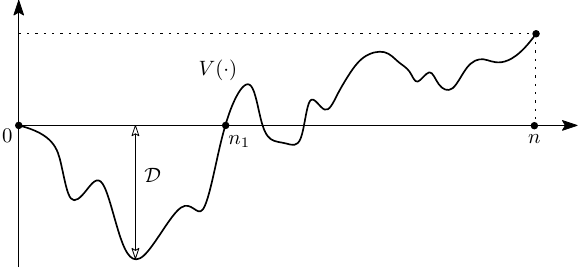}
\caption{Another situation with one stable cloud.}
\label{f_stable_right} 
\end{center}
\end{figure}

\section{Proofs in Sinai's regime}
\label{sec:sinai-proofs}

\subsection{The construction and growth of record wells}
\label{s_infinite_Sinai}

In Sinai's regime, conditions~\eqref{ass:ue} and~\eqref{Sinai's_regime} ensure that the potential
defined by~\eqref{df_potential} is a trajectory of a random walk whose increments are bounded and have mean zero. As already discussed in Section~\ref{sec:potentials}, the behaviour of the potential
can be linked to the problem of determining which groups of the particles in the particle system form stable clouds, and what the intrinsic speeds of these clouds are. In this section we give an explicit scheme for determining the stable clouds from the potential, and use the fact that the potential is well-approximated by a Brownian motion trajectory to draw conclusions about statistics of those stable clouds.

We describe a construction that lets us identify \emph{potential wells} through successive local maxima and minima of the potential, and to define associated well depths: see Figure~\ref{f_wells} for a diagram. First define the descent depth from the current maximum as
\begin{equation}
\label{df_left_climb}
 \cL(x) := \max_{0 \leq y \leq x} V(y) - V(x).
\end{equation}

\begin{figure}[!h]
\begin{center}
\includegraphics[width=\textwidth]{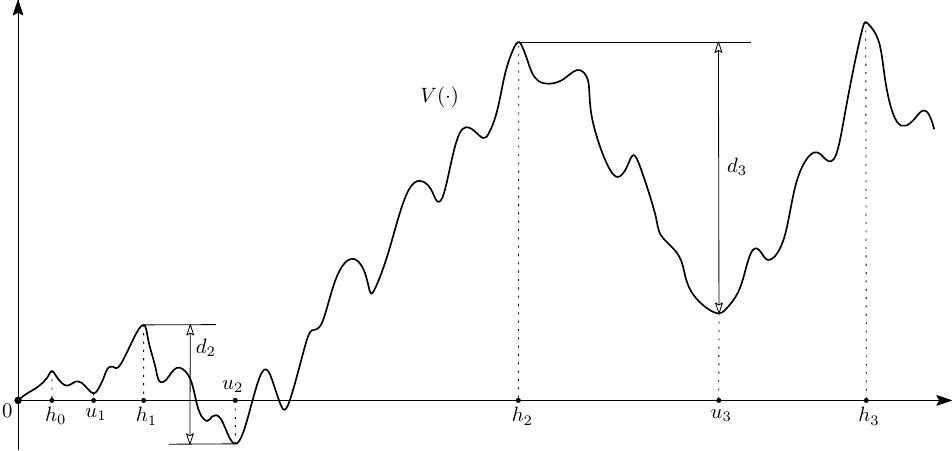}
\caption{Illustration of the definition of potential wells. Roughly, the $h_k$ and $u_k$ are interlacing local maxima and minima, and the $d_k$ are the corresponding depths of the wells.}
\label{f_wells}
\end{center}
\end{figure}

Then, we set 
(here and in the sequel $\argmin$ and $\argmax$ 
always choose the leftmost candidate in case there are 
ties)
\begin{align*}
 h_0 & := \min\{n\geq 0: V(n+1)<V(n)\}, \\
 \ell_1 & := \min\{x>h_0: \cL(x)>0\},\\
 m_1 & := \min\{x>\ell_1: V(x)>V(h_0)\},\\
 u_1 & := \argmin_{x\in[h_0,m_1]} V(x),\\
 d_1 & := \cL(u_1),
\end{align*}
and then, inductively, for $k \in \N$,
\begin{align*}
 \ell_{k+1} & := \min\{x>m_k: \cL(x)>d_k\},\\
 h_k & := \argmax_{x\in[m_k,\ell_{k+1}]} V(x),\\
 m_{k+1} & := \min\{x>\ell_{k+1}: V(x)>V(h_k)\},\\
 u_{k+1} & := \argmin_{x\in[h_k,m_{k+1}]} V(x),\\
 d_{k+1} & := \cL(u_{k+1}).
\end{align*}
Also, denote for $k\in \N$,
\begin{equation}
\label{eq:Delta-k}    
\Delta_k := V(h_k)-\max_{0\le j<h_k}V(j) > 0
\end{equation}
for the ladder height at $h_k$. 
Since $h_k$ is a strict record, $V(h_k)-V(j)\ge \Delta_k$ for all $j<h_k$.
We note that $u_k$ and $h_k$ are \emph{not} stopping times with respect to the sequence $V(0), V(1), \ldots$, but $m_k$ and $\ell_k$ are stopping times, which will be useful below.
We will need the following properties of this construction (again, see Figure~\ref{f_wells} for an illustration).

\begin{lemma}
    \label{lem:construction}
    The above construction defines two integer-valued sequences 
$(h_k)_{k\geq 0}$
and $(u_k)_{k\geq 1}$ such that 
$0\leq h_0<u_1<h_1<u_2<h_2<\cdots$
and the following properties hold.
\begin{thmenumi}[label=(\roman*)]
 \item 
     \label{lem:construction-i}
 For all $k \in \ZP$, $V(h_{k+1} ) > V (h_k)$.
 \item 
     \label{lem:construction-ii}
     For all $k \in \ZP$, $V(h_k) = \max_{0 \leq x \leq u_{k+1}} V(x) = \max_{u_k \leq x \leq u_{k+1}} V(x)$ (set $u_0:=0$).
 \item 
     \label{lem:construction-iii}
     For all $k \in \N$, $d_k = V(h_{k-1} ) - V (u_k) = \cD (0,h_k) = \cD (h_{k-1}, h_k)$.
 \item
     \label{lem:construction-iv}
     For all $k \in \N$, $ \cD (h_{k}, h_{k+1}) >  \cD (h_{k-1}, h_k)$.
 \end{thmenumi}
\end{lemma}
\begin{proof}
    First note that 
  by the definition of $m_k$, for all $k \in \N$,
  \[ V (h_k) = \max_{m_k \leq x \leq \ell_{k+1}} V(x) \geq V (m_k) > V (h_{k-1} ) ,\]
  which proves part~\ref{lem:construction-i}.
  
Next  we prove~\ref{lem:construction-ii}.  By definition of $h_0$, $V ( h_0 +1) < V(h_0)$ and $V(h_0) = \max_{0 \leq y \leq h_0+1} V(y)$, so that
    $\cL ( h_0 +1) = V(h_0) - V(h_0+1) > 0$. Thus $\ell_1 = h_0 +1$. Then by definition of $m_1$ we get 
    $V(h_0) = \max_{0 \leq y < m_1} V(y)$, and since $u_1 < m_1$ this means $V(h_0 ) = \max_{0 \leq x \leq u_1} V (x)$, which verifies the $k=0$ case of~\ref{lem:construction-ii}. Suppose $k \in \N$. By definitions of $h_k$ and $m_{k+1}$ we have $V(x) \leq V (h_k)$ for all $m_k \leq x < m_{k+1}$.
    Hence also $V(x) \leq V (h_{k-1} ) \leq V(h_k)$ for all $m_{k-1} \leq x < m_k$. Since $m_{k-1} \leq u_k \leq m_k \leq u_{k+1} < m_{k+1}$, it follows that $V(h_k) = \max_{u_k \leq x \leq u_{k+1}} V(x)$.
This completes the proof of~\ref{lem:construction-ii}.

For~\ref{lem:construction-iii}, first observe that $d_k = \cL (u_k) = V (h_{k-1}) - V (u_k)$, for $k \in \N$, by part~\ref{lem:construction-ii}. Then by definition of $\ell_{k+1}$ we have $\cL (x) \leq \cL (u_k)$ for $m_k \leq x < \ell_{k+1}$. Also, for $h_{k-1} \leq x < m_k$ we have $V (h_{k-1}) \geq V(x) \geq V(u_k)$ so $\cL (x) = V (h_{k-1} ) - V(x) \leq \cL (u_k)$. That is, $\cL(x) \leq \cL(u_k)$ for 
$h_{k-1} \leq x < \ell_{k+1}$ and in particular $\cL(u_k) = \max_{h_{k-1} \leq x \leq h_k} \cL (x)$ for all $k \in \N$. Also, $\cL ( \ell_{k+1} ) > \cL (u_k)$ where $h_k \leq \ell_{k+1} < h_{k+1}$, meaning that $\cL (u_k) < \cL ( u_{k+1} )$ for all $k \in \N$. Since $\max_{0 \leq x \leq h_0} \cL (x) =0$,
we conclude that 
\begin{equation}
\label{eq:L-max}
d_k = \cL (u_k) = \max_{0 \leq x \leq h_k} \cL (x) = \max_{h_{k-1} \leq x \leq h_k} \cL (x) .\end{equation}
Now using part~\ref{lem:construction-ii} and the formula~\eqref{eq:D-when-max} with~\eqref{eq:L-max}, we get
$\cD (0,h_k) = \max_{0 \leq x \leq h_k} \cL(x) = d_k$. Moreover,
by~\eqref{eq:D-when-max}, 
\[ \cD (0, h_k ) \geq \cD (h_{k-1} , h_k ) \geq \max_{h_{k-1} \leq \ell \leq u_k} V (\ell) - V(u_k) = V (h_{k-1}) - V (u_k) = d_k .\]
This completes the proof of~\ref{lem:construction-iii}, and~\ref{lem:construction-iv} also follows from the observation made above that  $\cL (u_k) < \cL ( u_{k+1} )$ for all $k \in \N$.
\end{proof}

In principle, 
one can understand how   the above quantities
behave by comparing the potential~$V$ to the Brownian
motion, via Donsker's theorem, for example.
Roughly speaking, we will have $h_{k+1}=Y_k h_k$ and $d_{k+1}=Z_k d_k$,
where $Y$s and $Z$s are approximately i.i.d.;
the details will be provided below.

Now, our programme is to show that
\begin{itemize}
 \item For all large enough~$k$, 
 $[h_{k}+1,h_{k+1}]$ are distinct stable clouds.
 \item All stable clouds go to the left 
 (i.e., have negative speeds), and the speed
 of the cloud of $[h_{k}+1,h_{k+1}]$ is 
 roughly $-\re^{-d_{k+1}}$.
 \item With $v_n$ being the speed of~$n$th particle and~$d$ being the depth of the well 
 containing~$n$,
 we have $v_n\approx -\exp(-d)\approx
  -\re^{-\sigma Z\sqrt{n}}$, where~$Z$ is a function
  of the environment with distribution 
  ``almost independent'' of~$n$.
\end{itemize}
To make the above rigorous, we need some preparations.
We intend to use 
the techniques of Section~\ref{sec:potentials}
(in particular, Lemma~\ref{l_stable_almost0}
 and the discussion after it, together 
with the ideas described in the last paragraph 
of that section)
to identify the stable clouds.
Then, we will try to use Lemma~\ref{l_depth_speed}
to estimate their speeds, and, in particular,
show that these speeds strictly decrease (so that
our candidate stable clouds are really stable clouds).
In Lemma~\ref{l_depth_speed}, the estimates for the speed
are off by essentially a factor proportional
to the cloud's size,
which, in our case
(for the would-be cloud $[h_{k}+1,h_{k+1}]$),
can be bounded above by~$h_{k+1}$, 
so, in principle, to ``separate'' the clouds, we need that 
$\frac{\re^{-d_k}}{h_k} \gg \re^{-d_{k+1}}$.
Also, we need to take
care of the height difference between the extremes,
with Condition~\eqref{ass:lh}; 
as explained later, this would mean that instead of~$h_k$
we would have~$h_k$ to some power, which changes
essentially nothing;
that is, we need the following relation to hold, 
for a (large) fixed~$M$
\begin{equation}
\label{e_dk_hk}
  \frac{\re^{-d_k}}{h^M_k} \gg \re^{-d_{k+1}}.
\end{equation}
The following result addresses this, showing that with high probability,
each $d_{k+1}$ has to be appreciably larger than $d_k$.
Recall the definition of $\kappa$ from~\eqref{unif_ellipticity},
and denote $\cF_m :=\sigma(V(j) : 0 \leq j \leq m)$.

\begin{lemma}
\label{l_dk_increase}
Assume Conditions~\eqref{ass:nestling}, \eqref{ass:ue}, and~\eqref{Sinai's_regime}.
Then, for any~$\delta>0$, we have
\begin{equation}
\label{eq_dk_increase}
1 - 
 \frac{\delta d_k + \kappa}{(1+\delta)d_k+\kappa}
 \leq
\IP\bigl[d_{k+1}>(1+\delta)d_k\mid \cF_{\ell_{k+1}}\bigr]
 \leq 1 - 
 \frac{\delta d_k - \kappa}{(1+\delta)d_k+\kappa}.
\end{equation}
\end{lemma}
\begin{proof}
 We have to be a bit careful with the preceding construction since 
$h_k$'s are not stopping times with respect
to the random walk~$V$. 
\begin{figure}
\begin{center}
\includegraphics{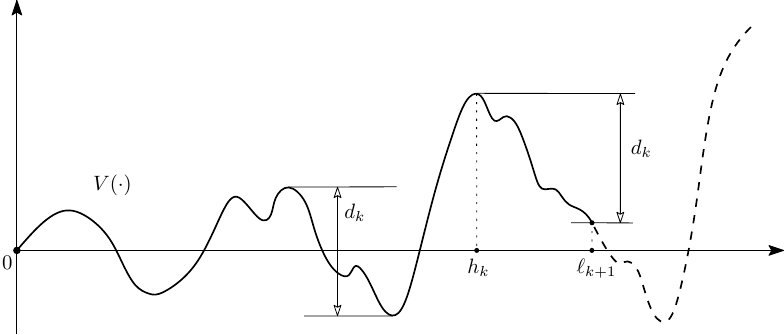}
\caption{Constructing the next well in the proof of Lemma~\ref{l_dk_increase}.}
\label{f_construction_hdl}
\end{center}
\end{figure}
However, $\ell_k$'s are stopping times,
so, when~$\ell_{k+1}$ happens, the process~$V$ is ``free to evolve''
(see Figure~\ref{f_construction_hdl}). 
Define
\[
 \tau := \min\big\{n \geq \ell_{k+1}: 
   V(n)\notin 
[V(h_k)-(1+\delta)d_k, V(h_k)]\big\},
\]
and 
\[
 A := \big\{V(\tau)<V(h_k)-(1+\delta)d_k\big\}
  = \big\{d_{k+1}>(1+\delta)d_k\big\}.
\]
Note that, using~\eqref{unif_ellipticity}, we have 
\begin{equation}
\label{V_ell_k+1_limits} 
 V(h_k)-d_k-\kappa \leq V(\ell_{k+1})
 \leq V(h_k)-d_k.
\end{equation}
We now use the optional stopping theorem for the martingale~$V$
with the stopping time~$\tau$: 
abbreviating
$p:=\IP ( A \mid \cF_{\ell_{k+1}})$ and 
 using~\eqref{unif_ellipticity} again, we get
\begin{align*}
 V(\ell_{k+1}) &= \IE(V(\tau) \2{A} \mid \cF_{\ell_{k+1}}) 
            +  \IE(V(\tau) \2{A^\rc} \mid \cF_{\ell_{k+1}})\\
   &\geq p(V(h_k)-(1+\delta)d_k - \kappa)
       + (1-p) V(h_k) \\
   & = V(h_k) - p((1+\delta)d_k+\kappa),
\end{align*}
so~\eqref{V_ell_k+1_limits} implies that
$p\geq \frac{d_k}{(1+\delta)d_k+\kappa}=1-
 \frac{\delta d_k + \kappa}{(1+\delta)d_k+\kappa}$.
Analogously,
we obtain that
\begin{align*}
 V(\ell_{k+1}) 
   &\leq p(V(h_k)-(1+\delta)d_k)
       + (1-p) (V(h_k)+\kappa) \\
   & = V(h_k) +\kappa - p((1+\delta)d_k+\kappa);
\end{align*}
again using~\eqref{V_ell_k+1_limits} we obtain
that $p\leq 1-
 \frac{\delta d_k - \kappa}{(1+\delta)d_k+\kappa}$
and conclude the proof of the lemma.
\end{proof}

Next, we need a result that permits us to control
the ``second deepest trap'' in~$[h_{k-1}+1,h_k]$.
\begin{lemma}
\label{lem:second_deepest}
Assume Conditions~\eqref{ass:nestling}, \eqref{ass:ue}, and~\eqref{Sinai's_regime}.
For any~$\delta \in (0,1)$, for all~$k\in \N$, 
\begin{equation}
\label{eq_second_deepest}
 \IP\big[\cD(m_k,h_k) \leq (1-\delta)d_k
       \mid \cF_{m_k}\big] 
 \geq 
 \frac{1-\delta}{1+\kappa d_k^{-1}}.
\end{equation}
\end{lemma}
\begin{proof}
Let $\delta \in (0,1)$. 
    For convenience write $\psi_k := (1-\delta) d_k$. Define
    \[ \sigma:= \min \{ n > m_k : \max_{m_k \leq s \leq n} V(s) - V(n) > \psi_k \} ,\]
    and set $M:= \max_{m_k \leq s \leq \sigma} V(s)$. Then define
    \[ \gamma := \min \{ n \geq \sigma : M - V(n) > d_k \}, \text{ and } \tau := \min \{ n \geq \sigma : V(n) > M \} .\]
    We claim that $\{ \gamma < \tau\} \subseteq \{ \cD (m_k , h_k ) \leq \psi_k \} $. Indeed,
if $\gamma < \tau$, then no new record above~$M$ is attained between~$\sigma$ and~$\gamma$. 
Hence at~$\gamma$ the running maximum is still~$M$, so $\cL (\gamma) > d_k$.
By definition of~$\sigma$, this means that
all descents before $\sigma$ are at most $\psi_k < d_k$, so $\ell_{k+1} \in [\sigma,\gamma]$. Thus $\max_{m_k \leq s \leq \ell_{k+1}} V(s) = M$ and its leftmost attainment~$h_k$ occurs before~$\sigma$. Therefore, by definition of~$\sigma$, $\cD ( m_k, h_k ) \leq \psi_k$. 

    Now consider the martingale $Y_n = M - V(n)$, $n \geq \sigma$. Applying the optional stopping theorem, and using the fact that $\psi_k < Y_\sigma \leq \psi_k +\kappa$, by~\eqref{unif_ellipticity},
    we get
    \[ \psi_k < Y_\sigma = \Exp ( Y_{\gamma \wedge \tau} \mid \cF_\sigma ) \leq \Pr ( \gamma < \tau \mid \cF_\sigma) ( d_k + \kappa) ,\]
    since $Y_\tau <0$. Hence
    \[ \Pr ( \gamma < \tau \mid \cF_\sigma) \geq \frac{\psi_k}{d_k + \kappa} = \frac{1-\delta}{1+ \kappa d_k^{-1} }.\]
    Taking conditional expectations given $\cF_{m_k}$ yields the claimed bound.
\end{proof}

The next result gives exponential growth bounds on the $d_k$.

\begin{lemma}
\label{lem:dk-asymptotics}
Assume Conditions~\eqref{ass:nestling}, \eqref{ass:ue}, and~\eqref{Sinai's_regime}.
Then
for any $\delta>0$ there exist $C_0 < \infty$ and $c_0 >0$ (depending on $\delta$ and on $d_1$) such that, for all $k \geq 2$,
\begin{equation}
\label{exp_control_dk}
\Pr \big[ \re^{(1-\delta) k} \leq d_k \leq \re^{(1+\delta) k} \mid \cF_{\ell_2} \big] \geq 1 - C_0 \re^{-c_0 k} .\end{equation}
\end{lemma}

\begin{remark}
\label{rem:dk-asymptotics}
    Consider
\[
 d_k = d_1 \cdot \frac{d_2}{d_1} \cdot \frac{d_3}{d_2}  \cdots \frac{d_k}{d_{k-1}} ,
\]  
and note that Lemma~\ref{l_dk_increase}, says, for $d_k$ large, for $\theta \geq 1$,
\[ \Pr \left[ \frac{d_{k+1}}{d_k} > \theta \mid \cF_{\ell_{k+1}} \right] \approx \frac{1}{\theta}. \]
Roughly speaking, this says that 
 $d_k$ is approximated by $Y_1 \cdots Y_k
 = \exp(\log Y_1 + \cdots +\log Y_k)$ where  $\IP[Y_j > \theta] = 1/\theta$
for $\theta \geq 1$. This means that
$\log Y_j$ is a unit mean exponential random variable, and then the intuition for the result
is by comparison with concentration for exponential sums. The following proof makes this idea precise.
\end{remark}

\begin{proof}[Proof of Lemma~\ref{lem:dk-asymptotics}]
For the duration of this proof, write $X_k := \log (d_{k+1}/d_k)$ and $\cG_k := \cF_{\ell_{k+1}}$; then $d_k$ is $\cG_k$-measurable.
Several times in this proof we use the fact that $d_{k+1} >d_k$ for all $k \in \N$
(see Lemma~\ref{lem:construction}).
First note that, for $\lambda \in (0,1)$, using the upper bound in  Lemma~\ref{l_dk_increase},
\begin{align*}
    \Exp [ \re^{\lambda X_k} \mid \cG_k ] & = \Exp [ (d_{k+1}/d_k)^\lambda \mid \cG_k ] \\
    & = 1 + \lambda \int_1^\infty \theta^{\lambda-1}  \Pr [ d_{k+1} > \theta d_k \mid \cG_k ] \ud \theta  \\
    & \leq 1 + \lambda \left(1 + \frac{2\kappa}{d_k} \right) \int_1^\infty \theta^{\lambda-2} \ud \theta.
\end{align*}
It follows that, for all $\lambda \in (0,1)$,
\begin{equation}
    \label{eq:M+}
       \Exp [ \re^{\lambda X_k }\mid \cG_k ] \leq M_+ (\lambda , d_k) := \frac{1}{1-\lambda} + \frac{2 \kappa \lambda /d_k}{1-\lambda} .
\end{equation}
A similar argument, using the lower bound in Lemma~\ref{l_dk_increase}, shows that, for all $\lambda >0$,
\begin{align}
    \label{eq:M-}
       \Exp [ \re^{-\lambda X_k }\mid \cG_k ] 
       & = 1 - \lambda \int_1^\infty \theta^{-\lambda-1} \Pr [ d_{k+1} > \theta d_k \mid \cG_k ] \ud \theta \nonumber\\
&       \leq M_- (\lambda , d_k) := \frac{1}{1+\lambda} + \frac{ \kappa \lambda /d_k}{2+\lambda} .
\end{align}
Iterating the bound~\eqref{eq:M+} and using that $d_{k+1} >d_k$ for all $k \in \N$
 leads to the bound, for all $\lambda \in (0,1)$ and  all $k > m$,
\begin{equation}
    \label{eq:M+iterated}
       \Exp \Bigl[ \exp \bigl(\lambda \sum_{j=m}^{k-1} X_j \bigr)\Bigmid \cG_m \Bigr] \leq M_+ (\lambda , D)^{k-m} , \text{ on } \{ d_m \geq D \}.
\end{equation}
Similarly, iterating~\eqref{eq:M-}, for all $\lambda >0$ and all $k > m$,
\begin{equation}
    \label{eq:M-iterated}
       \Exp \Bigl[ \exp \bigl(-\lambda \sum_{j=m}^{k-1} X_j \bigr)\Bigmid \cG_m \Bigr] \leq M_- (\lambda , D)^{k-m} , \text{ on } \{ d_m \geq D \}.
\end{equation}
Observe that
\begin{equation}
    \label{eq:M-limits}
    \lim_{D \to \infty} M_+ (\lambda, D ) = \frac{1}{1-\lambda}, \text{ and } \lim_{D \to \infty} M_- (\lambda, D ) = \frac{1}{1+\lambda}.
\end{equation}

 Next we show that sequence $d_k$ quickly exceeds any fixed value~$D$. To this end,
  let $A_k := \{ d_{k+1} \geq 2 d_k\}$. Then using the lower bound in Lemma~\ref{l_dk_increase} we have
 \begin{equation}
     \label{eq:d-doubles}
     \Pr ( A_k \mid \cG_k ) \geq \frac{d_k}{2d_k +\kappa} \geq p_0 := \frac{d_1}{2d_1+\kappa} >0 .
 \end{equation}
 Let $N_m$ denote the number of $A_1, \ldots, A_{m-1}$ that occur, and write $r:= \lceil \log_2 (D/d_1) \rceil$. Then if $N_m \geq r$ we have $d_m \geq 2^r d_1 \geq D$, so we have $\Pr ( d_m < D \mid \cG_1) \leq \Pr ( N_m < r \mid \cG_1 )$. Here we have from~\eqref{eq:d-doubles}, by iterating the conditional expectation, $\Exp [ \re^{-t N_m} \mid \cG_1 ] \leq (1 - p_0 + p_0 \re^{-t} )^{m-1}$. Hence we obtain the Chernoff-type bound $\Pr (N_m < r \mid \cG_1 ) \leq \re^{tr} (1-p_0 + p_0 \re^{-t} )^{m-1}$. It follows that, for $C_1 < \infty$ and $c_1 >0$ (depending on $d_1$)
 \begin{equation}
     \label{eq:fixed-depth}
     \Pr ( d_m < D \mid \cG_1 ) \leq C_1 \re^{-c_1 m}, \text{ for all } m \in \N.
 \end{equation}

Fix $\delta>0$. We prove that, for $C < \infty$ and $c>0$
(depending on $d_1$),
\begin{equation}
    \label{eq:dk-upper-bound}
    \Pr ( \log d_k > (1+\delta )k \mid \cG_1 ) \leq C \re^{-ck} .
\end{equation}
Set $m = \lfloor \eta k \rfloor$ for $\eta \in (0,1)$ that we will choose to be small later.
Write
\begin{align*}
    \Pr ( \log d_k >  (1+\delta )k \mid \cG_1 )
    \leq \Pr ( \log d_m > \tfrac{\delta}{4} k \mid \cG_1) + 
    \Pr \Bigl( \sum_{j=m}^{k-1} X_j > ( 1 + \tfrac{3\delta}{4} ) k \Bigmid \cG_1 \Bigr) .
\end{align*}
To bound the first term on the right-hand side,
take $\lambda \in (0,1)$. It follows from~\eqref{eq:M+} that $\Exp [ \re^{\lambda X_k} \mid \cG_k] \leq M_+ (\lambda , d_1 ) =: M_1 < \infty$. Iterating this gives $\Exp [ \re^{\lambda ( \log d_m - \log d_1)} \mid \cG_1] \leq M_1^{m-1}$. Then, by Markov's inequality, 
\begin{equation}
    \Pr ( \log d_m > \tfrac{\delta}{4} k \mid \cG_1 ) \leq d_1^{\lambda} \exp ( - \lambda \tfrac{\delta}{4} k + (m-1) \log M_1 ) \leq C_1 \re^{-c_1 k},
\end{equation}
for $C_1 < \infty$ and $c_1>0$ (depending on $d_1$) provided $\eta < \delta/8$ (say) is small enough (depending on $\delta$ and $d_1$). Thus we have, allowing $C_1$ and $c_1$ to change from line to line, 
\begin{align*}
    \Pr ( \log d_k >  (1+\delta )k \mid \cG_1 )
    & \leq C_1 \re^{-c_1 k}  + 
    \Pr \Bigl( \sum_{j=m}^{k-1} X_j > ( 1 + \tfrac{3\delta}{4} ) k \Bigmid \cG_1 \Bigr) \\
    & \leq C_1 \re^{-c_1 k} +     \Pr \Bigl( \sum_{j=m}^{k-1} X_j > ( 1 + \tfrac{3\delta}{4} ) k, \, d_m \geq D \Bigmid \cG_1 \Bigr),
\end{align*}
using~\eqref{eq:fixed-depth}.
Then from~\eqref{eq:M+iterated} and Markov's inequality, for $\lambda \in (0,1)$,
\begin{align*}
      \Pr \Bigl( \sum_{j=m}^{k-1} X_j > ( 1 + \tfrac{3\delta}{4} ) k, \, d_m \geq D  \Bigmid \cG_1 \Bigr)
      & \leq M_+ (\lambda, D)^{k-m} \exp \left( - \lambda ( 1+ \tfrac{3\delta}{4} )k \right) \\
      & \leq \exp \left( - k \left( \lambda ( 1+ \tfrac{3\delta}{4} ) - (1-\eta) \log M_+ (\lambda , D) \right) \right).
\end{align*}
We already chose $\eta < \delta/8$; then choose $\lambda >0$ small enough so that
\[ \lambda ( 1+ \tfrac{3\delta}{4} )  + (1-\eta) \log (1-\lambda) = \lambda  (   \tfrac{3\delta}{4} + \eta ) + O (\lambda^2) > \frac{\delta}{2} \lambda .\]
Then, by~\eqref{eq:M-limits}, we can choose $D$ large enough and we verify~\eqref{eq:dk-upper-bound}.

Finally, we prove for $\delta \in (0,1)$ that, for $C < \infty$ and $c>0$ (depending on $d_1$)
\begin{equation}
    \label{eq:dk-lower-bound}
    \Pr ( \log d_k < (1-\delta )k \mid \cG_1 ) \leq C \re^{-ck} .
\end{equation}
Choose $\eta \in (0,\delta/2)$ and again put $m = \lfloor \eta k \rfloor$. Then
\begin{align*}
\Pr ( \log d_k < (1-\delta) k \mid \cG_1 ) & \leq \Pr \Bigl( \sum_{j=m}^{k-1} X_j < ( 1 -\delta ) k, \, d_m \geq D  \Bigmid \cG_1 \Bigr) + \Pr ( d_m < D \mid \cG_1 ) ,
\end{align*}
where the final term satisfies~\eqref{eq:fixed-depth}. Using~\eqref{eq:M-iterated} and Markov's inequality,
\begin{align*}
    \Pr \Bigl( \sum_{j=m}^{k-1} X_j < ( 1 -\delta ) k, \, d_m \geq D  \Bigmid \cG_1 \Bigr)
    & \leq M_- (\lambda , D)^{k-m} \exp \left( \lambda (1-\delta) k \right) \\
    & \leq \exp \left( - k \left( -\lambda (1-\delta) - (1- \eta) \log M_- (\lambda,D) \right) \right).
\end{align*}
Here, by~\eqref{eq:M-limits}, $\log M_- (\lambda,D) \to - \log (1+\lambda)$ as $D \to \infty$, and 
$-\lambda (1-\delta) + (1- \eta) \log (1+\lambda ) = \lambda (\delta -\eta) + O (\lambda^2)$.
So we can choose $\lambda$ small enough and then $D$ large enough to verify~\eqref{eq:dk-lower-bound}.
\end{proof}

\subsection{Stability and separation for the candidate clouds}
\label{sec:candidate-clouds}

In Sinai's regime Theorem~\ref{t_stable_h_k_k+1}
gives our main structural result, which is
a central step in proving Theorems~\ref{thm:RE_Sinai} and~\ref{t_Sinai_particle}. The result shows that 
candidate clouds $[h_k+1,h_{k+1}]$ are indeed intrinsically stable,
and obtains some quantification of associated statistics. 
Let $u^{(k)}$
and ${\tilde u}^{(k)}$ be the candidates for 
the normal and the 
infinite-pressure speeds of the cloud candidate
$[h_{k-1}+1,h_k]$ (computed 
using the formulas~\eqref{V_v_explicit} 
and~\eqref{speed_infinite_pressure}, respectively).

\begin{theorem}
\label{t_stable_h_k_k+1}
Assume Conditions~\eqref{ass:nestling}, \eqref{ass:ue}, \eqref{ass:lh}, and~\eqref{Sinai's_regime}.
For $\IP$-a.a.\ environments~$\omega$ there exists
$k_0=k_0(\omega) \in \N$ such that the following hold.
\begin{thmenumi}[label=(\roman*)]
\item \label{t_stable_h_k_k+1-i}
 It holds that $\lim_{k \to \infty} k^{-1} \log d_k = 1$, a.s.
 \item \label{t_stable_h_k_k+1-ii}
 $[h_k+1,h_{k+1}]$ 
 is intrinsically stable
for each $k\geq k_0$.
 \item \label{t_stable_h_k_k+1-iii}
 $|u^{(k)}|>|{\tilde u}^{(k+1)}|$ for all $k\geq k_0$.
 \item \label{t_stable_h_k_k+1-iv}
It holds that  $\lim_{k \to \infty} d_k^{-1} \log|u^{(k)}| = -1$.
 \item \label{t_stable_h_k_k+1-v}
 Let~$\hat v\leq 0$ be the speed of the rightmost stable cloud 
in $[1,h_{k_0}]$ (considered as an isolated system);
then $|\hat v|>|{\tilde u}^{(k_0+1)}|$.
\end{thmenumi}
\end{theorem}
\begin{proof}
For part~\ref{t_stable_h_k_k+1-i}, fix $\delta >0$. Then from Lemma~\ref{lem:dk-asymptotics} and the Borel--Cantelli lemma (conditioned on $\cF_{\ell_2}$) we have that $\Pr ( 1-\delta \leq k^{-1} \log d_k  \leq 1+\delta \text{ all but finitely often} \mid \cF_{\ell_2} ) =1$, which gives the claimed convergence, since $\delta>0$ was arbitrary.

For the proofs of parts~\ref{t_stable_h_k_k+1-ii} and~\ref{t_stable_h_k_k+1-iii} we make use of some more notation and several families of events defined in terms of the potential.
Define $H_k:= h_k-h_{k-1}$ to be the size of the candidate
cloud $[h_{k-1}+1,h_k]$.
Define the events (with $\delta_0$ from condition~\eqref{ass:lh} and $\Delta_k$ from~\eqref{eq:Delta-k})
\begin{align*}
 G_1^{(k)} & := \big\{H_k\leq \re^{3k}\big\},
 \\
 G_2^{(k)} & := \big\{d_{k+1}>(1+k^{-2})d_k\big\},
 \\
 G_3^{(k)} & := \big\{1-\re^{-\Delta_k}\geq \re^{-5\delta_0^{-1}k}\big\},
 \\
 G_4^{(k)} & := \big\{\re^{k/2}\leq d_k \leq \re^{2k}\big\}. 
 \end{align*}

 We next show that, a.s., each of the events $G_1^{(k)}$, $G_2^{(k)}$, $G_3^{(k)}$, and $G_4^{(k)}$ occurs for all but finitely many $k$. 
 For $G_4^{(k)}$, this fact is a consequence of Lemma~\ref{lem:dk-asymptotics}. 
 We consider the rest.

Consider $G_1^{(k)}$.
Let us prove that  there exists $c>0$
such that
\begin{equation}
\label{Hk/dk2}
 \IP[H_{k+1} > \re^{k/4} d_k^2\mid \cF_{m_{k}}]\leq \re^{-c k}.
\end{equation}
Note that Lemma~\ref{l_dk_increase} implies that
 $d_{k+1}\leq \re^{\eps k}d_k$ with probability at least $1-c\re^{-\eps k}$. Then (recall~\eqref{df_left_climb}), define
\[
 T = \min\{ s>0 : \cL(m_k+s) > \re^{\eps k}d_k \}.
\]
It then holds that $m_k+T > h_{k+1}$ (and therefore $H_{k+1} < T$)
on the event $\{d_{k+1}\leq \re^{\eps k}d_k\}$; this is because $m_k+T$ has to be in a well of depth greater than $d_{k+1}$ in that case. 
Then, since on every interval of length $(\re^{\eps k}d_k)^2$ the descent $\cL$ is at least $\re^{\eps k}d_k$ with uniformly positive probability, we also have the estimate
\[
\IP\big[T> j (\re^{\eps k}d_k)^2\mid \cF_{m_{k}}\big] < \re^{-c' j},
\]
and hence, by taking $\eps<1/8$, we 
obtain~\eqref{Hk/dk2}.
It follows from~\eqref{Hk/dk2} and the Borel--Cantelli lemma that a.s., for all but finitely many $k \in \N$,
$H_{k+1} \leq \re^{k/4} d_k^2$. But $d_k \leq \re^{5k/4}$ for all but finitely many $k$, by part~\ref{t_stable_h_k_k+1-i}. Hence we verify that
\begin{equation}
    \label{eq:G1-occurs}
  \text{$\Pr$-a.s., }  G^{(k)}_1 \text{ occurs for all but finitely many $k \in \N$}.
\end{equation}

For $G_2^{(k)}$, we have from  the lower bound in Lemma~\ref{l_dk_increase} that 
\begin{equation}
 \label{est_G2k}
\Pr [ ( G_2^{(k)} )^\rc \mid \cF_{\ell_{k+1}} ] 
\leq \frac{k^{-2} d_k +\kappa}{(1+k^{-2}) d_k + \kappa} \leq k^{-2} + \kappa d_k^{-1}.
 \end{equation}
 Here $\sum_{k\in\N} d_k^{-1} < \infty$ by part~\ref{t_stable_h_k_k+1-i}. 
 Hence, by L\'evy's conditional Borel--Cantelli lemma and~\eqref{est_G2k}, 
\begin{equation}
    \label{eq:G2-occurs}
  \text{$\Pr$-a.s., } G^{(k)}_2 \text{ occurs for all but finitely many $k \in \N$}.
\end{equation}

Next consider $G_3^{(k)}$. Define
\begin{align*}
 G_5^{(k)} &:= \big\{h_k \leq 2\re^{3k}\big\}, \\
 G_6^{(k)} &:= \big\{\text{the first }3\re^{3k} 
 \text{ ladder heights of $V(\cdot)$ are each at least }
   \re^{-4\delta_0^{-1}k}\big\}. \end{align*}
Observe that $G_5^{(k)} \cap G_6^{(k)} \subseteq G_3^{(k)}$,
since,  on $G_5^{(k)} \cap G_6^{(k)}$,
$\Delta_k$ is one of the first~$2\re^{3k}$ ladder heights,
and so 
\[ 1 - \re^{-\Delta_k} \geq \frac{\Delta_k}{1+\Delta_k} \geq \frac{ \re^{-4\delta_0^{-1}k}}{1+\re^{-4\delta_0^{-1}k}} > \re^{-5\delta_0^{-1}k}, \]
for $k$ sufficiently large.
Since $h_k = h_0 + \sum_{j=1}^k H_j$, 
and $G^{(j)}_1$ occurs for all sufficiently large $j$, a.s., it follows that $G_5^{(k)}$ occurs
for all large enough $k$, a.s.
Now, due to Condition~\eqref{ass:lh},
\begin{align}
  \IP[G_6^{(k)}] 
  \geq 1 - ( 3\re^{3k} ) \gamma (\re^{-4\delta_0^{-1}k})^{\delta_0} = 1-3\gamma \re^{-k}.
  \label{est_G3k}
\end{align}
Then from the Borel--Cantelli lemma with~\eqref{est_G3k} and~\eqref{eq:G1-occurs}, we obtain
\begin{equation}
    \label{eq:G3-occurs}
  \text{$\Pr$-a.s., } G^{(k)}_3 \text{ occurs for all but finitely many $k \in \N$}.
\end{equation}
Combining~\eqref{eq:G1-occurs}, \eqref{eq:G2-occurs} and~\eqref{eq:G3-occurs} with Lemma~\ref{lem:dk-asymptotics}
  verifies 
that $G_1^{(k)}, G_2^{(k)}, G_3^{(k)}$ and $G^{(k)}_4$ occur a.s.~for all but finitely many $k$.

The idea of part~\ref{t_stable_h_k_k+1-ii}, i.e., 
to show that $[h_k+1,h_{k+1}]$ is  
intrinsically stable for large enough $k$, is to reason in the following way.
First, Lemma~\ref{l_stable_almost0}
implies that $[h_k+1, m_{k+1}]$ is intrinsically stable, and Lemma~\ref{l_depth_speed} shows that its intrinsic speed is about $\re^{-d_{k+1}}$. We will then show, roughly speaking, that 
the particles $[m_{k+1}+1,h_{k+1}]$ will join
that group. Note that the particles $[m_{k+1}+1,h_{k+1}]$
alone need not form a stable cloud; however,
Theorem~2.1 of~\cite{MMPW23} implies that
they can be decomposed into several stable
clouds, and the speed of the rightmost one is 
 maximal among these
(indeed, since it also has to be negative due to the fact
that the unique global maximum of the potential
on~$[h_k+1,h_{k+1}]$ is attained at~$h_{k+1}$,
the absolute value of the speed of that rightmost 
cloud should be minimal among these). Next,
we use Lemma~\ref{lem:second_deepest} to 
control the depth of that rightmost cloud
(i.e., to assure that it is likely much less
than~$d_{k+1}$),
and we also have some control
for the height increase at~$h_{k+1}$ (needed 
for the lower bound in Lemma~\ref{l_depth_speed}).
Then, Lemma~\ref{l_depth_speed} will
 show that the rightmost cloud should
catch up 
with $[h_k+1, m_{k+1}]$. We now give the details.

Assume that $k$ is large, 
and let us consider the particles
$[h_{k-1}+1,h_k]$ as a separate (finite) system.
By Theorem~2.1 of~\cite{MMPW23} this finite system
can be decomposed
into a finite number of stable clouds; let $[i_k,h_k]$
be the rightmost of these.
Our goal is to prove that $i_k=h_{k-1}+1$.
Note that Lemma~\ref{l_stable_almost0} 
and Remark \ref{rems:intrinsic-stability}\ref{rems:intrinsic-stability-iii}
imply that
the particles in $[h_{k-1}+1,m_k]$ belong to the same
stable cloud, so it is enough to prove that $i_k\leq m_k$. 

 Define
   \[  G_7^{(k)}  := \big\{\cD(m_k,h_k) \leq (1-k^{-2})d_k\big\}. \]
Lemma~\ref{lem:second_deepest} implies that
\begin{equation}
\label{est_G5k}
\IP[ ( G_7^{(k)} )^\rc \mid \cF_{m_k}] \leq 1 - \frac{1-k^{-2}}{1+\kappa d_k^{-1}} \leq  \kappa d_k^{-1} + k^{-2},
\end{equation}
which is a.s.~summable, by part~\ref{t_stable_h_k_k+1-i}. Hence L\'evy's conditional Borel--Cantelli lemma implies that, for a.a.~$\omega$,
the events $G_7^{(k)}$ occur for all but finitely many $k \in \N$.
Assume that $k$ is large enough so that $G_1^{(k)} \cap G_3^{(k)} \cap G_4^{(k)} \cap G_7^{(k)}$ occurs, and assume that
$i_k>m_k$. Then, let~$v'$ be
 speed of the leftmost cloud of the isolated system $[h_{k-1}+1,h_k]$,
and~$v''$ be the speed of $[i_k,h_k]$
 (when considered as separate systems).
Since $v''<0$ (as $V(h_k)>V(i_k-1)$) and cloud speeds are non-decreasing 
from left to right, we have $v'\leq v''<0$, and hence $|v'|\geq|v''|$.
We now show that, for all large~$k$, $|v'|<|v''|$, which gives a contradiction.
Since $v'<0$, the leftmost cloud $[h_{k-1}+1,j]$ has $V(j)>V(h_{k-1})$,
so $V(j)=\max_{[h_{k-1},j]}V$ by Lemma~\ref{l_global_maxima}; moreover,
$j\geq m_k$ (this cloud contains $[h_{k-1}+1,m_k]$), so $\cD(h_{k-1},j)\geq d_k$.
Hence Lemma~\ref{l_depth_speed} implies that
$|v'|\leq \eps_0^{-1}\re^{-d_k}$ and that, 
on $G_1^{(k)}\cap G_3^{(k)} \cap G_7^{(k)}$,
$|v''|\geq \eps_0 \re^{-3k} \re^{-5\delta_0^{-1}k}\re^{-d_k(1-k^{-2})}$
(apply~\eqref{depth_speed_lower} to $[i_k,h_k]$ using $V(h_k) - V(i_k-1) \geq \Delta_k$, $h_k -i_k +1 \leq H_k$, and $\cD(i_k -1,h_k) \leq \cD(m_k,h_k)$, since $i_k -1 \geq m_k$). We thus need to check that
$\re^{k^{-2}d_k}>\eps_0^{-2}\re^{(3+5\delta_0^{-1})k}$,
which is indeed so on~$G_4^{(k)}$ for all large
enough~$k$. This gives us a contradiction, and so   $[h_{k-1}+1, h_k]$
is intrinsically stable, proving part~\ref{t_stable_h_k_k+1-ii}.

Roughly speaking, part~\ref{t_stable_h_k_k+1-iii} says that the candidate
stable clouds $[h_k+1,h_{k+1}]$ (which part~\ref{t_stable_h_k_k+1-ii} says are intrinsically stable for large enough $k$) have, for large enough $k$, strictly ordered speeds. In fact,
it will follow from part~\ref{t_stable_h_k_k+1-iii} that, in the infinite system, all but finitely many of the candidate stable clouds remain distinct
stable clouds. 
The idea of the proof of part~\ref{t_stable_h_k_k+1-iii} is as follows.
 Lemma~\ref{l_depth_speed} tells us that $u^{(k)}$ and ${\tilde u}^{(k+1)}$
 are  roughly~$\re^{-d_k}$ and $\re^{-d_{k+1}}$, respectively, up to correction factors that  we need to control (here
 Condition~\eqref{ass:lh} is used to make sure $V(h_k) - V(h_{k-1})$ is not too small). Then Lemma~\ref{l_dk_increase} ensures that $d_{k+1}$ is sufficiently larger than $d_k$ to ensure the claimed inequality in speeds.


Part~\ref{t_stable_h_k_k+1-iii} follows from the claim
\begin{equation}
    \label{eq:Gs-imply-speeds}
    \text{ for all $k$ sufficiently large, } G_1^{(k)} \cap G_2^{(k)} \cap G_3^{(k)} \cap G_4^{(k)} \text{ implies } | u^{(k)} | > | \tilde u^{(k+1)} | .
\end{equation}
To verify the claim~\eqref{eq:Gs-imply-speeds}, 
recall from Lemma~\ref{lem:construction}\ref{lem:construction-iii}
that $d_k = \cD (h_{k-1},h_k)$. Then 
we have by Lemma~\ref{l_depth_speed} applied to
the system $[h_{k-1}+1,h_k]$ and the system $[h_{k}+1,h_{k+1}]$ we have
(since $V(h_k)-V(h_{k-1})\geq \Delta_k$)
\begin{equation}
\label{est_uk_tilde_uk}
 |u^{(k)}|\geq \frac{\eps_0}{H_k}(1-\re^{-\Delta_k})\re^{-d_k}
  \text{ and }
 |{\tilde u}^{(k+1)}| \leq \eps_0^{-1} \re^{-d_{k+1}}.
\end{equation}
Then, we need to prove that 
$\frac{\eps_0}{H_k}(1-\re^{-\Delta_k})\re^{-d_k}
> \eps_0^{-1} \re^{-d_{k+1}}$, which is equivalent
to $\re^{d_{k+1}-d_k}\geq \eps_0^{-2}H_k(1-\re^{-\Delta_k})^{-1}$.
But, on $G_2^{(k)}\cap G_4^{(k)}$ we have 
$\re^{d_{k+1}-d_k}\geq \exp(k^{-2}d_k)
\geq \exp(k^{-2}\re^{k/2})$, and, 
 on $G_1^{(k)}\cap G_3^{(k)}$ we have 
$\eps_0^{-2}H_k(1-\re^{-\Delta_k})^{-1}
\leq \eps_0^{-2} \exp(3k+5\delta_0^{-1}k)$.
Comparing these expressions,
we verify the claim~\eqref{eq:Gs-imply-speeds}. This proves part~\ref{t_stable_h_k_k+1-iii}.

Next we prove part~\ref{t_stable_h_k_k+1-iv}. Using the bounds in Lemma~\ref{l_depth_speed} we get, on $G_1^{(k)} \cap G_3^{(k)}$,
\[ \eps_0^{-1} \re^{-d_k} \geq | u^{(k)} | \geq \eps_0 \re^{-3k} \re^{-5\delta_0^{-1} k} \re^{-d_k} ,\]
which, together with part~\ref{t_stable_h_k_k+1-i}, yields part~\ref{t_stable_h_k_k+1-iv}. 


As for part~\ref{t_stable_h_k_k+1-v},
let $[i,h_k]$ be the rightmost cloud of the isolated system $[1,h_k]$, with speed~$\hat v$. 
Since $h_k$ is a strict record, 
we have $V(h_k)-V(i-1)\ge\Delta_k$ and $\hat v<0$. 
Also, $\cD(i-1,h_k)\le\cD(0,h_k)=d_k$ by Lemma~\ref{lem:construction}\ref{lem:construction-iii}. Hence, by~\eqref{depth_speed_lower},
\[
|\hat v|\ge\eps_0 h_k^{-1}(1-\re^{-\Delta_k})\re^{-d_k},
\]
while $|\tilde u^{(k+1)}|\le\eps_0^{-1}\re^{-d_{k+1}}$ by~\eqref{depth_speed_upper}. On $G_2^{(k)}\cap G_4^{(k)}\cap G_5^{(k)}\cap G_6^{(k)}$ we have
\[
\re^{d_{k+1}-d_k}\ge\exp(k^{-2}\re^{k/2}) > 4\eps_0^{-2}\re^{(3+4\delta_0^{-1})k}\ge\eps_0^{-2}h_k(1-\re^{-\Delta_k})^{-1}
\]
for large $k$. So~\ref{t_stable_h_k_k+1-v} holds for all large $k$; enlarge $k_0$ accordingly.
%
\end{proof}

\subsection{Proofs of Theorems~\ref{thm:RE_Sinai} and~\ref{t_Sinai_particle}}
\label{sec:sinai-theorem-proofs}

Basically, Theorem~\ref{t_stable_h_k_k+1} is the most 
complete result, which shows where the stable 
clouds are located and what are their speeds. Now let us discuss
how to obtain the results about convergence in distribution
in Theorems~\ref{thm:RE_Sinai} and~\ref{t_Sinai_particle}.

\begin{proof}[Proof of Theorem~\ref{thm:RE_Sinai}.]
Consider the isolated system $[1,h_{k_0}]$ whose cloud decomposition is
into some $\cC_1, \ldots, \cC_\ell$ and whose rightmost stable cloud $\cC_{\ell}$ has speed $\hat v$ where $| \hat v | > | \tilde u^{(k_0+1)}|$, as in 
Theorem~\ref{t_stable_h_k_k+1}\ref{t_stable_h_k_k+1-v}. The isolated system $[1,h_{k_0+1}]$ with infinite pressure from the right then 
has cloud decomposition into clouds $\cC_1, \ldots, \cC_\ell$ and $\cC_{\ell+1} = [h_{k_0}+1,h_{k_0+1}]$
(this follows from Algorithm 1 in~\cite{MMPW23}).
Since the speed of cloud $\cC_\ell$ is strictly less than that of cloud $\cC_{\ell+1}$, we may apply Lemma~\ref{lem:couple_infinite_pressure} to show that there is a coupling of the infinite system and the finite system  $[1,h_{k_0+1}]$ with infinite pressure from the right, in which the inter-particle distances corresponding to clouds $\cC_1, \ldots, \cC_\ell$ coincide after some finite time. It follows that $\cC_1, \ldots, \cC_\ell$ are stable clouds in the infinite system, in the sense of Definition~\ref{def:stable-cloud-finite}. Similarly, consider the finite system $[1, h_{k+1}]$, $k > k_0$, with infinite pressure from the right. By Theorem~\ref{t_stable_h_k_k+1}\ref{t_stable_h_k_k+1-iii},
this finite system has cloud decomposition $\cC_1, \ldots, \cC_\ell, [h_{k_0}+1, h_{k_0+1}], \ldots, [h_k+1,h_{k+1}]$, with strictly increasing cloud speeds. Another application of Lemma~\ref{lem:couple_infinite_pressure}, using the fact that $k>k_0$ was arbitrary, shows that 
 $[h_{k_0}+1, h_{k_0+1}]$, $[h_{k_0+1}+1,h_{k_0+2}]$, \ldots, are stable clouds in the infinite system.
 In particular, this means that sequences $s_k$
and $u^{(k)}$ coincide up to a finite ($\omega$-dependent) shift, and the same for sequences $S_k$ and $H_k$.

To prove part~\ref{thm:RE_Sinai-i} of the theorem, we need to show that each particle in a cloud satisfies a strong law with speed given by the cloud's speed. Consider one of the clouds~$\cC_k$, with speed~$s_k$, that we now know is stable, and suppose particle~$j$ is in cloud~$\cC_k$.
There is a coupling 
(analogous to the coupling of two finite systems  in Proposition~\ref{prop:stoch-monotonicity}\ref{prop:stoch-monotonicity-ii})
of the infinite system and the finite system made up of particles in clouds $\cC_1, \ldots, \cC_k$ in which all particles in the infinite system are no further to the right than in the finite system. In the finite system we have from~Theorem~2.1 of~\cite{MMPW23} that particle $j$ 
satisfies $t^{-1} X_j (t) \to s_k$, a.s. So in the infinite system we get
\[ \limsup_{t \to \infty} t^{-1} X_j (t) \leq s_k , \text{ $\Pr_\omega^x$-a.s., for $\Pr$-a.a.\ $\omega$.} \]
On the other hand, consider the finite system made up of particles in clouds $\cC_1, \ldots, \cC_k, \cC_{k+1}$ with infinite pressure from the right.
Now there is a coupling 
(analogous to the coupling of two finite systems  in Proposition~\ref{prop:stoch-monotonicity}\ref{prop:stoch-monotonicity-i})
of the infinite system and the finite system  in which all particles in the finite system are no further to the right than in the infinite system. We argued above that $\cC_k$ is still a stable cloud in this bigger finite system (with infinite pressure). So, again using Theorem~2.1 of~\cite{MMPW23}, in the infinite system we get
\[ \liminf_{t \to \infty} t^{-1} X_j (t) \geq s_k ,  \text{ $\Pr_\omega^x$-a.s., for $\Pr$-a.a.\ $\omega$.} \]
Thus we obtain the strong law~\eqref{eq:slln}, proving part~\ref{thm:RE_Sinai-i}.

Recall that $s_k$
and $u^{(k)}$ coincide up to a finite  shift. Then combining Theorem~\ref{t_stable_h_k_k+1}\ref{t_stable_h_k_k+1-i} and~\ref{t_stable_h_k_k+1-iv} shows that $u^{(k)} = - \exp ( - \re^{k(1+o(1))})$, as $k \to \infty$, $\IP$-a.s. Then part~\ref{thm:RE_Sinai-ii} of the theorem follows.

Next we prove part~\ref{thm:RE_Sinai-iii}.
For large $k$, as argued above, $s_k = u^{(k+c)}$ and $S_k = H_{k+c}$ where $c=c(\omega)$ is random. 
From Theorem~\ref{t_stable_h_k_k+1}\ref{t_stable_h_k_k+1-iv}
we have that, a.s., as $k \to \infty$,
\begin{equation}
\label{eq:s-d-ratios}
\frac{\log |s_{k+1}|}{\log | s_k|} = (1+o(1)) \frac{d_{k+c+1}}{d_{k+c}} .\end{equation}
From the upper and lower bounds in Lemma~\ref{l_dk_increase},
and the fact that $d_k \to \infty$ we have, $\Pr$-a.s.,
\[ \lim_{k \to \infty} \Pr \Bigl[ \frac{d_{k+1}}{d_k} > y \Bigmid \cF_{\ell_{k+1}}  \Bigr] = \frac{1}{y},\]
and hence, by dominated convergence,
$\lim_{k \to \infty} \Pr [ A \cap \{ d_{k+1}/d_k > y \} ] = \Pr (A)/y$
for every $A \in \cF'_{\ell_j}$, with fixed $j \in \N$,  where $\cF'_m := \sigma ( a_i, b_i : i \leq m)$.
Here we use that Lemma~\ref{l_dk_increase} holds, with the same
proof, with $\cF_{\ell_{k+1}}$ replaced by $\cF'_{\ell_{k+1}}$. Every event in the environment $\sigma$-algebra can be approximated in probability by such events, so the same holds for every such event $A$. Applying this with $A = \{ c = j\}$ and summing over $j$, we get
$\lim_{k \to \infty} \Pr [  d_{k+c+1}/d_{k+c} > y  ] = 1/y$.
Combined with~\eqref{eq:s-d-ratios} this yields~\eqref{eq_relative_speeds}. Similarly, we have
\begin{align*}
    \Pr \Bigl[ \frac{d_{k+1}}{d_k} > y_1, \,
    \frac{d_{k+2}}{d_{k+1}} > y_2 \Bigmid \cF_{\ell_{k+1}}  \Bigr]
    & = \Exp  \Bigl[ \1 { {d_{k+1}} /{d_k} > y_1 }  \Pr \Bigl[
    \frac{d_{k+2}}{d_{k+1}} > y_2 \Bigmid \cF_{\ell_{k+2}}  \Bigr] \Bigmid \cF_{\ell_{k+1}} \Bigr],
\end{align*}
so that, by two applications of dominated convergence,
\begin{align*}
    \lim_{k \to \infty} \Pr \Bigl[ \frac{d_{k+1}}{d_k} > y_1, \,
    \frac{d_{k+2}}{d_{k+1}} > y_2  \Bigr] & = \frac{1}{y_1} \cdot \frac{1}{y_2} .
\end{align*}
The same idea applies for arbitrary finite sequences.
A similar argument to the marginal result,  with an indicator $1_A$, followed by conditioning on $\{c=j\}$, handles the random shift for the whole finite vector, 
and hence
yields asymptotic independence, completing the proof of part~\ref{thm:RE_Sinai-iii}.

To prove part~\ref{thm:RE_Sinai-iv},
we will apply Donsker's theorem.
To describe the limit random variable,
consider 
a standard Brownian motion $(B_t, t\geq 0)$,
and define
\begin{align*}
\ell &: = \inf\big\{t>0: \max_{0\leq s\leq t}B_s-B_t=1\big\},\\
h & := \argmax_{t\in [0,\ell]} B_t,\\
m & := \inf\{t>\ell : B_t = B_h\},\\
u & := \argmin_{t \in [h,m] } B_t ,\\
d' & := \max_{0 \leq s \leq u} B_s - B_u , \\
\ell' & := \inf\big\{t>m: \max_{s\leq t}B_s-B_t=d'\big\},\\
h' & := \argmax_{t\in [m,\ell']} B_t;
\end{align*}
see Figure~\ref{f_Brownian_func}.
\begin{figure} 
\begin{center}
\includegraphics{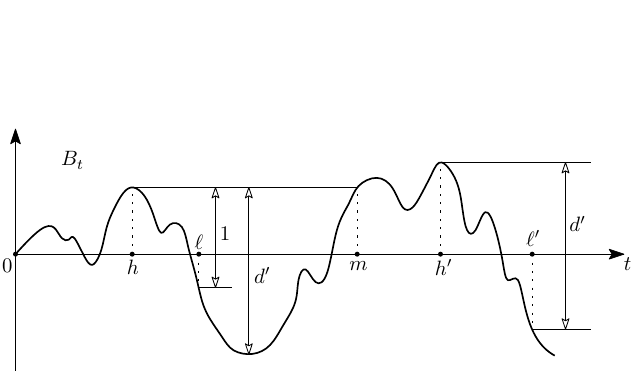}
\caption{On the proof of the convergence-in-distribution results in Theorem~\ref{thm:RE_Sinai}: definition of the auxiliary
random variables on a Brownian trajectory.}
\label{f_Brownian_func}
\end{center}
\end{figure} 
Then define the random variable
\begin{equation}
    \label{eq:Y-def}
    Y := \frac{h'-h}{(d')^2} .
\end{equation}
Note that if we replace the definition of $\ell$ by 
$\inf \{t>0: \max_{0\leq s\leq t}B_s-B_t= d \}$ for $d>0$,
the distribution of $Y$ is unchanged, by Brownian scaling.

Let $\sigma^2 := \IE[V(1)^2]$. We define a quantity like $Y$ for an arbitrary continuous trajectory $f$, which we call $\Phi (f)$. Formally, for continuous $f: \RP \to \R$ set $M_f (t) := \max_{0 \leq s \leq t} f(s)$ and   
\begin{align*}
\ell_f &: = \inf\big\{t>0: M_f (t) - f(t) =1\big\},\\
h_f & := \argmax_{t\in [0,\ell_f]} f(t),\\
m_f & := \inf\{t>\ell_f : f(t) = f(h_f) \},\\
u_f & := \argmin_{t \in [h_f,m_f] } f(t) ,\\
d'_f & := M_f (u_f) - f (u_f) , \\
\ell'_f & := \inf\big\{t>m_f: M_f(t) - f(t) = d'_f\big\},\\
h'_f & := \argmax_{t\in [m_f,\ell'_f]} f(t).
\end{align*}
Define $\Phi (f) := \frac{h'_f-h_f}{(d'_f)^2}$; then $Y = \Phi ( B)$.
When $f=B$ all the above quantities are finite and the relevant extrema are a.s.~unique. Moreover, passage-time-like quantities are genuine crossings of levels. It follows by the usual continuity of first passage times at genuine crossings and of $\argmax/\argmin$ at unique extrema that $\Phi$ is a.s.~continuous at $B$. 
For $c>0$, define $f_c$ by $f_c (t) := f (ct)$; then $h_{f_c} = h_f /c$, $h'_{f_c} = h'_f /c$ and $d'_{f_c} = d'_f$. Consequently $\Phi (f_c) = c^{-1} \Phi (f)$.
Thus Brownian scaling shows that $\Phi ( \sigma B )$ has the same distribution as $Y/\sigma^2$.

Now condition on $\cF_{m_{k-1}}$.  For $t\ge0$ write
\[
 j_k(t)=\lfloor d_{k-1}^2t\rfloor,
 \text{ and } \theta_k(t)=d_{k-1}^2t-j_k(t),
\]
and define the linearly interpolated process
\[
 W_k(t)=\frac{1}{d_{k-1}}\Big[
 V(m_{k-1}+j_k(t))-V(m_{k-1})
 +\theta_k(t)\{V(m_{k-1}+j_k(t)+1)-V(m_{k-1}+j_k(t))\}
 \Big].
\]
Since $m_{k-1}$ is a stopping time, the post-$m_{k-1}$ increments are fresh i.i.d.\ increments.  Since $d_{k-1}$ is $\cF_{m_{k-1}}$-measurable and $d_{k-1}\to\infty$ a.s. by 
Theorem~\ref{t_stable_h_k_k+1}\ref{t_stable_h_k_k+1-i}, the conditional form of Donsker's theorem gives $W_k\Rightarrow\sigma B$ (denoting weak convergence).  


For $\eps = (\eps_1,\eps_2,\eps_3) \in [0,1]^3$, let $\Phi_\eps$ be defined as $\Phi$ but with the passage levels $1$, $f(h_f)$, and $d'_f$ (in $\ell_f, m_f$, and $\ell'_f$, respectively) replaced by $1+\eps_1, f(h_f) +\eps_2,$ and $d'_f + \eps_3$. Thus $\Phi_0 = \Phi$. The argument above (genuine crossings,
unique extrema) shows that, a.s., 
$\sup_{\eps \in [0,\delta]^3} | \Phi_\eps (f) - \Phi (\sigma B) | \to 0$
as $f \to \sigma B$ locally uniformly and $\delta \to 0$.

The discrete construction starting at $m_{k-1}$
is the same successive well construction. Its first threshold is $d_{k-1}$,
its first maximum is $h_{k-1}$, the subsequent minimum has depth $d_k$, and the next
maximum is $h_k$. Since $W_k$ is linear between mesh points, its running descent 
is monotone between mesh points, and all extrema of $W_k$ are attained at mesh points. Hence
each first passage in the discrete construction (which occurs at a mesh point) coincides with the continuous first passage of $W_k$ at the level attained there. By~\eqref{unif_ellipticity}, this level exceeds the nominal one by at most $\kappa/d_{k-1}$. Thus
\[ \frac{H_k}{d_k^2} = \Phi_{\eps_k} (W_k) ,\]
for some $\eps_k \in [0,\kappa/d_{k-1}]^3$, and the continuous mapping theorem
gives $H_k/d_k^2 \Rightarrow \Phi ( \sigma B ) \eqd Y/\sigma^2$,
conditionally on $\cF_{m_{k-1}}$. The same argument shows $H_k/d_k^2 \Rightarrow Y/\sigma^2$ conditionally on $\cF'_{m_{k-1}}$, and then, as in the proof of part~\ref{thm:RE_Sinai-iii}, the conditional form of the convergence handles the shift $c$, so $H_{k+c}/d_{k+c}^2 \Rightarrow Y/\sigma^2$.
Finally, $S_k = H_{k+c}$, and, by 
 Theorem~\ref{t_stable_h_k_k+1}\ref{t_stable_h_k_k+1-iv},  $\log|s_k|^{-1}/d_{k+c}\to1$. Then Slutsky's theorem completes the proof of~\eqref{eq_sizes_depths}. All that remains is the explicit distribution for $Y$ given at~\eqref{eq:Y-distribution}: we give the proof of this in Lemma~\ref{lem:Y-distribution} below.
\end{proof}

Now we turn to the proof of Theorem~\ref{t_Sinai_particle}, the per-particle speed result.

\begin{proof}[Proof of Theorem~\ref{t_Sinai_particle}]
First we have from Theorem~\ref{thm:RE_Sinai}\ref{thm:RE_Sinai-i} that
$v_n = \lim_{t \to \infty} t^{-1} X_n (t) = s_k$ where $k$ is such that $n \in \mathcal{C}_k$, which, as explained in the proof of Theorem~\ref{thm:RE_Sinai},
means that, for all large enough $n$, $n \in [ h_{k'-1}+1, h_{k'}]$ for $k'$
with $| k - k'|$ bounded by an $\omega$-dependent constant. In particular, the strong law~\eqref{eq:slln-by-particle} follows from the strong law in  Theorem~\ref{thm:RE_Sinai}.
Theorem~\ref{t_stable_h_k_k+1}\ref{t_stable_h_k_k+1-iv} then implies that $v_n = -\re^{-d_{k'}(1+o(1))}$, a.s. Thus to complete the proof of the theorem
we must show that 
$n^{-1/2} d_{k'}$ converges in law, where $d_{k'}$ is
 the depth of the potential well containing $n$, and we must identify the limit distribution.

Let $B_t$ be a standard Brownian motion. 
We want to define a random variable $Z$ as the depth of the well (in trajectory $B$) that contains time~1. One cannot repeat the 
construction of quantities associated with successive potential wells
from Section~\ref{s_infinite_Sinai} in the Brownian setting, as there is no ``first'' well in that sense. Instead, we need a local definition of the depth of the given well. 
To this end, recall that $M_f (t) = \max_{0\leq s \leq t} f(s)$, and for continuous $f : \RP \to \R$, define
\[ \rho (f) := \inf \Bigl\{ t \geq 1 : f (t) > M_f (1) \Bigr\} \]
and then (see Figure~\ref{f_valley1})
\[
D_1 (f) := \sup_{0 \leq t \leq \rho(f)} \left\{ M_f (t) - f(t) \right\} .\]
\begin{figure} 
\begin{center}
\includegraphics{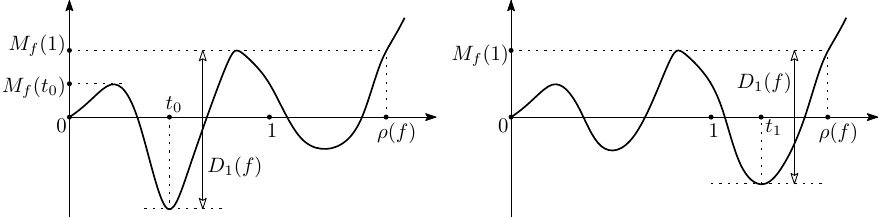}
\caption{On the definition of $D_1(f)$: on the left picture
the supremum is attained at $t_0<1$, while on the right 
picture it is attained at $t_1>1$ with $M_f(t_1)=M_f(1)$.}
\label{f_valley1}
\end{center}
\end{figure} 
Then set 
\begin{equation}
\label{df_Y_Sinai}
Z := D_1 (B) , \end{equation}
which we take as our definition of the depth of the Brownian well containing time~1.

Consider the linearly interpolated and scaled potential,
\[ \mathcal{V}_n (t) := n^{-1/2} \bigl( V ( \lfloor n t \rfloor ) +
(nt -  \lfloor n t \rfloor ) (V (\lfloor n t \rfloor +1) 
- V (\lfloor n t \rfloor ) ) \bigr) .\]
 Almost surely, $B_1<M_B(1)$, the first post-$1$ hitting time of $M_B(1)$ is finite, and Brownian motion crosses this level immediately.  It follows that $f \mapsto D_1(f)$ is a.s.~continuous at $f = B$.  Hence, for the linearly interpolated potential $\mathcal{V}_n$, Donsker's theorem and the continuous mapping theorem give
$ D_1(\mathcal{V}_n)\Rightarrow D_1(\sigma B)=\sigma Z$.

If $h_{k'-1}<n<h_{k'}$, Lemma~\ref{lem:construction} implies that the largest running 
descent before the first future record maximum after $n$ is exactly $d_{k'}$; since linear interpolation creates no new extrema,
$D_1(\mathcal{V}_n)=n^{-1/2}d_{k'}$.
Finally, 
every $h_k$ is a running-maximum time of the potential, so
\[
 \{n=h_k\text{ for some }k\}
 \subseteq
 \left\{V(n)=\max_{0\le j\le n}V(j)\right\}.
\]
By reversing the increments, the probability of this event equals
 the probability that a centered finite-variance random walk stays non-negative for $n$ steps, and hence 
$\Pr(n=h_k\text{ for some }k)=O(n^{-1/2})$.  Therefore 
if we set $Z_n := - n^{-1/2} \log | v_n |$ we have
$Z_n = (1+o(1)) n^{-1/2}d_{k'}\Rightarrow\sigma Z$.
It remains to identify the explicit distribution of $Z$ given at~\eqref{eq:Z-distribution}: this is
done in Lemma~\ref{lem:Z-distribution} below. 
\end{proof}

\subsection{Explicit limit distributions}
\label{sec:explicit_limits}

Let BES$(q)$ denote the Bessel process with dimension $q$.

\begin{lemma}
\label{lem:Y-distribution}
The distribution of $Y$ given by~\eqref{eq:Y-def} is the same as the first hitting time of~2 by a BES(3) process started at $0$, and hence has distribution function given by~\eqref{eq:Y-distribution}.
\end{lemma}
\begin{proof}
Let $T_x$ denote the hitting time of $x$ by a BES(3) process started at $0$,
and note that
\begin{equation}
    \label{eq:BES-hitting}
    \Exp \re^{-\lambda T_x} =  \frac{x \sqrt{2\lambda}}{\sinh (x \sqrt{2\lambda})} ;
\end{equation}
see equation 2.0.1 of~\cite[p.~469]{borodin}.
    Recalling the definition of $Y$ from~\eqref{eq:Y-def}, and the quantities $m$, $h$, $h'$ and $d'$ defined prior to $Y$, decompose $Y = Y_1 + Y_2$
    where
    \[ Y_1 := \frac{m-h}{(d')^2}, \text{ and } Y_2 := \frac{h'-m}{(d')^2} .\]
    Condition on $d'$. Williams's decomposition of a Brownian excursion at its maximum says that, conditional on the maximum being $d'$, the part before the maximum and the time-reversal of the part after the maximum are independent BES(3) processes started from $0$ and stopped on first hitting $d'$. Hence, given $d'$, $m-h$ has the same law as $T_{d'} + T'_{d'}$ where $T_{x}'$ is an independent copy of $T_x$. By Bessel scaling, $T_{d'}$ has the same distribution as $(d')^2 T_1$. Hence $Y_1$ has the same law as $T_1 + T'_1$ meaning that by~\eqref{eq:BES-hitting}, for $\lambda >0$,
    \[ \Exp \re^{-\lambda Y_1} =   \frac{ {2\lambda}}{\sinh^2 \sqrt{2\lambda}}  .\]
    Now, conditional on $d'$, $(B_{m+t} - B_m)_{t \geq 0}$ is again a Brownian motion, and so $h'-m$ is the left endpoint of the first reflected-Brownian excursion of height at least $d'$. By Brownian scaling, $Y_2$ has the same distribution as $\sigma_1$, where,
    if $R_t$ denotes reflected Brownian motion and $\tau_1 := \inf \{ t > 0 : R_t =1 \}$, $\sigma_1 := \sup \{ s < \tau_1 : R_s = 0\}$, the left endpoint of the excursion containing $\tau_1$. The Laplace transform of $\sigma_1$ 
   follows from results of Neveu--Pitman~\cite{neveu} (for the explicit statement see e.g.~equation (5.9) of~\cite{orlandi}): for $\lambda >0$,
    \begin{equation}
    \label{eq:sigma-laplace} 
    \Exp \re^{-\lambda Y_2} = \Exp \re^{-\lambda \sigma_1} = \frac{\tanh \sqrt{2 \lambda}}{\sqrt{2\lambda}}. \end{equation}
    Since the law $Y_2$ conditional on $d'$ is independent of $d'$, the strong Markov property at time $m$ shows that $Y_2$ is independent of the Brownian motion up to time $m$, and, in particular, $Y_2$ is independent of $Y_1$. Then we get
    \[ \Exp \re^{-\lambda Y} = \frac{ {2\lambda}}{\sinh^2 \sqrt{2\lambda}} \cdot 
    \frac{\tanh \sqrt{2 \lambda}}{\sqrt{2\lambda}} = \frac{\sqrt{2\lambda}}{\sinh \sqrt{2\lambda} \cosh \sqrt{2\lambda}} 
    = \frac{2 \sqrt{2\lambda}}{\sinh ( 2\sqrt{2\lambda} )}  .\]
    Comparison with~\eqref{eq:BES-hitting} shows that $Y$ has the same law as $T_2$.
    The explicit form of the distribution function 
    at~\eqref{eq:Y-distribution} is given  in~\cite{hamana} (in their equation (2.7) take $\nu = 1/2$ so $j_{\nu,k}=k\pi$ and $J_{\nu+1} (j_{\nu,k}) = (-1)^{k+1} \sqrt{2/(\pi^2k)}$).
    \end{proof}

\begin{lemma}
    \label{lem:Z-distribution}
The random variable $Z$ as defined at~\eqref{df_Y_Sinai} 
has the distribution given at~\eqref{eq:Z-distribution}.
\end{lemma}
\begin{proof}
Let $M_t := \sup_{0 \leq s \leq t} B_s$ and $\rho := \inf \{ t \geq 1 : B_t > M_1\}$, so that $Z = \sup_{0 \leq t \leq \rho} (M_t - B_t)$. Here $R_t := M_t - B_t$ is reflected
Brownian motion. For $a>0$, let $\tau_a := \inf \{ t \geq 0 : R_t = a\}$ and $\sigma_a := \sup \{ s \leq \tau_a : R_s = 0\}$. We relate $Z$ to $\sigma_a$. A.s., $B_1 < M_1$, so time $1$ lies strictly inside an excursion of $R$ away from $0$. Brownian motion immediately crosses level $M_1$, and hence the infimum in the definition of $\rho$ is the terminal time of that excursion. 
We have $Z \geq a$ if and only if an excursion of $R$ whose height reaches $a$ has begun by time $1$: either such an excursion started earlier and finished by time $1$, or the excursion includes time $1$ and finishes at time $\rho$.
In other words, up to events of probability zero, $\{ Z < a \} = \{ \sigma_a > 1\}$. Brownian scaling implies that $\sigma_a$ has the same law as $a^2 \sigma_1$. Thus we get, for $z>0$,
\begin{equation}
\label{eq:Z-sigma}
    \Pr ( Z < z ) = \Pr ( \sigma_1 > z^{-2} ).
\end{equation}
As at~\eqref{eq:sigma-laplace} above, we have from~\cite{neveu} that, for $\lambda >0$,
\[ \Exp \re^{-\lambda \sigma_1} = \frac{\tanh \sqrt{2 \lambda}}{\sqrt{2 \lambda}} = \sum_{m=0}^\infty \frac{1}{\lambda+\pi^2 (2m+1)^2/8} .\]
Since every term is non-negative, Tonelli's theorem 
shows that the following density has the required Laplace transform
\[ f(t) = \sum_{m=0}^\infty \exp \left\{ - \frac{\pi^2 (2m+1)^2 t}{8} \right\} ,\]
and hence, by uniqueness, $f$ is the density of $\sigma_1$.
It follows that, for $t > 0$, 
\[ \Pr ( \sigma_1 > t ) = \frac{8}{\pi^2} \sum_{m=0}^\infty (2m+1)^{-2} \exp \left\{ - \frac{\pi^2 (2m+1)^2 t}{8} \right\} .\]
Combined with~\eqref{eq:Z-sigma} this yields~\eqref{eq:Z-distribution}.
\end{proof}

\section{Proofs in the Kesten--Kozlov--Spitzer regime}
\label{sec:proofs_KKS}

\begin{proof}[Proof of Theorem~\ref{thm:KKS<0}]
Define a $\Pr$-a.s.\ finite random variable 
$N=N(\omega)$
in the following way: 
\begin{equation}
\label{df_N_KKS}
 N := \min\bigl\{k\geq 1: V(j)<V(k-1)\text{ for all }
j\geq k \bigr\}.
\end{equation}
Note that $N-1$ is the last global maximum of $V$; since $V(0)=0$, this means that
\[ \{ N > k \} \subseteq \Bigl\{ \sup_{j \geq k} V(j) \geq 0 \Bigr\} .\]
Write $\xi_k := \log (a_k/b_k)$. Since $\xi_1$ is bounded (by~\eqref{ass:ue}) and has negative mean (by~\eqref{eq:KKS<0}), a Chernoff bound gives $\Pr ( V(j) \geq 0 ) \leq C \re^{-cj}$ and hence, for some $C' <\infty$ and $c' >0$,
\[ \Pr ( N>k) \leq \sum_{j \geq k} \Pr ( V(j) \geq 0 ) \leq C' \re^{-c' k} ,\]
as claimed.

Denote the shifted potential by $W(r) := V (N-1+r) - V(N-1)$, $r \geq 0$.
Then by definition of $N$, $W(r) <0$ for all $r \in \N$ and, by the strong law,
$W(r) /r \to \Exp \xi_1 <0$. The zero-speed
solution to the general traffic equation for the system $[N,\infty)$ 
(see Section 1.2 of~\cite{MPW25} for definitions)
is then $\alpha_r = \re^{W(r)}$ satisfying $0 < \alpha_r <1$ and
$\sum_{r \geq 1} \alpha_r < \infty$. Thus the zero-speed solution is admissible and summable. Theorem 1.12 of~\cite{MPW25} then shows that $[N,\infty)$ is intrinsically stable started from any finite configuration.

Since $[N,\infty)$ is intrinsically stable, a coupling argument shows that $[N,\infty)$ must be either an infinite stable cloud in the full system, or a subset of one, since
deleting the particles to the left can only increase the internal gap process of $[N,\infty)$, so intrinsic tightness gives the required internal tightness in the full system. 
If $N=1$, then the proof is complete. Otherwise, suppose from this point on that $N \geq 2$: to complete the proof we must show 
(see Definition~\ref{def:stable-cloud-infinite}) that $X_N(t) - X_{N-1}(t) \to \infty$ in probability. 

First suppose that $N-1$ is the unique time at which the global maximum for $V$ is attained. Then the rightmost cloud of the isolated finite system $[1,N-1]$ has strictly negative speed, say $s<0$. Using the basic couplings
(similarly to Proposition~\ref{prop:stoch-monotonicity}) $X_{N-1} (t) \leq X'_{N-1} (t)$ where~$X'$ is the system obtained by deleting all particles to the right of $N-1$, while $X_N (t) \geq X''_N (t)$ where~$X''$ is obtained by deleting all particles to the left of $N$. It follows that
\[ \limsup_{t \to \infty} t^{-1} X_{N-1} (t) \leq s < 0 ,\]
applying the strong law for the finite system $X'$ (Theorem 2.1 of~\cite{MMPW23})
whereas
\[ \liminf_{t \to \infty} t^{-1} X_{N} (t) \geq 0 ,\]
applying the strong law for the stable semi-infinite system $X''$ (Theorem 1.9 of~\cite{MPW25}). Thus $X_{N}(t) - X_{N-1} (t) \to \infty$ a.s.~(at linear speed) and hence also in probability.

It remains to consider the case where $N-1$ is \emph{not} the 
unique   global maximum for $V$. 
Put $H := V(N-1)$ and $j_0 := \min \{ j \geq 0 : V(j) = H\}$; 
we are in the case $j_0 < N-1$. Then $V(k) \leq H$ for $k \geq j_0$ with strict inequality
for every $k \geq N$. 

We use a truncation and approximation argument. For $M \geq 1$, consider the finite particle system $[1,M]$ with infinite pressure from the right, i.e., with $b_M$ set to $0$, and let $\eta^{(M)} (t) := (\eta^{(M)}_j (t))_{j \in [1,M-1]}$ be the associated gap process. Start from the same gaps as the infinite system, so $\eta_j ^{(M)} (0) = \eta_j (0)$ for all $j \in [1,M-1]$.
Using the second-class customer coupling from the proof of Lemma~\ref{lem:couple_infinite_pressure},
we may couple the infinite system with the finite system truncated at $M$ with infinite pressure from the right. Starting from the same queues on $[1,M-1]$, the customers of the truncated system may be regarded as first-class customers in the infinite system, with all additional customers regarded as second class. It follows that $\eta_j^{(M)} (t) \leq \eta_j (t)$ for all $j \in [1,M-1]$ and all $t \geq 0$. 

For $M \geq N$ consider the block $[j_0+1,M]$ with infinite pressure at $M$. Define
\[ A_k := \sum_{\ell =j_0+1}^k a_\ell^{-1} \exp \{ V(j_0) - V (\ell-1)\} , \text{ for } k \in [j_0+1,M].\]
The infinite-pressure speed of block $[j_0+1,M]$ is $\tilde v_M = - 1/A_M$ (this is the shifted version of
the formula~\eqref{eq:tilde_v_explicit_variant}). The corresponding traffic intensities are
\begin{equation}
\label{eq:rho-trucated}
\rho_k^{(M)} = \re^{V(k)-V(j_0)} + \tilde v_M \sum_{\ell=j_0+1}^k a_\ell^{-1} \re^{V(k)-V(\ell-1)}
= \re^{V(k)-V(j_0)} \left(1 - \frac{A_k}{A_M} \right) .\end{equation}
Here $\re^{V(k)-V(j_0)} \leq 1$, while $0 < A_k < A_M$ for $k < M$. Thus it holds that $  \rho_k ^{(M)} \in (0,1)$ for all $k \in [j_0+1,M-1]$. Hence $[j_0+1,M]$ with infinite pressure at $M$ is intrinsically stable (and $\tilde v_M$ is its speed). 

At the queue $N-1$, since $V(N-1)= V(j_0)$ formula~\eqref{eq:rho-trucated} reduces to 
\[ \rho^{(M)}_{N-1} = 1 - \frac{A_{N-1}}{A_M}. \]
Moreover, by hypotheses~\eqref{ass:ue} and~\eqref{eq:KKS<0}, $A_M \to \infty$ as $M \to \infty$. Therefore \[ \tilde v_M \uparrow 0, \text{ and } \rho_{N-1}^{(M)} \uparrow 1 .\]
We next argue that, for all sufficiently large $M$, the block $[j_0+1,M]$ is the rightmost stable cloud in the system $[1,M]$ with infinite pressure from the right. If $j_0=0$ there is nothing to prove; so suppose $j_0 >0$. By definition, $V(j) < V(j_0)$ for $0 \leq j < j_0$. In particular, the rightmost cloud
in the isolated finite system $[1,j_0]$ has strictly negative speed, $s<0$, say. The block $[j_0+1,M]$ has speed $\tilde v_M \to 0$, so for all $M$ large enough $s < \tilde v_M < 0$. The finite-cloud decomposition algorithm (Algorithm 1 from~\cite{MMPW23})
implies that the cloud decomposition of $[1,M]$ is given by the decomposition of $[1,j_0]$ followed by the single rightmost cloud $[j_0+1,M]$. Moreover, the queue separating these two parts is transient. Hence from some finite time onwards, the internal gaps of $[j_0+1,M]$ evolve exactly as the gaps in the isolated stable system $[j_0+1,M]$ with infinite pressure at $M$. Consequently, from Theorem 2.1 of~\cite{MMPW23} applied to that finite system,
\[ \lim_{t \to \infty} \Pr \left[ \eta_{N-1}^{(M)} (t) \leq B \right] = 1-  (\rho_{N-1}^{(M)})^{B+1} , \text{ for } B \in \ZP .\]
It follows from the coupling described earlier that, in the infinite system,
\[ \limsup_{t \to \infty} \Pr \left[ \eta_{N-1} (t) \leq B \right] \leq 1-  (\rho_{N-1}^{(M)})^{B+1} , \text{ for } B \in \ZP .\]
Take $M \to \infty$. Since $\rho_{N-1}^{(M)} \to 1$, we can conclude, for every fixed $B$,
$\Pr (\eta_{N-1} (t) \leq B ) \to 0$, as was required.
\end{proof}

\begin{remark}
\label{rem:KKS-IC}
We expect that in most cases the restriction on the initial condition in Theorem~\ref{thm:KKS<0} can be removed, but since there are some subtleties, and
the theory of non-finite initial conditions is not so developed~\cite{MPW25}, we do not pursue this extension, but make some observations here.

Indeed, suppose additionally that $\Pr [\xi_1 >0] >0$. Since $\Exp \xi_1 <0$, we have $\Pr [\xi_1 <0]>0$ also; i.e.~\eqref{ass:nestling} holds. Choose $\delta > 0$ such that $q:= \Pr [\xi_1 \geq \delta] >0$. For every given $m \in \N$ there are infinitely many $r \in \N$ for which $\xi_{r+1}, \ldots, \xi_{r+m}$ are all at least $\delta$. At the endpoint $k=r+m$ of such a run, $V (\ell ) \leq V(k)$ for $r+1 \leq \ell \leq k$.
The formula for $\beta_k$ in~\eqref{alpha_to_V} implies that,
under condition~\eqref{ass:ue},
\[
 \beta_k \geq \eps_0 \cdot \card\{\ell\leq k: V(\ell)\leq V(k)\},
\]
and so $\beta_k \geq \eps_0 m$. Since $m$ was arbitrary we get $\limsup_{k \to \infty} \beta_k = \infty$. 
Applying the same argument to the shifted environment on $[N,\infty)$, for every fixed $m$ there are a.s.\ infinitely many runs of $m$ consecutive non-negative (indeed, uniformly positive) increments. Hence the corresponding shifted $\beta_k$ is unbounded. Thus no $v>0$ can give an admissible solution. Nor can $v<0$, since $\alpha_k\to0$ while $\beta_k$ is bounded below away from zero. Consequently the zero-speed solution is the unique admissible solution. Hence we anticipate the initial condition can be general in that case.

The other case is the \emph{one-sided} case where $\Pr [ \xi_1 > 0 ] =0$.
Then $\xi_k \leq 0$, a.s.    Now $V$ is non-increasing
and $N = \min \{ k \geq 1 : \xi_k < 0\}$. Here there can be other admissible solutions:
the homogeneous case $a_k \equiv a < b \equiv b_k$ already shows this:
see Example 3.3 of~\cite{MPW25}. Thus this case may be more delicate with regards to extending the initial condition in Theorem~\ref{thm:KKS<0}.
\end{remark}

\begin{proof}[Proof of Theorem~\ref{thm:KKS>0}]
The main idea is the following. For the upper 
bound on the speed, we argue that one can find 
a trap of depth essentially $\beta\log k$
to the left of~$k$; then, even with infinite pressure,
the cloud that contains this trap cannot go faster
than about $\re^{-\beta\log k}=k^{-\beta}$.
For the lower bound, 
we consider the stable cloud containing the $k$th
particle, and, with a suitable refinement of
the arguments of Lemma~\ref{l_depth_speed},
prove that the magnitude of its speed 
cannot be much less than~$k^{-\beta}$.

The key
fact about the potential that we need is that the deepest trap
in the interval~$[0,k]$ has depth roughly $\beta\log k$
(as shown on Figure~\ref{f_KKS_trap}).
\begin{figure}
\begin{center}
\includegraphics{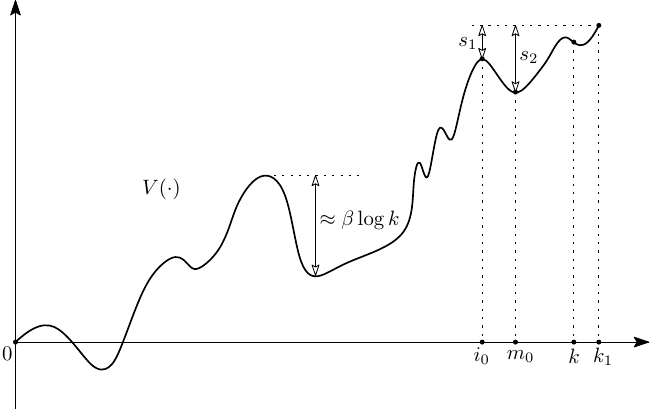}
\caption{On the proof of Theorem~\ref{thm:KKS>0}: the deepest trap
in~$[0,k]$ has depth about $\beta\log k$.}
\label{f_KKS_trap}
\end{center}
\end{figure}
As shown in Sections~2 and~3 of~\cite{FGP10},
one can obtain quite precise control
on the number of traps of given depth (note 
that the analysis of~\cite{FGP10} is in the case
$\IE\log\frac{\Theta^-}{\Theta^+}<0$, but one can obtain
the results of our interest by symmetry).
In any case, here we only need the following consequence
of the results in Section~3 (specifically, see Lemmas~3.4 and~3.5) of~\cite{FGP10}:
 for any $\delta>0$, 
\begin{equation}
\label{depth_control_KKS}
 (\beta-\delta)\log k \leq 
 \max_{[j_0,j_1]\subset [0,k]}\cD(j_0,j_1)
 \leq (\beta+\delta)\log k , 
 \text{ eventually a.s.}
\end{equation}

It follows from~\eqref{depth_control_KKS} applied at $\lfloor k/2 \rfloor$ that
$\max_{[j_0,j_1]\subset [0,\lfloor k/2 \rfloor ]}\cD(j_0,j_1)
\geq (\beta-\delta)\log k$ eventually. Choose a time $m \leq \lfloor k/2 \rfloor$ such that 
\[ \max_{0 \leq i \leq m} V(i) - V(m) = \max_{0 \leq j \leq \lfloor k/2 \rfloor} \max_{0 \leq i \leq j} ( V(i) - V(j) ) \geq (\beta-\delta)\log k.\] 
Then set $a:=\max \{ i \leq m : V(i) = \max_{0 \leq j \leq m} V(j) \}$.
Then $a < m$ and $V(j) < V(a)$ for $a < j \leq m$. The strong law
for $V$, with $\mu:= \Exp \log (\Theta^-/\Theta^+)  >0$
by~\eqref{eq:KKS>0}, implies that, $\Pr$-a.s.,
\[ \lim_{k \to \infty} \frac{1}{k} \max_{0 \leq j \leq \lfloor k/2 \rfloor} V(j) = \frac{\mu}{2}, \text{ and } \lim_{k \to \infty} \frac{V(k)}{k} = \mu .\]
Thus $V(k) > V(a)$ eventually.
Let $b:= \min\{j>m:V(j) \geq V(a)\}$. Thus we have shown existence of $a, b, m$ with $a < m< b \leq k$,
$V(j) < V(a)$ for $a < j < b$, $V(b) \geq V(a)$, and $V(a) - V(m) \geq (\beta-\delta) \log k$.
Thus by
Lemma~\ref{l_stable_almost0},
 we have that $[a+1,b]$ is intrinsically stable with
 depth $\cD (a, b) \geq (\beta-\delta) \log k$.

 Consider the finite system $[1,k']$ with $k' > k$ and infinite pressure at~$k'$. Then in the cloud decomposition of $[1,k']$, either $[a+1, b]$ and~$k$ are in the same cloud, or else~$k$ is in a cloud to the right of the cloud
 containing $[a+1,b]$. In either case, let~$w$ be the speed of the cloud containing $[a+1,b]$.
As in the proof of~\eqref{depth_speed_upper}, Proposition~\ref{prop:stoch-monotonicity}
gives a coupling under which the particles $[a+1,b]$ of $[1,k']$ are no farther 
to the left than in the system $[a+1,b]$ with infinite pressure at~$b$.
Hence $w\geq\tv$, where $\tv$ is the speed of the latter system, and 
$|\tv|\leq\eps_0^{-1}\re^{-\cD(a,b)}$ by~\eqref{depth_speed_upper}
(note that $V(b)=\max_{[a,b]}V$). 
Since the cloud speeds in $[1,k']$ are negative and non-decreasing from left to right, the magnitude of the speed 
of the stable cloud in $[1,k']$ containing~$k$ is at most~$\eps_0^{-1} k^{-(\beta-\delta)}$.

We next seek a lower bound on the magnitude of the speed. Unfortunately, Lemma~\ref{l_depth_speed} is not 
sharp enough in this case, because we are now aiming for
polynomial bounds in~$k$, and~\eqref{depth_speed_lower}
contains a potentially polynomial factor.
Therefore, to prove that $|v_k|$ is at least~$ck^{-(\beta+\delta)}$, we need to do a finer analysis of the 
formula~\eqref{V_v_explicit}.

Recall $\mu =\IE\log (\Theta^-/\Theta^+)>0$
and notice that there exist $u_0,c_0>0$ such that 
\begin{equation}
\label{can_go_up}
 \IP\Big[\log\frac{\Theta^-}{\Theta^+}>u_0\Big]>c_0.
\end{equation}
Then, we say that a site~$m$ is \emph{good}
if
\[
 V(m) \geq u_0 + \max_{0\leq j<m} V(j).
\]
There are typically plenty of good sites; this is a consequence
of the following elementary fact about ladder epochs. Denote 
by $L_1= \min\{n>0: V(n)>0\}$ the first ladder epoch.
Provided that $\mu>0$, usual large deviation estimates
imply that we have for some $c=c(\mu, \eps_0)$
\begin{equation}
\label{ladder_tail}
 \IP[L_1>t] \leq \re^{-ct}.
\end{equation}
Then, notice that at each ladder epoch, the next site 
will be good with probability at least~$c_0$. From this,
it is elementary to obtain that 
\begin{equation}
\label{many_good_sites}
\text{$\Pr$-a.s.\ for all~$k$ large enough, the interval }
 [k,k+\log^2 k]\text{ contains a good site.}
\end{equation}

Next, for an interval~$I$, define
\[
 \Psi(I) := \sum_{j\in I}\re^{V(m^I)-V(j)},
  \text{ where }m^I:=\argmin_{j\in I} V(j).
\]
Let us prove that
\begin{equation}
\label{not_so_large_Psi}
\text{$\Pr$-a.s.\ for all~$k$ large enough it holds that }
 \Psi(I) \leq 2\log^2 k \text{ for all }I\subset [0,2k].
\end{equation}
Indeed, let $I=[n_1,n_2]$, and assume that
$n_2-n_1\geq 2\log^2 k$ (otherwise~\eqref{not_so_large_Psi}
trivially holds). First, usual large deviation bounds
imply that
\[
 \IP\big[V(n_1+\log^2 k)-V(n_1)\geq \tfrac{2\mu}{3}
  \log^2 k\big] \geq 1 - \re^{-c_1\log^2 k}.
\]
Then, we also have
\[
 \IP\big[V(i)-V(n_1+\log^2 k)\geq -\tfrac{\mu}{3}
  \log^2 k \text{ for all }i\in[n_1+\log^2 k,n_2]\big] 
  \geq 1 - \re^{-c_2\log^2 k},
\]
and so (observe that $V(m^I)\leq V(n_1)$)
\begin{equation}
\label{LD_V_log2}
  \IP\big[V(i)-V(m^I)\geq \tfrac{\mu}{3}
  \log^2 k \text{ for all }i\in[n_1+\log^2 k,n_2]\big] 
  \geq 1 - 2\re^{-c_3\log^2 k}.
\end{equation}
Next, on the event in~\eqref{LD_V_log2}, 
for $I\subset [0,2k]$ we have that
\[
\Psi(I)\leq \log^2 k + |I|\exp(-\tfrac{\mu}{3}\log^2 k)
 \leq 2 \log^2 k,
\]
and therefore~\eqref{not_so_large_Psi} follows from~\eqref{LD_V_log2} with Borel--Cantelli and the fact that there are $O(k^2)$ choices for interval~$I \subset [0,2k]$.

Now, we are ready to obtain a lower bound on~$|v_k|$.
Let~$k_1\geq k$ be the smallest good site one can find
in $[k,\infty)$; observe that, by~\eqref{many_good_sites},
we have $k_1\leq k+\log^2 k$.
Then, consider the finite system with particles
$\{1,\ldots,k_1\}$ and let $[i_0+1,k_1]$ be
the stable cloud 
containing~$k_1$ (see Figure~\ref{f_KKS_trap}),
let~$\hv$ be its speed.
Let $m_0 = \argmin_{j\in [i_0,k_1]}V(j)$, and abbreviate
$s_1=V(k_1)-V(i_0)$, $s_2=V(k_1)-V(m_0)$. 
 From~\eqref{V_v_explicit} and using that $s_1\geq u_0$
(because $k_1$ is good) together with Condition~\eqref{ass:ue},
we obtain that
\begin{align*}
 |\hv| &=\frac{\re^{V(k_1)-V(i_0)}-1}
{\frac{1}{a_{k_1}}\re^{V(k_1)-V(k_1-1)}
  +\cdots + \frac{1}{a_{i_0+1}}\re^{V(k_1)-V(i_0)}}\\
 &\geq  \frac{C \re^{s_1}}{\eps_0^{-1}\re^{s_2}\Psi([i_0,k_1])}\\
 &\geq C' \re^{-(s_2-s_1)} \log^{-2}k.
\end{align*}
Since $s_2-s_1\leq \cD(0,2k)$, 
and finite-system cloud speeds are non-decreasing from left to right, 
this implies the desired
lower bound, i.e., that in the system $[1,k_1]$ the cloud containing $k$ has speed of magnitude at least $c k^{-(\beta+\delta)}$.

Next we argue that in the infinite system there must be infinitely many clouds. Let $k_1$ be the good site used above, and put $M=k^2$. Denote particle speeds in the system $[1,M]$ with infinite pressure by $w_j^{(M)}$. Removing particles to the right of $k_1$, and releasing the pressure, can only increase the positions of the retained particles. Hence the lower-magnitude estimate in the ordinary system $[1,k_1]$ gives
$ w_k^{(M)}\le -c k^{-(\beta+\eps)}$.
On the other hand,  the trap estimate at scale $M$ gives
\[ w_M^{(M)}\ge -C M^{-(\beta-\varepsilon)} =-Ck^{-2(\beta-\varepsilon)}.\]
Choose $0<\varepsilon<\beta/3$. Eventually, $ w_k^{(M)}<w_M^{(M)}$.
Therefore there is a strict increase in cloud speed across some queue with index in $[k,M-1]$. That queue is transient. Since $k$ is arbitrarily large, these transient queues occur arbitrarily far to the right; extracting increasing sequences gives exactly the hypothesis of Corollary~\ref{cor:inf_finite_queues}. 

To finish the proof, we apply Lemma~\ref{lem:couple_infinite_pressure} to show that $t^{-1} X_k(t)$ has a limit and the finite-system speed estimates derived above apply to the resulting speed $v_k$.
Indeed,
for fixed $k$, choose one of the finite infinite-pressure truncations having a transient separating queue strictly to the right of $k$.  Lemma~\ref{lem:couple_infinite_pressure}  then identifies the dynamics on the clouds to its left, after finite time, with those of that finite system (up to an irrelevant spatial translation), so the finite-system strong law gives the strong law for particle $k$. The infinite-pressure truncation gives the lower bound on $v_k$, while removing particles to the right and using the ordinary $[1,k_1]$ system gives the upper bound on $v_k$. Hence, for every $\varepsilon>0$, $-k^{-(\beta-\eps)} \leq v_k \leq -k^{-(\beta+\eps)}$, as claimed.
\end{proof}

\appendix
\section{Ladder height condition}
\label{sec:lh}

An example where Condition~\eqref{ass:lh} trivially
holds is when the distribution 
of~$\log (\Theta^-/\Theta^+)$ is lattice, 
because in that case~$\eta$ will always be 
at least the mesh of the lattice. Another example is provided 
under a bounded density hypothesis, as in the following result.
Condition~\eqref{ass:lh} can fail, even under~\eqref{ass:ue} and~\eqref{ass:nestling}, for example by taking $\Theta^- = \re^{X/2}$ and $\Theta^+ = \re^{-X/2}$ where $X$ has a distribution with a sequence of atoms going rapidly to $0$ from above. 
 
\begin{lemma}
    \label{lem:lh}
    Suppose that~\eqref{ass:nestling} and~\eqref{ass:ue} hold. Then sufficient for~\eqref{ass:lh} to hold is either condition (a)~that the distribution of~$\log\frac{\Theta^-}{\Theta^+}$
    under $\IP$ has a bounded density, or (b)~that $\Theta^+$ and $\Theta^-$ are independent and each has a bounded density under $\IP$.
\end{lemma}
\begin{proof}
Let $\zeta := \log \frac{\Theta^-}{\Theta^+}$. First we prove the lemma under condition (a). Suppose that~\eqref{ass:nestling} and~\eqref{ass:ue} hold, and that $\zeta$ has a bounded density. 
Let $f_n$ denote the density of  $V_n$.
Then $0 < \Var \zeta < \infty$ and (since $f_1$ is bounded over a bounded support) $f_1 \in L^2 (\R)$. Let $\phi (t) := \Exp \re^{it \zeta}$;
then $\phi \in L^2(\R)$ by Plancherel's theorem.
If $\zeta'$ is an independent copy of $\zeta$, then $| \phi (t) |^2 = \Exp \re^{it(\zeta-\zeta')} = 1 - t^2 \Var \zeta + o(t^2)$ as $t \to 0$. Hence for some $c, \delta >0$, $| \phi (t) | \leq  \re^{-ct^2}$ for all $|t| <\delta$. Moreover, $| \phi (t) |<1$ for every $t \neq 0$ (see~\cite[p.~18]{lukacs}) and the Riemann--Lebesgue lemma \cite[p.~19]{lukacs} says that $\phi(t) \to 0$ as $|t| \to \infty$. Hence $\rho := \sup_{|t| \geq \delta} | \phi(t) | <1$. For $n \geq 2$,
\[ \int_\R | \phi(t) |^n \ud t \leq \int_{|t| \leq \delta} \re^{-c n t^2} \ud t + \rho^{n-2} \int_{|t|> \delta} | \phi (t) |^2 \ud t
= O (n^{-1/2} ) + O( \rho^{n-2} ) = O (n^{-1/2} ).\]
 Thus  by Fourier inversion (see Theorem~3.2.2 of~\cite{lukacs})
\[ 
\| f_n \|_\infty \leq \frac{1}{2\pi} \int_\R | \phi (t) |^n \ud t = O (n^{-1/2} ) ,
\]
as $n \to \infty$. 
Formula~(3) (with~$k=1$) of~\cite{AD99}
implies that
\[
 \IP[\eta<y] 
  = \sum_{n=1}^\infty \frac{1}{n}
   \IP\big[0<V_n\leq \eta, \, V_n<y\big]
  \leq  \sum_{n=1}^\infty \frac{1}{n} 
     \IP\big[0<V_n<y\big] ,
\]
where, for some $C < \infty$ and all $n \in \N$ (recall $\| f_1 \|_\infty < \infty$ by hypothesis),
\[ \IP\big[0<V_n<y\big] = \int_0^y f_n (z) \ud z \leq y \| f_n \|_\infty \leq C y n^{-1/2} .\]
Thus we get $\IP[\eta<y] = O (y)$ as $y \downarrow 0$, which verifies~\eqref{ass:lh}.

Finally we show that condition (b) implies (a).
Suppose that $\Theta^-$ and $\Theta^+$ are independent and have densities $f^-$ and $f^+$ respectively, and that, for some $B < \infty$, $f^+ (x) \leq B$ and $f^-(x) \leq B$ for all $\eps_0 \leq x \leq \eps_0^{-1}$ (recall that~\eqref{ass:ue} is assumed). Then a Jacobian calculation shows that~$\zeta$ has density~$f$ given by
\[ f (y) = \re^{-y} \int_{\eps_0}^{\eps_0^{-1}} u f^- (u) f^+ ( u \re^{-y} ) \ud u \leq \re^{-y} B^2 \eps_0^{-2} ,\]
showing that $f$ is bounded.
\end{proof}

\section*{Acknowledgements}

The authors thank Vladislav Vysotskiy for 
discussions about ascending ladder heights and 
for indicating the paper~\cite{AD99} to them.
ChatGPT was used to help identify the explicit distributions of the limit random variables~$Y$ 
and~$Z$, to help with the proof of Lemma~\ref{lem:lh}, and to identify some errors in earlier drafts; 
all other mathematical arguments and proofs are entirely human-generated. 
MM and AW were supported by EPSRC grant EP/W00657X/1.
SP was partially supported by
CMUP, member of LASI, which is financed by national funds
through FCT --- Funda\c{c}\~ao
para a Ci\^encia e a Tecnologia, I.P., 
under the project with reference UID/00144/2025,
\url{https://doi.org/10.54499/UID/00144/2025}.
Some of this work was carried
 out during the programme ``Stochastic systems for anomalous diffusion'' (July--December 2024) hosted by the  Isaac Newton Institute, under EPSRC grant EP/Z000580/1.

\printbibliography

@article{jsi,
  title = {Partially asymmetric exclusion processes with sitewise disorder},
  author = {Juh\'asz, R\'obert and Santen, Ludger and Igl\'oi, Ferenc},
  journal = {Phys. Rev. E},
  volume = {74},
  pages = {061101},
  numpages = {10},
  year = {2006},
  publisher = {American Physical Society},
  doi = {10.1103/PhysRevE.74.061101},
  url = {https://link.aps.org/doi/10.1103/PhysRevE.74.061101}
}

@book {liggett,
    AUTHOR = {Liggett, Thomas M.},
     TITLE = {Interacting particle systems},
 PUBLISHER = {Springer-Verlag}, 
   ADDRESS = {Berlin},
      YEAR = {2005},
     PAGES = {xvi+496},
      ISBN = {3-540-22617-6},
   MRCLASS = {60-02 (60K35 82C22)},
  MRNUMBER = {2108619},
MRREVIEWER = {Michael\ Pr\"ahofer},
       DOI = {10.1007/b138374},
       URL = {https://doi.org/10.1007/b138374},
}

@article {golosov,
    AUTHOR = {Golosov, A. O.},
     TITLE = {Limit distributions for random walks in random environments},
   JOURNAL = {Sovet Math. Dokl.},
    VOLUME = {28},
      YEAR = {1983},
    NUMBER = {1},
     PAGES = {18--22},
}

@article {kesten,
    AUTHOR = {Kesten, Harry},
     TITLE = {The limit distribution of {S}ina\u i's random walk in random
              environment},
   JOURNAL = {Phys. A},
  FJOURNAL = {Physica A},
    VOLUME = {138},
      YEAR = {1986},
    NUMBER = {1-2},
     PAGES = {299--309},
      ISSN = {0378-4371},
   MRCLASS = {60J15 (82A42)},
  MRNUMBER = {865247},
MRREVIEWER = {Gregory\ F.\ Lawler},
       DOI = {10.1016/0378-4371(86)90186-X},
       URL = {https://doi.org/10.1016/0378-4371(86)90186-X},
}

@article {hamana,
    AUTHOR = {Hamana, Yuji and Matsumoto, Hiroyuki},
     TITLE = {The probability distributions of the first hitting times of
              {B}essel processes},
   JOURNAL = {Trans. Amer. Math. Soc.},
  FJOURNAL = {Transactions of the American Mathematical Society},
    VOLUME = {365},
      YEAR = {2013},
    NUMBER = {10},
     PAGES = {5237--5257},
      ISSN = {0002-9947,1088-6850},
   MRCLASS = {60J60 (33C10 44A10)},
  MRNUMBER = {3074372},
MRREVIEWER = {Samuel\ Herrmann},
       DOI = {10.1090/S0002-9947-2013-05799-6},
       URL = {https://doi.org/10.1090/S0002-9947-2013-05799-6},
}

@book {borodin,
    AUTHOR = {Borodin, Andrei N. and Salminen, Paavo},
     TITLE = {Handbook of {B}rownian motion---facts and formulae},
    SERIES = {Probability and its Applications},
   EDITION = {Second},
 PUBLISHER = {Birkh\"auser Verlag, Basel},
      YEAR = {2002},
     PAGES = {xvi+672},
      ISBN = {3-7643-6705-9},
   MRCLASS = {60-00 (60H05 60J25 60J55 60J60 60J65)},
  MRNUMBER = {1912205},
MRREVIEWER = {S\'andor\ Cs\"org\H o},
       DOI = {10.1007/978-3-0348-8163-0},
       URL = {https://doi.org/10.1007/978-3-0348-8163-0},
}

@incollection {neveu,
    AUTHOR = {Neveu, J. and Pitman, J.},
     TITLE = {Renewal property of the extrema and tree property of the
              excursion of a one-dimensional {B}rownian motion},
 BOOKTITLE = {S\'eminaire de {P}robabilit\'es, {XXIII}},
    SERIES = {Lecture Notes in Math.},
    VOLUME = {1372},
     PAGES = {239--247},
 PUBLISHER = {Springer}, 
   ADDRESS = {Berlin},
      YEAR = {1989},
      ISBN = {3-540-51191-1},
   MRCLASS = {60J65},
  MRNUMBER = {1022914},
MRREVIEWER = {L.\ C. G. Rogers},
       DOI = {10.1007/BFb0083976},
       URL = {https://doi.org/10.1007/BFb0083976},
}

@article {orlandi,
    AUTHOR = {Orlandi, Enza and Picco, Pierre},
     TITLE = {One-dimensional random field {K}ac's model: weak large
              deviations principle},
   JOURNAL = {Electron. J. Probab.},
  FJOURNAL = {Electronic Journal of Probability},
    VOLUME = {14},
      YEAR = {2009},
     PAGES = {no. 47, 1372--1416},
      ISSN = {1083-6489},
   MRCLASS = {60K35 (60F10 60K37 82B43)},
  MRNUMBER = {2511287},
MRREVIEWER = {Jean-Baptiste\ Bardet},
       DOI = {10.1214/ejp.v14-662},
       URL = {https://doi.org/10.1214/ejp.v14-662},
}

@book {lukacs,
    AUTHOR = {Lukacs, Eugene},
     TITLE = {Characteristic functions},
   EDITION = {2nd ed.},
 PUBLISHER = {Hafner Publishing Co.}, 
   ADDRESS = {New York},
      YEAR = {1970},
     PAGES = {x+350},
   MRCLASS = {60E05 (62EXX)},
  MRNUMBER = {346874},
}

@article {evans,
    AUTHOR = {Evans, M. R.},
     TITLE = {Exact steady states of disordered hopping particle models with
              parallel and ordered sequential dynamics},
   JOURNAL = {J. Phys. A},
  FJOURNAL = {Journal of Physics. A. Mathematical and General},
    VOLUME = {30},
      YEAR = {1997},
    NUMBER = {16},
     PAGES = {5669--5685},
      ISSN = {0305-4470,1751-8121},
   MRCLASS = {82C44},
  MRNUMBER = {1478380},
MRREVIEWER = {Dimitri\ Petritis},
       DOI = {10.1088/0305-4470/30/16/011},
       URL = {https://doi.org/10.1088/0305-4470/30/16/011},
}

@article {evans96,
    AUTHOR = {Evans, M. R.},
     TITLE = {Bose-Einstein condensation in disordered exclusion models and relation to traffic flow},
   JOURNAL = {Europhys. Lett.},
    VOLUME = {36},
      YEAR = {1996},
    NUMBER = {1},
     PAGES = {13--18},
}

@book {revesz,
    AUTHOR = {R\'ev\'esz, P\'al},
     TITLE = {Random walk in random and non-random environments},
   EDITION = {3rd edition},
 PUBLISHER = {World Scientific Publishing Co.}, 
   address = {Hackensack, NJ},
      YEAR = {2013},
     PAGES = {xviii+402},
      ISBN = {978-981-4447-50-8},
   MRCLASS = {60G50 (60F15 60F17 60J55 60K37)},
  MRNUMBER = {3060348},
MRREVIEWER = {Ofer\ Zeitouni},
       DOI = {10.1142/8678},
       URL = {https://doi.org/10.1142/8678},
}

@article {frs,
    AUTHOR = {Floreani, Simone and Redig, Frank and Sau, Federico},
     TITLE = {Hydrodynamics for the partial exclusion process in random
              environment},
   JOURNAL = {Stochastic Process. Appl.},
  FJOURNAL = {Stochastic Processes and their Applications},
    VOLUME = {142},
      YEAR = {2021},
     PAGES = {124--158},
      ISSN = {0304-4149,1879-209X},
   MRCLASS = {60K37 (60F17 60K35)},
  MRNUMBER = {4314096},
       DOI = {10.1016/j.spa.2021.08.006},
       URL = {https://doi.org/10.1016/j.spa.2021.08.006},
}

@article {faggionato,
    AUTHOR = {Faggionato, Alessandra},
     TITLE = {Random walks and exclusion processes among random conductances
              on random infinite clusters: homogenization and hydrodynamic
              limit},
   JOURNAL = {Electron. J. Probab.},
  FJOURNAL = {Electronic Journal of Probability},
    VOLUME = {13},
      YEAR = {2008},
     PAGES = {no. 73, 2217--2247},
      ISSN = {1083-6489},
   MRCLASS = {60K35 (60J27 82C44)},
  MRNUMBER = {2469609},
MRREVIEWER = {Da-Quan\ Jiang},
       DOI = {10.1214/EJP.v13-591},
       URL = {https://doi.org/10.1214/EJP.v13-591},
}

@article {zerner,
    AUTHOR = {Zerner, Martin P. W.},
     TITLE = {Lyapounov exponents and quenched large deviations for
              multidimensional random walk in random environment},
   JOURNAL = {Ann. Probab.},
  FJOURNAL = {The Annals of Probability},
    VOLUME = {26},
      YEAR = {1998},
    NUMBER = {4},
     PAGES = {1446--1476},
      ISSN = {0091-1798,2168-894X},
   MRCLASS = {60F10 (60G50 60K37 82C41)},
  MRNUMBER = {1675027},
MRREVIEWER = {Wolfgang\ K\"onig},
       DOI = {10.1214/aop/1022855870},
       URL = {https://doi.org/10.1214/aop/1022855870},
}

@Inbook{Zeitouni2004,
author="Zeitouni, Ofer",
editor="Picard, Jean",
title="Part II: Random Walks in Random Environment",
bookTitle="Lectures on Probability Theory and Statistics: Ecole d'Et{\'e} de Probabilit{\'e}s de Saint-Flour XXXI - 2001",
year="2004",
publisher="Springer",
address="Berlin, Heidelberg",
pages="190--312",
isbn="978-3-540-39874-5",
doi="10.1007/978-3-540-39874-5_2",
url="https://doi.org/10.1007/978-3-540-39874-5_2"
}

@article {solomon,
    AUTHOR = {Solomon, Fred},
     TITLE = {Random walks in a random environment},
   JOURNAL = {Ann. Probability},
  FJOURNAL = {The Annals of Probability},
    VOLUME = {3},
      YEAR = {1975},
     PAGES = {1--31},
      ISSN = {0091-1798},
   MRCLASS = {60J15},
  MRNUMBER = {362503},
MRREVIEWER = {V.\ K.\ Rohatgi},
       DOI = {10.1214/aop/1176996444},
       URL = {https://doi.org/10.1214/aop/1176996444},
}

@article {AD99,
    AUTHOR = {Alili, L. and Doney, R. A.},
     TITLE = {Wiener-{H}opf factorization revisited and some applications},
   JOURNAL = {Stochastics Stochastics Rep.},
  FJOURNAL = {Stochastics and Stochastics Reports},
    VOLUME = {66},
      YEAR = {1999},
    NUMBER = {1-2},
     PAGES = {87--102},
      ISSN = {1045-1129},
   MRCLASS = {60K05 (60G50)},
  MRNUMBER = {1687803},
       DOI = {10.1080/17442509908834187},
       URL = {https://doi.org/10.1080/17442509908834187},
}

@article {FGP10,
    AUTHOR = {Fribergh, Alexander and Gantert, Nina and Popov, Serguei},
     TITLE = {On slowdown and speedup of transient random walks in random
              environment},
   JOURNAL = {Probab. Theory Related Fields},
  FJOURNAL = {Probability Theory and Related Fields},
    VOLUME = {147},
      YEAR = {2010},
    NUMBER = {1-2},
     PAGES = {43--88},
      ISSN = {0178-8051,1432-2064},
   MRCLASS = {60K37 (60F10 60G50)},
  MRNUMBER = {2594347},
MRREVIEWER = {Marcel\ Ortgiese},
       DOI = {10.1007/s00440-009-0201-2},
       URL = {https://doi.org/10.1007/s00440-009-0201-2},
}

@article{MPW25_elastic,
    author = {Menshikov, Mikhail and Popov, Serguei and Wade, Andrew},
    title = {One-dimensional particle clouds with elastic collisions},
    journal = {Preprint, ArXiv: \href{https://arxiv.org/abs/2509.08430}{2509.08430}},
    year = {2025},
}

@article {MPW25,
    AUTHOR = {Menshikov, M. and Popov, S. and Wade, A.},
     TITLE = {Semi-infinite particle systems with exclusion interaction and heterogeneous jump rates},
   JOURNAL = {Probab. Theory Relat. Fields},
   VOLUME = {195},
   PAGES = {73-107},
      YEAR = {2026},
}

@article{MPW26_leftmost,
    author = {Menshikov, Mikhail and Popov, Serguei and Wade, Andrew},
    title = {Dynamics of the leftmost particle in heterogeneous semi-infinite exclusion systems},
    journal = {Markov Process.\ Relat.\ Fields},
    volume = {32},
    year = {2026},
    pages = {205--232},
}

@article {AFGL00,
    AUTHOR = {Andjel, E. D. and Ferrari, P. A. and Guiol, H. and Landim, C.},
     TITLE = {Convergence to the maximal invariant measure for a zero-range
              process with random rates},
   JOURNAL = {Stochastic Process. Appl.},
  FJOURNAL = {Stochastic Processes and their Applications},
    VOLUME = {90},
      YEAR = {2000},
    NUMBER = {1},
     PAGES = {67--81},
      ISSN = {0304-4149,1879-209X},
   MRCLASS = {60K35},
  MRNUMBER = {1787125},
MRREVIEWER = {Timo\ Sepp\"al\"ainen},
       DOI = {10.1016/S0304-4149(00)00037-5},
       URL = {https://doi.org/10.1016/S0304-4149(00)00037-5},
}

@article {sinai82,
    AUTHOR = {Sina\u\i, Ya.\ G.},
     TITLE = {The limiting behavior of a one-dimensional random walk in a random medium},
   JOURNAL = {Theory Probab. Appl.}, 
    VOLUME = {27},
      YEAR = {1982},
    NUMBER = {2},
     PAGES = {256--268},
      ISSN = {0040-361X},
   MRCLASS = {60J15 (58F11 60K35 82A42 82A70)},
  MRNUMBER = {657919},
MRREVIEWER = {F.\ L.\ Spitzer},
}

@article {BMRS15,
    AUTHOR = {Bahadoran, Christophe and Mountford, Thomas and Ravishankar, Krishnamurthi and Saada, Ellen},
     TITLE = {Supercriticality conditions for asymmetric zero-range process
              with sitewise disorder},
   JOURNAL = {Braz. J. Probab. Stat.},
  FJOURNAL = {Brazilian Journal of Probability and Statistics},
    VOLUME = {29},
      YEAR = {2015},
    NUMBER = {2},
     PAGES = {313--335},
      ISSN = {0103-0752,2317-6199},
   MRCLASS = {60K37 (82B44)},
  MRNUMBER = {3336869},
MRREVIEWER = {Rongfeng\ Sun},
       DOI = {10.1214/14-BJPS273},
       URL = {https://doi.org/10.1214/14-BJPS273},
}

@article {zhu94,
    AUTHOR = {Zhu, Yixin},
     TITLE = {Markovian queueing networks in a random environment},
   JOURNAL = {Oper. Res. Lett.},
  FJOURNAL = {Operations Research Letters},
    VOLUME = {15},
      YEAR = {1994},
    NUMBER = {1},
     PAGES = {11--17},
      ISSN = {0167-6377,1872-7468},
   MRCLASS = {90B22},
  MRNUMBER = {1284183},
       DOI = {10.1016/0167-6377(94)90009-4},
       URL = {https://doi.org/10.1016/0167-6377(94)90009-4},
}

@article {LY24,
    AUTHOR = {Lacoin, Hubert and Yang, Shangjie},
     TITLE = {Mixing time for the asymmetric simple exclusion process in a
              random environment},
   JOURNAL = {Ann. Appl. Probab.},
  FJOURNAL = {The Annals of Applied Probability},
    VOLUME = {34},
      YEAR = {2024},
    NUMBER = {1A},
     PAGES = {388--427},
      ISSN = {1050-5164,2168-8737},
   MRCLASS = {60K37 (60J27)},
  MRNUMBER = {4696281},
       DOI = {10.1214/23-aap1967},
       URL = {https://doi.org/10.1214/23-aap1967},
}

@article {schmid19,
    AUTHOR = {Schmid, Dominik},
     TITLE = {Mixing times for the simple exclusion process in ballistic
              random environment},
   JOURNAL = {Electron. J. Probab.},
  FJOURNAL = {Electronic Journal of Probability},
    VOLUME = {24},
      YEAR = {2019},
     PAGES = {Paper No. 22, 25pp},
      ISSN = {1083-6489},
   MRCLASS = {60K37 (60J27)},
  MRNUMBER = {3933201},
       DOI = {10.1214/19-EJP286},
       URL = {https://doi.org/10.1214/19-EJP286},
}

@article {KDO16,
    AUTHOR = {Krenzler, Ruslan and Daduna, Hans and Otten, Sonja},
     TITLE = {Jackson networks in nonautonomous random environments},
   JOURNAL = {Adv. in Appl. Probab.},
  FJOURNAL = {Advances in Applied Probability},
    VOLUME = {48},
      YEAR = {2016},
    NUMBER = {2},
     PAGES = {315--331},
      ISSN = {0001-8678,1475-6064},
   MRCLASS = {60K37 (60K25 90B22 90B25)},
  MRNUMBER = {3511763},
       URL = {http://projecteuclid.org/euclid.aap/1465490750},
}

@article {GL98,
    AUTHOR = {Gielis, G. and Koukkous, A. and Landim, C.},
     TITLE = {Equilibrium fluctuations for zero range processes in random
              environment},
   JOURNAL = {Stochastic Process. Appl.},
  FJOURNAL = {Stochastic Processes and their Applications},
    VOLUME = {77},
      YEAR = {1998},
    NUMBER = {2},
     PAGES = {187--205},
      ISSN = {0304-4149,1879-209X},
   MRCLASS = {60K35 (82C22)},
  MRNUMBER = {1649004},
MRREVIEWER = {Timo\ Sepp\"al\"ainen},
       DOI = {10.1016/S0304-4149(98)00044-1},
       URL = {https://doi.org/10.1016/S0304-4149(98)00044-1},
}

@article {kipnis,
    AUTHOR = {Kipnis, Claude},
     TITLE = {Central limit theorems for infinite series of queues and
              applications to simple exclusion},
   JOURNAL = {Ann. Probab.},
  FJOURNAL = {The Annals of Probability},
    VOLUME = {14},
      YEAR = {1986},
    NUMBER = {2},
     PAGES = {397--408},
      ISSN = {0091-1798,2168-894X},
   MRCLASS = {60K35},
  MRNUMBER = {832016},
MRREVIEWER = {M.\ Roussignol},
       URL =
              {http://links.jstor.org/sici?sici=0091-1798(198604)14:2<397:CLTFIS>2.0.CO;2-G&origin=MSN},
}

@article {GPSY16,
    AUTHOR = {Gannon, M. and Pechersky, E. and Suhov, Y. and Yambartsev, A.},
     TITLE = {Random walks in a queueing network environment},
   JOURNAL = {J. Appl. Probab.},
  FJOURNAL = {Journal of Applied Probability},
    VOLUME = {53},
      YEAR = {2016},
    NUMBER = {2},
     PAGES = {448--462},
      ISSN = {0021-9002,1475-6072},
   MRCLASS = {60J27 (60J28 60K20 60K37)},
  MRNUMBER = {3514290},
MRREVIEWER = {Dimitri\ Petritis},
       DOI = {10.1017/jpr.2016.12},
       URL = {https://doi.org/10.1017/jpr.2016.12},
}

@article{KF96,
    author = {Krug, J. and Ferrari, P. A.},
    title = { Phase transitions in driven diffusive systems with random
 rates},
    journal = {J. Phys. A: Math. Gen.},
VOLUME ={29},
year = {1996},
PAGES = {L465–L471},
}

@incollection {FS07,
    AUTHOR = {Ferrari, Pablo A. and Sisko, Valentin V.},
     TITLE = {Escape of mass in zero-range processes with random rates},
 BOOKTITLE = {Asymptotics: particles, processes and inverse problems},
    SERIES = {IMS Lecture Notes Monogr. Ser.},
    VOLUME = {55},
     PAGES = {108--120},
 PUBLISHER = {Inst. Math. Statist., Beachwood, OH},
      YEAR = {2007},
   MRCLASS = {60K35 (82C22)},
  MRNUMBER = {2459934},
       DOI = {10.1214/074921707000000300},
       URL = {https://doi.org/10.1214/074921707000000300},
}

@article {BFL96,
    AUTHOR = {Benjamini, I. and Ferrari, P. A. and Landim, C.},
     TITLE = {Asymmetric conservative processes with random rates},
   JOURNAL = {Stochastic Process. Appl.},
  FJOURNAL = {Stochastic Processes and their Applications},
    VOLUME = {61},
      YEAR = {1996},
    NUMBER = {2},
     PAGES = {181--204},
      ISSN = {0304-4149,1879-209X},
   MRCLASS = {60K35},
  MRNUMBER = {1386172},
MRREVIEWER = {Shui\ Feng},
       DOI = {10.1016/0304-4149(95)00077-1},
       URL = {https://doi.org/10.1016/0304-4149(95)00077-1},
}

@article {BMRS17,
    AUTHOR = {Bahadoran, C. and Mountford, T. and Ravishankar, K. and Saada, E.},
     TITLE = {Supercritical behavior of asymmetric zero-range process with sitewise disorder},
   JOURNAL = {Ann. Inst. Henri Poincar\'e{} Probab. Stat.},
  FJOURNAL = {Annales de l'Institut Henri Poincar\'e{} Probabilit\'es et
              Statistiques},
    VOLUME = {53},
      YEAR = {2017},
    NUMBER = {2},
     PAGES = {766--801},
      ISSN = {0246-0203,1778-7017},
   MRCLASS = {60K35 (60K37 82C22)},
  MRNUMBER = {3634274},
MRREVIEWER = {Francis\ Comets},
       DOI = {10.1214/15-AIHP736},
       URL = {https://doi.org/10.1214/15-AIHP736},
}

@incollection {BMRS21,
    AUTHOR = {Bahadoran, C. and Mountford, T. S. and Ravishankar, K. and  Saada, E.},
     TITLE = {Zero-range process in random environment},
 BOOKTITLE = {From particle systems to partial differential equations},
    SERIES = {Springer Proc. Math. Stat.},
    VOLUME = {352},
     PAGES = {51--77},
 PUBLISHER = {Springer, Cham},
      YEAR = {2021},
   MRCLASS = {60-02 (35Q35 60K37 82C22 82C44)},
  MRNUMBER = {4376180},
       DOI = {10.1007/978-3-030-69784-6\_4},
       URL = {https://doi.org/10.1007/978-3-030-69784-6_4},
}

@article {MMPW23,
    AUTHOR = {Malyshev, Vadim and Menshikov, Mikhail and Popov, Serguei and
              Wade, Andrew},
     TITLE = {Dynamics of finite inhomogeneous particle systems with
              exclusion interaction},
   JOURNAL = {J. Stat. Phys.},
  FJOURNAL = {Journal of Statistical Physics},
    VOLUME = {190},
      YEAR = {2023},
    NUMBER = {11},
     PAGES = {Paper No. 184, 32pp},
      ISSN = {0022-4715,1572-9613},
   MRCLASS = {60K35 (60J27 60K25 90B22)},
  MRNUMBER = {4666403},
MRREVIEWER = {Florent\ Benaych-Georges},
       DOI = {10.1007/s10955-023-03190-8},
       URL = {https://doi.org/10.1007/s10955-023-03190-8},
}

@article {KKS,
    AUTHOR = {Kesten, H. and Kozlov, M. V. and Spitzer, F.},
     TITLE = {A limit law for random walk in a random environment},
   JOURNAL = {Compositio Math.},
  FJOURNAL = {Compositio Mathematica},
    VOLUME = {30},
      YEAR = {1975},
     PAGES = {145--168},
      ISSN = {0010-437X,1570-5846},
   MRCLASS = {60J15 (60F05 60J80)},
  MRNUMBER = {380998},
MRREVIEWER = {K.\ B.\ Erickson},
}

\end{document}